\documentclass{article}
\usepackage{authblk}
\usepackage{fancyhdr}
\usepackage{amscd}
\usepackage{graphicx}
\usepackage{amsmath}
\usepackage{amssymb}
\usepackage{amsthm}
\usepackage{mathrsfs}
\newtheorem{dfn}{Definition}[section]
\newtheorem{thm}[dfn]{Theorem}
\newtheorem{prop}[dfn]{Proposition}
\newtheorem{lem}[dfn]{Lemma}
\newtheorem{cor}[dfn]{Corollary}
\newtheorem{rem}[dfn]{Remark}

\newtheorem{ass}[dfn]{Assumption}

\numberwithin{equation}{section}

\newcommand{\KMa}[1]{{\color{black} #1}}
\newcommand{\KMb}[1]{{\color{black} #1}}
\newcommand{\KMc}[1]{{\color{black} #1}}

\newcommand{\KMd}[1]{{\color{black} #1}}
\newcommand{\KMe}[1]{{\color{black} #1}}
\newcommand{\KMf}[1]{{\color{black} #1}}

\newcommand{\KMg}[1]{{\color{black} #1}}
\newcommand{\KMh}[1]{{\color{black} #1}}
\newcommand{\KMi}[1]{{\color{black} #1}}
\newcommand{\KMj}[1]{{\color{black} #1}}
\newcommand{\KMk}[1]{{\color{black} #1}}
\newcommand{\KMl}[1]{{\color{black} #1}}

\usepackage{color}
\begin{document}

\title{Multi-order asymptotic expansion of blow-up solutions for autonomous ODEs. II - 
Dynamical Correspondence}

\author[1]{Hisatoshi Kodani}
\author[1,2]{Kaname Matsue\thanks{(Corresponding author, {\tt kmatsue@imi.kyushu-u.ac.jp})}}
\author[1]{\\ Hiroyuki Ochiai}
\author[3]{Akitoshi Takayasu}

\affil[1]{
\normalsize{
Institute of Mathematics for Industry, Kyushu University, Fukuoka 819-0395, Japan
}
}
\affil[2]{International Institute for Carbon-Neutral Energy Research (WPI-I$^2$CNER), Kyushu University, Fukuoka 819-0395, Japan}
\affil[3]{Faculty of Engineering, Information and Systems, University of Tsukuba, 1-1-1 Tennodai, Tsukuba, Ibaraki 305-8573, Japan}

\maketitle

\begin{abstract}
In this paper, we provide a natural correspondence of eigenstructures of \KMb{Jacobian} matrices associated with equilibria for appropriately transformed two systems describing finite-time blow-ups for ODEs with quasi-homogeneity in an asymptotic sense.
As a corollary, we see that asymptotic expansions of blow-ups proposed in Part I \cite{asym1} themselves provide a criterion of the existence of blow-ups with an intrinsic gap structure of stability information among two systems.
Examples provided in Part I \cite{asym1} are revisited to show the above correspondence.
\end{abstract}


\tableofcontents

\section{Introduction}
\label{section-intro}
As the sequel to Part I \cite{asym1}, this paper aims at describing an intrinsic nature of blow-up solutions of the Cauchy problem of an autonomous system of ODEs
\begin{equation}
\label{ODE-original}
{\bf y}' = \frac{d{\bf y}(t)}{dt}=f ({\bf y}(t)),\quad {\bf y}(t_0) = {\bf y}_0,
\end{equation}
where $t\in[t_0,t_{\max})$ with $t_0<t_{\max} \leq \infty$, $f:\mathbb{R}^n\to\mathbb{R}^n$ is a $C^r$ function with $r\geq 2$ and ${\bf y}_0\in\mathbb{R}^n$.
A solution ${\bf y}(t)$ is said to \KMi{\em blow up} at $t = t_{\max} < \infty$ if \KMh{its modulus} diverges as $t\to t_{\max}-0$.
The value $t_{\max}$, the maximal existence time, is then referred to as the {\em blow-up time} of a blow-up solution.
Throughout the rest of this paper, let $\theta(t)$ be given by
\begin{equation*}
\theta(t) = t_{\max}-t,
\end{equation*}
where $t_{\max}$ is assumed to \KMg{be a finite value} and to be known a priori.
\par
In Part I \cite{asym1}, a systematic methodology for calculating multi-order asymptotic expansions of blow-up solutions is provided for (\ref{ODE-original}) with an asymptotic property of $f$.
We have observed there that, assuming the existence of blow-up solutions with uniquely determined leading asymptotic behavior (referred to as {\em type-I} blow-up in the present paper), roots of a nonlinear system of (algebraic) equations called {\em the balance law} and the associated eigenvalues of Jacobian matrices, called {\em the blow-up power eigenvalues} of {\em the blow-up power-determining matrices}, essentially determine all possible terms appeared in asymptotic expansions of blow-up solutions.
In particular, in the case of type-I blow-ups, these algebraic objects can determine all the essential aspects of type-I blow-ups.
A quick review of this methodology is shown in Section \ref{section-review-asym}.
\par
On the other hand, the second, the fourth authors and \KMb{their} collaborators have recently developed a framework to characterize blow-up solutions from the viewpoint of dynamics at infinity (e.g. \cite{Mat2018, Mat2019}), and have derived machineries of computer-assisted proofs for the existence of blow-up solutions as well as their qualitative and quantitative features (e.g. \cite{LMT2021, MT2020_1, MT2020_2, TMSTMO2017}).
\KMh{As in} the present paper, \KMh{finite-dimensional} vector fields with scale invariance in an asymptotic sense, {\em asymptotic quasi-homogeneity} defined precisely in Definition \ref{dfn-AQH}, are mainly concerned.
The main idea is to apply {\em compactifications} of phase spaces associated with asymptotic quasi-homogeneity of vector fields and {\em time-scale desingularizations} at infinity to obtaining {\em desingularized vector fields} \KMa{so that {\em dynamics at infinity} makes sense.}
In this case, \KMa{divergent} solutions of the original ODE (\ref{ODE-original}) correspond to trajectories on (local) stable manifolds of invariant sets on the geometric object expressing the infinity in the compactified phase space, which is referred to as the {\em horizon}.
It is shown in preceding works that a generic dynamical property of invariant sets such as equilibria and  periodic orbits on the horizon, {\em hyperbolicity}, yields blow-up solutions whose leading asymptotic behavior, {\em the blow-up rate}, is uniquely determined by the quasi-homogeneous component of $f$, namely type-I, and these invariant sets.
The precise statements are reviewed in Section \ref{section-preliminary} in the case of blow-ups characterized by equilibria on the horizon.
This approach reduces the problem involving divergent solutions, including blow-up solutions, to the standard theory of dynamical systems, and several successful examples are shown in preceding works \cite{LMT2021, Mat2018, MT2020_1, MT2020_2, TMSTMO2017}.
In particular, this approach provides \KMb{a criterion for} the (non-)existence of blow-up solutions of systems we are interested in and their characterizations by means of the standard theory of dynamical systems, without any a priori knowledge of the blow-up structure in the original system (\ref{ODE-original}).
%
In successive studies \cite{Mat2019}, it is demonstrated that blow-up behavior other than type-I can be characterized by {\em nonhyperbolic} invariant sets on the horizon for associated desingularized vector fields.
\par
\bigskip
Now we have two characterizations of blow-up solutions: (i). {\em multi-order asymptotic expansions} and (ii). {\em trajectories on local stable manifolds of invariant sets on the horizon for desingularized vector fields}, where the the asymptotic behavior of leading terms is assumed to be identical\footnote{
Indeed, assumptions for characterizing multi-order asymptotic expansions discussed in Part I \cite{asym1} are based on characterizations of blow-ups by means of dynamics at infinity.
}. 
It is then natural to ask {\em whether there is a correspondence between \KMb{these} two characterizations of blow-up solutions}. 
The main aim of the present paper is to answer this question.
More precisely, we provide the following one-to-one correspondences.
\begin{itemize}
\item Roots of the balance law (in asymptotic expansions) and equilibria on the horizon (for desingularized vector fields) describing blow-ups in forward time direction \KMc{(Theorem \ref{thm-balance-1to1})}.
\item Eigenstructure between the associated blow-up power-determining matrices (in asymptotic expansions) and the Jacobian matrices at the above equilibria on the horizon (for desingularized vector fields) \KMc{(Theorem \ref{thm-blow-up-estr})}.
\end{itemize}
These correspondences provide us with significant benefits about blow-up characterizations.
First, {\em asymptotic expansions of blow-up solutions themselves provide a criterion of their existence} \KMc{(Theorem \ref{thm-existence-blow-up})}.
In general, asymptotic expansions are considered {\em assuming the existence of the corresponding blow-up solutions in other arguments}.
On the other hand, the above correspondences imply that roots of the balance law and blow-up power eigenvalues, which essentially characterize asymptotic expansions of blow-ups, can provide the existence of blow-ups.
More precisely, these algebraic objects provide linear information of dynamics around equilibria on the horizon for desingularized vector fields, which is sufficient to verify the existence of blow-ups, as provided in preceding works.
Second, the correspondence of eigenstructure provides the gap of stability information between two systems of our interests \KMc{(Theorem \ref{thm-stability})}.
In particular, {\em stabilities of the corresponding equilibria for two systems describing an identical blow-up solution are always different}.
This gap warns us to take care of the dynamical interpretation of blow-up solutions and their perturbations, depending on the choice of systems we consider.

\par
\bigskip

The rest of this paper is organized as follows.
In Section \ref{section-preliminary}, a methodology for characterizing blow-up solutions from the viewpoint of dynamical systems based on preceding works (e.g. \cite{Mat2018, MT2020_1}) is quickly reviewed.
The precise definition of the class of vector fields we mainly treat is presented there.
The methodology successfully extracts blow-up solutions without those knowledge in advance, as already reported in preceding works (e.g. \cite{Mat2018, Mat2019, MT2020_1}).
\par
In Section \ref{section-correspondence}, the correspondence of structures characterizing blow-ups is discussed.
First, all notions necessary to characterize multi-order asymptotic expansions of type-I blow-up solutions proposed in Part I \cite{asym1} are reviewed.
Second, we extract one-to-one correspondence between roots of the balance law in asymptotic expansions and equilibria on the horizon for desingularized vector fields.
Note that equilibria on the horizon for desingularized vector fields can \KMb{also} characterize blow-ups in {\em backward} time direction, but the present correspondence excludes such equilibria due to the form of the system for deriving asymptotic expansions.
Third, we prove the existence of a common eigenstructure which {\em all} roots of the balance law, in particular all blow-up solutions of our interests, must possess.
As a consequence of the correspondence of eigenstructure, we also prove the existence of the corresponding eigenstructure in the desingularized vector fields such that all solutions asymptotic to equilibria on the horizon for desingularized vector fields in forward time direction must possess.
The common structures and their correspondence provide the gap of stability information for blow-up solutions between two systems we mentioned.
Finally, we provide the full correspondence of eigenstructures with possible multiplicity of eigenvalues.
As a corollary, we obtain a new criterion of the existence of blow-up solutions by means of roots of the balance law and associated blow-up power eigenvalues, in particular asymptotic expansions of blow-up solutions, and the stability gap depending on the choice of systems.
\par
In Section \ref{section-examples}, we revisit examples shown in Part I \cite{asym1} and confirm that our results indeed extract characteristic features of type-I blow-ups and the correspondence stated in main results.

\begin{rem}[\KMe{A correspondence to Painlev\'{e}-type analysis}]
Several results shown in Section \ref{section-correspondence} are closely related to Painlev\'{e}-type analysis for complex ODEs.
We briefly refer to several preceding works for accessibility.
In \cite{AM1989}, quasi-homogeneous (in the similar sense to Definition \ref{dfn-AQH}, referred to as {\em weight-homogeneity} in \cite{AM1989}) complex ODEs are concerned\KMi{,} and the {\em algebraic complete integrability} for complex Hamiltonian systems is considered. 
One of main results \KMe{there} is the characterization of {\em formal} Laurent solutions of quasi-homogeneous complex ODEs to be {\em convergent} by means of the {\em indicial locus} \KMi{$\mathscr{C}$,} and algebraic properties of the {\em Kovalevskaya matrix (Kowalewski matrix in \cite{AM1989}) $\mathscr{L}$} evaluated at points on $\mathscr{C}$.
Eigenstructures of $\mathscr{L}$ is also related to invariants and geometric objects by means of invariant manifolds for the flow and divisors in \KMe{Abelian varieties}.
Furthermore, the family of convergent Laurent solutions leads to affine varieties of parameters called \KMe{{\em Painlev\'{e} varieties}.}
%
The similar characterization of integrability appears in \cite{C2015}, where asymptotically quasi-homogeneous polynomial vector fields are treated.
It is proved in \cite{C2015} that eigenvalues ${\rm Spec}(\mathscr{L})$ of the Kovalevskaya matrix $\mathscr{L}$, referred to as {\em Kovalevskaya exponents}, are invariants under locally analytic transformations around points on $\mathscr{C}$, and formal Laurent series as convergent solutions of polynomial vector fields on weighted projective spaces are characterized by the structure of ${\rm Spec}(\mathscr{L})$.
These results are applied to integrability of polynomial vector fields (the {\em extended Painlev\'{e} test} in \cite{C2015}) and the first Painlev\'{e} hierarchy, as well as classical Painlev\'{e} equations in the subsequent papers \cite{C2016_124, C2016_356}. 
\par
We remark that 
\KMi{the indicial locus $\mathscr{C}$ corresponds to a collection of roots of the balance law (Definition \ref{dfn-balance}), and that }
the matrix $\mathscr{L}$ is essentially the same as blow-up power-determining matrix, and ${\rm Spec}(\mathscr{L})$ corresponds to blow-up power eigenvalues.
It is therefore observed that there are several similarities of characterizations between Painlev\'{e}-type properties and blow-up behavior.
It should be noted here\KMe{, however, } that exponents being {\em integers with an identical sign} and {\em semi-simple} have played key roles in characterizing integrability in studies of the Painlev\'{e}-type properties, as stated in the above references.
\KMe{In contrast}, only the identical sign is essential to determine blow-up asymptotics\KMh{.}
We notice that essential ideas appeared in algebraic geometry and Painlev\'{e}-type analysis also make contributions to extract blow-up characteristics for ODEs in a general setting.
\end{rem}


\section{Preliminaries: blow-up description through dynamics at infinity}
\label{section-preliminary}

In this section, we briefly review a characterization of blow-up solutions for autonomous, finite-dimensional ODEs from the viewpoint of dynamical systems.
Details of the present methodology are already provided in \cite{Mat2018, MT2020_1}.
%
%
\subsection{Asymptotically quasi-homogeneous vector fields}
\label{sec:QH}

First of all, we review a class of vector fields in our present discussions.

\begin{dfn}[Asymptotically quasi-homogeneous vector fields, cf. \cite{D1993, Mat2018}]\rm
\label{dfn-AQH}
Let $f_0: \mathbb{R}^n \to \mathbb{R}$ be a function.
Let $\alpha_1,\ldots, \alpha_n$ \KMl{be nonnegative integers with $(\alpha_1,\ldots, \alpha_n) \not = (0,\ldots, 0)$ and $k > 0$.}
We say that $\KMf{f_0}$ is a {\em quasi-homogeneous function\footnote{
\KMl{In preceding studies, all $\alpha_i$'s and $k$ are typically assumed to be natural numbers.
In the present study, on the other hand, the above generalization is valid.
}
} of type $\KMh{\alpha = }(\alpha_1,\ldots, \alpha_n)$ and order $k$} if
\begin{equation*}
f_0( s^{\Lambda_\alpha}{\bf x} ) = s^k f_0( {\bf x} )\quad \text{ for all } {\bf x} = (x_1,\ldots, x_n)^T \in \mathbb{R}^n \text{ and } s>0,
\end{equation*}
where\KMb{\footnote{
Throughout the rest of this paper, the power of real positive numbers or functions to matrices is described in the similar manner.
}}
\begin{equation*}
\Lambda_\alpha =  {\rm diag}\left(\alpha_1,\ldots, \alpha_n\right),\quad s^{\Lambda_\alpha}{\bf x} = (s^{\alpha_1}x_1,\ldots, s^{\alpha_n}x_n)^T.
\end{equation*}
Next, let $X = \sum_{i=1}^n f_i({\bf x})\frac{\partial }{\partial x_i}$ be a continuous vector field on $\mathbb{R}^n$.
We say that $X$, or simply $f = (f_1,\ldots, f_n)^T$ is a {\em quasi-homogeneous vector field of type $\alpha = (\alpha_1,\ldots, \alpha_n)$ and order $k+1$} if each component $f_i$ is a quasi-homogeneous function of type $\alpha$ and order $k + \alpha_i$.
\par
Finally, we say that $X = \sum_{i=1}^n f_i({\bf x})\frac{\partial }{\partial x_i}$, or simply $f$ is an {\em asymptotically quasi-homogeneous vector field of type $\alpha = (\alpha_1,\ldots, \alpha_n)$ and order $k+1$ at infinity} if there is a quasi-homogeneous vector field  $f_{\alpha,k} = (f_{i; \alpha,k})_{i=1}^n$ of type $\alpha$ and order $k+1$ such that
\begin{equation*}
\KMk{f_i( s^{\Lambda_\alpha}{\bf x} ) - s^{k+\alpha_i} f_{i;\alpha,k}( {\bf x} ) = o(s^{k+\alpha_i})}
 \end{equation*}
\KMk{as $s\to +\infty$} uniformly for \KMf{${\bf x} = (x_1,\ldots, x_n)\in S^{n-1} \equiv \{{\bf x}\in \mathbb{R}^n \mid \sum_{i=1}^n x_i^2 = 1\}$}.
\end{dfn}

A fundamental property of quasi-homogeneous functions and vector fields is reviewed here.
\begin{lem}\label{temporary-label2}
A quasi-homogenous function $f_0$ of type $(\alpha_1,\ldots,\alpha_n)$ and order $k$ satisfies the following differential equation:
\begin{equation}\label{temporary-label1}
\sum_{l=1}^n \alpha_l y_l \frac{\partial f_0}{\partial y_l}({\bf y}) = k f_0({\bf y}).
\end{equation}
This equation is rephrased as
\begin{equation*}
(\nabla_{\bf y} f_0({\bf y}))^T \Lambda_{\alpha} {\bf y} = k f_0({\bf y}).
\end{equation*}
\end{lem}

\begin{proof}
Differentiating the identity
\[
f_0(s^{\Lambda_\alpha}{\bf y}) = s^k f_0({\bf y})
\]
in $s$ and put $s=1$, we obtain
the desired equation \eqref{temporary-label1}.
\end{proof}
The same argument yields that, for any quasi-homogenenous function $f_0$ of type $\alpha = (\alpha_1,\ldots, \alpha_n)$ and order $k$, 
\begin{equation*}
\sum_{\KMb{l}=1}^n \alpha_\KMb{l} s^{\alpha_\KMb{l}}y_\KMb{l} \frac{\partial f_0}{\partial y_l}(s^{\Lambda_\alpha}{\bf y}) = ks^k f_0({\bf y})
\end{equation*}
and 
\begin{equation*}
\sum_{\KMb{l}=1}^n \alpha_\KMb{l} (\alpha_\KMb{l} - 1) s^{\alpha_\KMb{l}} y_\KMb{l} \frac{\partial f_0}{\partial y_l}(s^{\Lambda_\alpha}{\bf y}) + \sum_{j,l = 1}^n \alpha_j \alpha_l (s^{\alpha_j}y_j) (s^{\alpha_l}y_l) \frac{\partial^2 f_0}{\partial y_j \partial y_l}(s^{\Lambda_\alpha}{\bf y})  = k(k-1)s^k f_0({\bf y})
\end{equation*}
for any ${\bf y}\in \mathbb{R}^n$.
In particular, each partial derivative satisfies
\begin{equation}
\label{order-QH-derivatives}
\frac{\partial f_0}{\partial y_l}(s^{\Lambda_\alpha}{\bf y}) = O\left(s^{k-\alpha_l} \right),\quad 
\frac{\partial^2 f_0}{\partial y_j \partial y_l}(s^{\Lambda_\alpha}{\bf y}) = O\left(s^{k-\alpha_j - \alpha_l} \right)
\end{equation}
as $s\to 0, \infty$ for any ${\bf y}\in \mathbb{R}^n$, as long as $f_0$ is $C^2$ in the latter case.
In particular, \KMl{for any fixed ${\bf y}\in \mathbb{R}^n$,} both derivatives are $O(1)$ as $s\to 1$.
\KMk{In other words, we have quasi-homogeneous relations for partial derivatives in the above sense.}

\begin{lem}
\label{lem-identity-QHvf}
A quasi-homogeneous vector field $f=(f_1,\ldots,f_n)$ of 
type $\alpha = (\alpha_1,\ldots,\alpha_n)$ and order $k+1$ satisfies the following differential equation:
\begin{equation}
\label{temporary-label3}
\sum_{l=1}^n \alpha_l y_l \frac{\partial f_i}{\partial y_l}({\bf y}) = (k+\alpha_i) f_i({\bf y}) \qquad (i=1,\ldots,n).
\end{equation}
This equation can be rephrased as
\begin{equation}\label{temporay-label4}
(D f)({\bf y}) \Lambda_\alpha \mathbf{y} = \left( k I+ \Lambda_\alpha \right) f({\bf y}).
\end{equation}
\end{lem}
\begin{proof}
By Lemma~\ref{temporary-label2}, 
we obtain \eqref{temporary-label3}.
For \eqref{temporay-label4},
we recall that $Df=(\frac{\partial f_i}{\partial y_l})$ is Jacobian matrix, 
while $kI + \Lambda_\alpha$ is the diagonal matrix with diagonal entries $k+\alpha_i$.
Finally $\mathbf{y} = (y_1,\ldots,y_n)^T$ is the column vector,
so that the left-hand side of \eqref{temporay-label4} is the product of two matrices and one column vector.
\end{proof}

Throughout successive sections, consider an (autonomous) $C^r$ vector field\footnote{
$C^1$-smoothness is sufficient to consider the correspondence discussed in Section \ref{section-correspondence} in the present paper.
$C^2$-smoothness is actually applied to justifying multi-order asymptotic expansions of blow-up solutions, which is discussed in Part I \cite{asym1}.
} (\ref{ODE-original}) with $r\geq 2$, 
where $f: \mathbb{R}^n \to \mathbb{R}^n$ is asymptotically quasi-homogeneous of type $\alpha = (\alpha_1,\ldots, \alpha_n)$ and order $k+1$ at infinity.

%
%
%
%

\subsection{Quasi-parabolic compactifications}
\label{sec:global}

Here we review an example of compactifications which embed the original (locally compact) phase space into a compact manifold to characterize \lq\lq infinity" as a bounded object.
While there are several choices of compactifications, the following compactification to characterize {\em dynamics at infinity} is applied here.

\begin{dfn}[Quasi-parabolic compactification, \cite{MT2020_1}]\rm
\label{dfn-quasi-para}
Let the type $\alpha = (\alpha_1,\ldots, \alpha_n)\in \mathbb{N}^n$ fixed.
Let $\{\beta_i\}_{i=1}^n$ be the collection of natural numbers so that 
\begin{equation}
\label{LCM}
\alpha_1 \beta_1 = \alpha_2 \beta_2 = \cdots = \alpha_n \beta_n \equiv c \in \mathbb{N}
\end{equation}
is the least common multiplier.
In particular, $\{\beta_i\}_{i=1}^n$ is chosen to be the smallest among possible collections.
Let $p({\bf y})$ be a functional given by
\begin{equation}
\label{func-p}
p({\bf y}) \equiv \left( y_1^{2\beta_1} +  y_2^{2\beta_2} + \cdots +  y_n^{2\beta_n} \right)^{1/2c}.
\end{equation}
Define the mapping $T: \mathbb{R}^n \to \mathbb{R}^n$ as the inverse of
\begin{equation*}
S({\bf x}) = {\bf y},\quad y_j = \kappa^{\alpha_j} x_j,\quad j=1,\ldots, n,
\end{equation*}
where 
\begin{equation*}
\kappa = \kappa({\bf x}) = (1- p({\bf x})^{2c})^{-1} \equiv \left( 1 - \sum_{j=1}^n x_j^{2\beta_j}\right)^{-1}.
\end{equation*}
We say the mapping $T$ the {\em quasi-parabolic compactification (with type $\alpha$)}.
\end{dfn}

\begin{rem}
\label{rem-kappa}
The functional $\kappa = \tilde \kappa({\bf y})$ as a functional determined by ${\bf y}$ is implicitly determined by $p({\bf y})$.
Details of such a characterization of $\kappa$ in terms of ${\bf y}$, and the bijectivity and smoothness of $T$ are shown in \cite{MT2020_1} with a general class of compactifications including quasi-parabolic compactifications.
\end{rem}

\begin{rem}
\label{rem-zero-comp-compactification}
\KMb{When there is a component $i_0$ such that $\alpha_{i_0} = 0$, we apply the compactification {\em only for components with nonzero $\alpha_i$}.}
\end{rem}

As proved in \cite{MT2020_1}, $T$ maps $\mathbb{R}^n$ one-to-one onto
\begin{equation*}
\mathcal{D} \equiv \{{\bf x}\in \mathbb{R}^n \mid p({\bf x}) < 1\}.
\end{equation*}
Infinity in the original coordinate then corresponds to a point on the boundary
\begin{equation*}
\mathcal{E} = \{{\bf x} \in \mathbb{R}^n \mid p({\bf x}) = 1\}.
\end{equation*}
\begin{dfn}\rm
We call the boundary $\mathcal{E}$ of $\mathcal{D}$ the {\em horizon}.
\end{dfn}

It is easy to understand the geometric nature of the present compactification when $\alpha = (1,\ldots, 1)$, in which case $T$ is defined as
\begin{equation*}
x_j = \frac{2y_j}{1+\sqrt{1+4\|{\bf y}\|^2}}\quad \Leftrightarrow \quad y_j = \frac{x_j}{1-\|{\bf x}\|^2},\quad j=1,\ldots, n.
\end{equation*}
See \cite{EG2006, G2004} for the homogeneous case, which is called the {\em parabolic compactification}.
In this case, the functional $\kappa = \tilde \kappa({\bf y})$ mentioned in Remark \ref{rem-kappa} is explicitly determined by
\begin{equation*}
\kappa = \tilde \kappa({\bf y}) = \frac{1+\sqrt{1+4\|{\bf y}\|^2}}{2} = \kappa({\bf x}) = \frac{1}{1-\|{\bf x}\|^2}.
\end{equation*}
A homogeneous compactification of this kind is shown in Figure \ref{fig:compactification}-(a), while an example of quasi-parabolic one is shown in Figure \ref{fig:compactification}-(b).

\begin{figure}[h!]\em
\begin{minipage}{0.5\hsize}
\centering
\includegraphics[width=6cm]{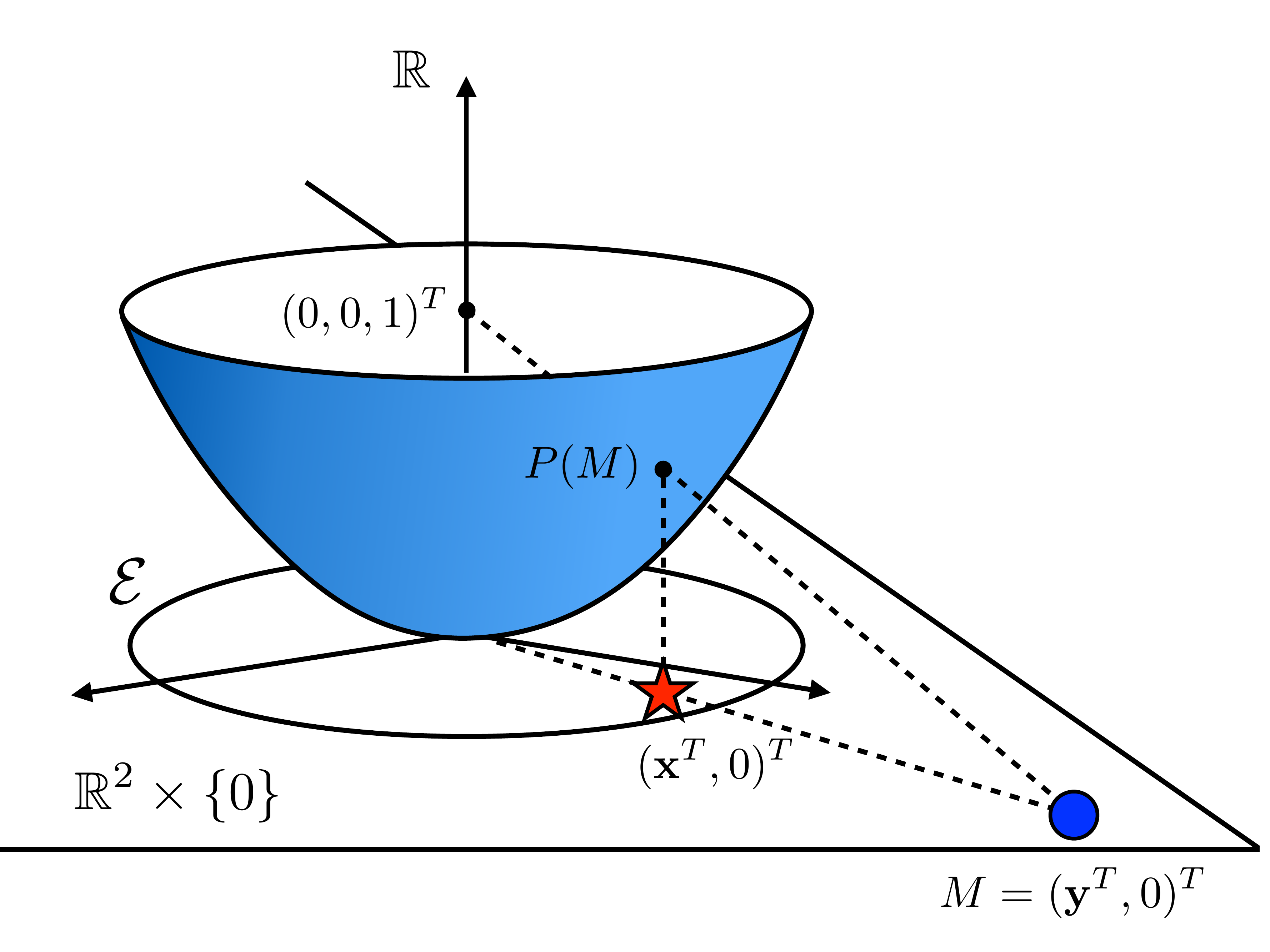}
(a)
\end{minipage}
\begin{minipage}{0.5\hsize}
\centering
\includegraphics[width=6cm]{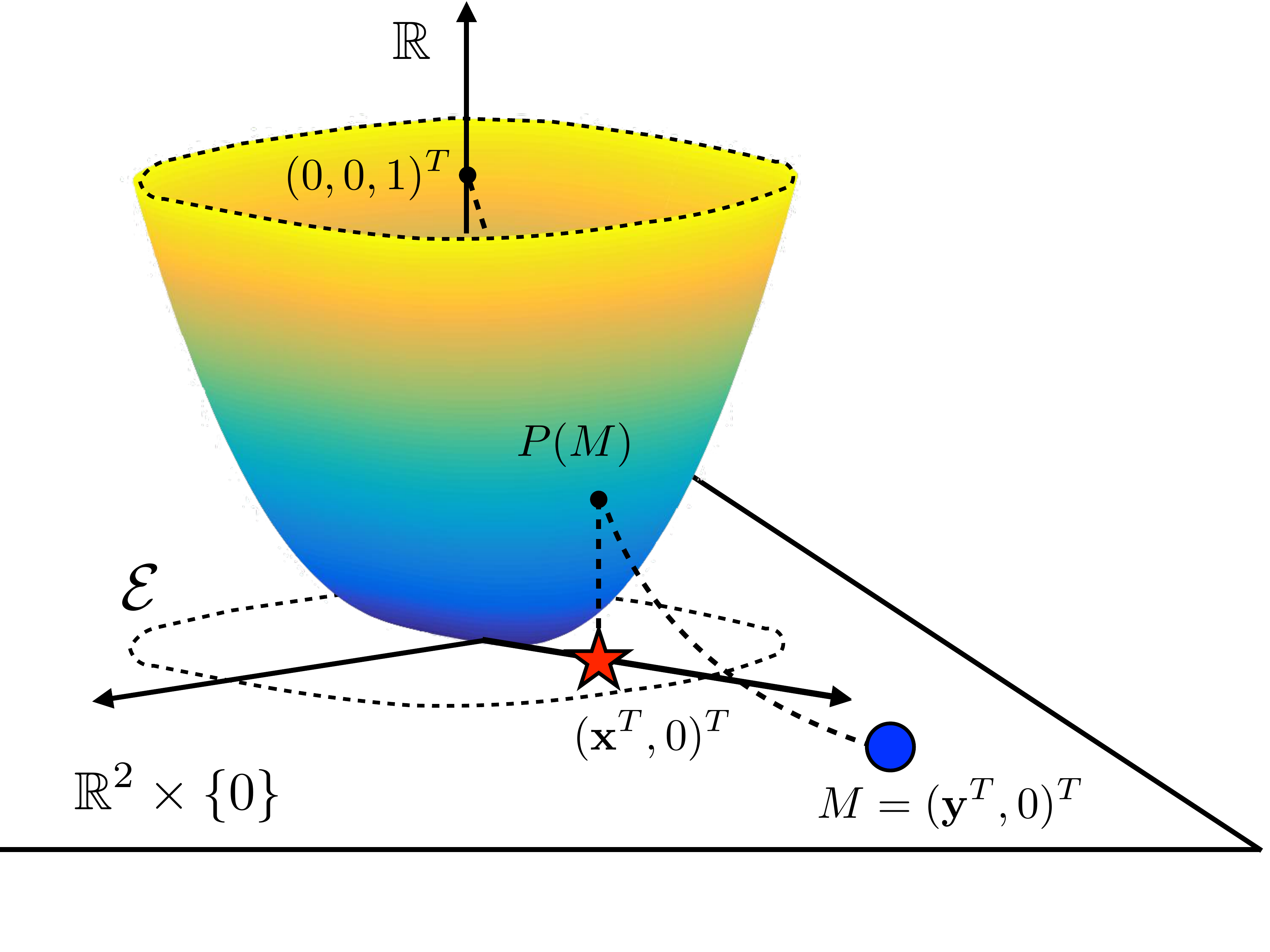}
(b)
\end{minipage}
\caption{Parabolic-type compactifications of $\mathbb{R}^2$}
\label{fig:compactification}
\par
(a): Parabolic compactification with type $\alpha = (1,1)$.
The image ${\bf x}$ of the original point ${\bf y}\in \mathbb{R}^2$ is defined as the projection of the intersection point $P(M)\in \mathcal{H}$ determined by the paraboloid $x_1^2 + x_2^2 = x_3$ in $\mathbb{R}^3$ and the line segment connecting $M = ({\bf y}^T,0)^T\in \mathbb{R}^3$ and the focus point $(0,0,1)^T\in \mathbb{R}^3$, onto the original phase space $\mathbb{R}^2$.
The horizon is identified with the circle $\{x_1^2 + x_2^2 = 1, x_3 = 1\}$ on the parabola. 
The precise definition is its projection onto $\mathbb{R}^2\times \{0\}$.
(b): Quasi-parabolic compactification with type $\alpha = (2,1)$.
\end{figure}

\begin{rem}
\label{rem-compactification}
Global-type compactifications like parabolic ones are typically introduced as {\em homogeneous} ones, namely $\alpha_1 = \cdots = \alpha_n = 1$.
Simple examples of global compactifications are {\em Bendixson}, or {\em one-point} compactification (e.g. embedding of $\mathbb{R}^n$ into $S^n$) and {\em Poincar\'{e}} compactification (i.e., embedding of $\mathbb{R}^n$ into the hemisphere).
Among such compactifications, Poincar\'{e}-type ones are considered to be a prototype of {\em admissible} compactifications discussed in \cite{MT2020_1} (see also \cite{EG2006}), which distinguishes directions of infinity and characterizes dynamics at infinity appropriately, as mentioned in Section \ref{label-dynamics-infinity}.
Quasi-homogeneous-type, global compactifications are introduced in \cite{Mat2018, MT2020_1} as quasi-homogeneous counterparts of homogeneous compactifications.
However, the Poincar\'{e}-type compactifications include radicals in the definition (e.g. \cite{Mat2018}), which cause the loss of smoothness of vector fields on the horizon (mentioned below) and the applications to dynamics at infinity are restrictive in general.
On the other hand, {\em homogeneous}, parabolic-type compactifications were originally introduced in \cite{G2004} so that unbounded rational functions are transformed into rational functions. 
In particular, smoothness of rational functions are preserved through the transformation.
The present compactification is the quasi-homogeneous counterpart of the homogeneous parabolic compactifications.
\par
There are alternate compactifications which are defined {\em locally}, known as e.g. {\em Poincar\'{e}-Lyapunov disks} (e.g. \cite{DH1999, DLA2006}) and are referred to as {\em directional compactifications} in e.g. \cite{Mat2018, Mat2019, MT2020_2}.
These compactifications are quite simple and are widely used to study dynamics at infinity.
However, one chart of these compactifications can lose symmetric features of dynamics at infinity and perspectives of correspondence between dynamics at infinity and asymptotic behavior of blow-up solutions.
This is the reason why we have chosen parabolic-type compactifications as the one of our central issues.
\end{rem}

\subsection{Dynamics at infinity and blow-up characterization}
\label{label-dynamics-infinity}
Once we fix a compactification associated with the type $\alpha = (\alpha_1, \ldots, \alpha_n)$ of the vector field $f$ with order $k+1$, we can derive the vector field which makes sense including the horizon.
Then the {\em dynamics at infinity} makes sense through the appropriately transformed vector field called the {\em desingularized vector field}, \KMb{denoted by $g$}.
The common approach is twofold.
Firstly, we rewrite the vector field (\ref{ODE-original}) with respect to the new variable used in compactifications.
Secondly, we introduce the time-scale transformation of the form $d\tau = q({\bf x})\kappa({\bf x}(t))^k dt$ for some function $q({\bf x})$ which is bounded including the horizon. 
We then obtain the vector field with respect to the new time variable $\tau$, which is continuous, including the horizon.

\begin{rem}
\label{rem-choice-cpt}
Continuity of the desingularized vector field $g$ including the horizon is guaranteed by the smoothness of $f$ and asymptotic quasi-homogeneity (\cite{Mat2018}).
In the case of parabolic-type compactifications, $g$ inherits the smoothness of $f$ including the horizon, which is not always the case of other compactifications in general. 
Details are discussed in \cite{Mat2018}.
\end{rem}

\begin{dfn}[Time-scale desingularization]\rm 
Define the new time variable $\tau$ by
\begin{equation}
\label{time-desing-para}
d\tau = (1-p({\bf x})^{2c})^{-k}\left\{1-\frac{2c-1}{2c}(1-p({\bf x})^{2c}) \right\}^{-1}dt,
\end{equation}
equivalently,
\begin{equation*}
t - t_0 = \int_{\tau_0}^\tau \left\{1-\frac{2c-1}{2c}(1-p({\bf x}(\tau))^{2c}) \right\}(1-p({\bf x}(\tau))^{2c})^k d\tau,
\end{equation*}
where $\tau_0$ and $t_0$ denote the correspondence of initial times, ${\bf x}(\tau) = T({\bf y}(\tau))$ and ${\bf y}(\tau)$ is a solution ${\bf y}(t)$ under the parameter $\tau$.
We shall call (\ref{time-desing-para}) {\em the time-scale desingularization of order $k+1$}.
\end{dfn}
The change of coordinate and the above desingularization yield the following vector field $g = (g_1, \ldots, g_n)^T$, which is continuous on $\overline{\mathcal{D}} = \{p({\bf x}) \leq 1\}$:
\begin{align}
\label{desing-para}
\dot x_i \equiv \frac{dx_i}{d\tau} = g_i({\bf x}) = \left(1-\frac{2c-1}{2c}(1-p({\bf x})^{2c}) \right) \left\{ \tilde f_i({\bf x}) - \alpha_i x_i \sum_{j=1}^n (\nabla \kappa)_j \kappa^{\alpha_j - 1}\tilde f_j({\bf x})\right\},
\end{align}
where 
\begin{equation}
\label{f-tilde}
\tilde f_j(x_1,\ldots, x_n) := \kappa^{-(k+\alpha_j)} f_j(\kappa^{\alpha_1}x_1, \ldots, \kappa^{\alpha_n}x_n),\quad j=1,\ldots, n,
\end{equation}
and
$\nabla \kappa = \nabla_{\bf x} \kappa = ((\nabla_{\bf x} \kappa)_1, \ldots, (\nabla_{\bf x} \kappa)_n)^T$ is
\begin{equation*}
(\nabla_{\bf x} \kappa)_j 
= \frac{\kappa^{1-\alpha_j} x_j^{2\beta_j-1}}{\alpha_j \left(1- \frac{2c-1}{2c} (1 - p({\bf x})^{2c} )\right)},\quad j=1,\ldots, n,
\end{equation*}
as derived in \cite{MT2020_1}.
In particular, the vector field $g$ is written as follows:
\begin{align}
\label{desing-vector}
g({\bf x}) = \left(1-\frac{2c-1}{2c}(1-p({\bf x})^{2c}) \right)\tilde f({\bf x}) - G({\bf x})\Lambda_\alpha {\bf x},
\end{align}
where $\tilde f = (\tilde f_1,\ldots, \tilde f_n)^T$ and
\begin{align}
\label{Gx}
G({\bf x}) &\equiv \sum_{j=1}^n \frac{x_j^{2\beta_j-1}}{\alpha_j}\tilde f_j({\bf x}).
\end{align}
Smoothness of $f$ and the asymptotic quasi-homogeneity guarantee the smoothness of the right-hand side $g$ of (\ref{desing-para}) including the horizon $\mathcal{E}\equiv \{p({\bf x}) = 1\}$.
In particular, {\em dynamics at infinity}, such as divergence of solution trajectories to specific directions,  is characterized through dynamics generated by (\ref{desing-para}) around the horizon. 
See \cite{Mat2018, MT2020_1} for details.

\begin{rem}[Invariant structure]
\label{rem-invariance}
The horizon $\mathcal{E}$ is a codimension one invariant submanifold of $\overline{\mathcal{D}}$. 
Indeed, direct calculations yield that
\begin{equation*}
\left. \frac{d}{d \tau}p({\bf x}(\tau))^{2c}\right|_{\tau=0} = 0\quad \text{ whenever }\quad {\bf x}(0)\in \mathcal{E}.
\end{equation*}
See e.g. \cite{Mat2018}, where detailed calculations are shown in a similar type of global compactifications.
We shall apply this invariant structure to extracting the detailed blow-up structure later.
\end{rem}

\subsection{Type-I stationary blow-up}
Through the compactification we have introduced, 
dynamics around the horizon characterize dynamics at infinity, including blow-up behavior.
\par
For an equilibrium $\bar {\bf x}$ for a desingularized vector field $g$, let 
\begin{equation*}
W_{\rm loc}^s(\bar {\bf x}) = W_{\rm loc}^s(\bar {\bf x};g) := \{{\bf x}\in U \mid |\varphi_g(t,{\bf x}) - \bar {\bf x}| \to 0\, \text{ as }\,t\to +\infty\}
\end{equation*}
be the {\em (local) stable set} of $\bar {\bf x}$ for the dynamical system generated by $g$, where $U$ is a neighborhood of $\bar {\bf x}$ in $\mathbb{R}^n$ or an appropriate phase space, and $\varphi_g$ is the flow generated by $g$.
In a special case where $\bar {\bf x}$ is a {\em hyperbolic equilibrium} for $g$, that is, an equilibrium satisfying ${\rm Spec}(Dg(\bar {\bf x})) \cap i\mathbb{R} = \emptyset$, the stable set $W_{\rm loc}^s(\bar {\bf x})$ admits a smooth manifold structure in a small neighborhood of $\bar {\bf x}$ (see {\em the Stable Manifold Theorem} in e.g. \cite{Rob}). 
In such a case, the set is referred to as {\em (local) stable manifold} of $\bar {\bf x}$.
Here we review a result for characterizing blow-up solutions by means of stable manifolds of equilibria, which is shown in \cite{MT2020_1} (cf. \cite{Mat2018}).

\begin{thm}[Stationary blow-up, \cite{Mat2018, MT2020_1}]
\label{thm:blowup}
Assume that the desingularized vector field $g$ given by (\ref{desing-para}) associated with (\ref{ODE-original}) admits an equilibrium on the horizon ${\bf x}_\ast \in \mathcal{E}$.
Suppose that ${\bf x}_\ast$ is hyperbolic, in particular, the Jacobian matrix $Dg({\bf x}_\ast)$ of $g$ at ${\bf x}_\ast$ possesses $n_s > 0$ (resp. $n_u = n-n_s$) eigenvalues with negative (resp. positive) real part.
If there is a solution ${\bf y}(t)$ of (\ref{ODE-original}) with a bounded initial point ${\bf y}(0)$ whose image ${\bf x} = T({\bf y})$ is on the stable manifold\footnote{
In the present case, a neighborhood $U$ of ${\bf x}_\ast$ determining $W^s_{\rm loc}$ is chosen as a subset of $\overline{\mathcal{D}}$.
} $W_{\rm loc}^s({\bf x}_\ast; g)$, then $t_{\max} < \infty$ holds; namely, ${\bf y}(t)$ is a blow-up solution.
Moreover,
\begin{equation*}
\kappa\equiv \kappa({\bf x}(t)) \sim \tilde c\theta(t)^{-1/k}\quad \text{ as }\quad t\to t_{\max}\KMc{-0},
\end{equation*}
where $\tilde c > 0$ is a constant.
Finally, if the $j$-th component $x_{\ast, j}$ of ${\bf x}_\ast$ is not zero, then we also have
\begin{equation*}
y_j(t) \sim \KMb{\tilde c_j}\theta(t)^{-\alpha_j /k}\quad \text{ as }\quad t\to t_{\max},
\end{equation*}
where $\KMb{\tilde c_j}$ is a constant with the same sign as $y_j(t)$ as $t\to t_{\max}$.
\end{thm}

The key point of the theorem is that {\em blow-up solutions for (\ref{ODE-original}) are characterized as trajectories on \KMb{local} stable manifolds of equilibria or general invariant sets\footnote{
In \cite{Mat2018}, a characterization of blow-up solutions with infinite-time oscillations in $t < t_{\max}$ and unbounded amplitude is also provided by means of time-periodic orbits on the horizon for desingularized vector fields, which referred to as {\em periodic blow-up}.
Hyperbolicity ensures not only blow-up behavior of solutions but their asymptotic behavior with the specific form.
Several case studies of blow-up solutions beyond hyperbolicity are shown in \cite{Mat2019}.
} on the horizon $\mathcal{E}$ for the desingularized vector field}.
Investigations of blow-up structure are therefore reduced to those of stable manifolds of invariant sets, such as (hyperbolic) equilibria, for the associated vector field.
Moreover, the theorem also claims that hyperbolic equilibria on the horizon induce {\em type-I blow-up. 
That is, the leading term of the blow-up behavior is determined by the type $\alpha$ and the order $k+1$ of $f$.}
\KMj{In particular, Theorem \ref{thm:blowup} can be used to verify the existence of blow-up solutions assumed in our construction of their asymptotic expansions.}
\par
The blow-up time $t_{\max}$ is explicitly given by (\ref{time-desing-para}):
\begin{equation}
\label{blow-up-time}
t_{\max} = t_0 + \frac{1}{2c}\int_{\tau_0}^\infty \left\{1+ (2c-1)p({\bf x}(\tau))^{2c} \right\}(1-p({\bf x}(\tau))^{2c})^k d\tau.
\end{equation}
The above formula is consistent with a well-known fact that $t_{\max}$ depends on initial points ${\bf x}_0 = {\bf x}(\tau_0)$.

\begin{rem}
Theorem \ref{thm:blowup} itself does not tell us the asymptotic behavior of the component $y_j(t)$ as $t\to t_{\max}-0$ when $x_{\ast, j} = 0$.
Nevertheless, asymptotic expansion of the solution can reveal the detailed behavior near $t=t_{\max}$, as demonstrated in Part I \cite{asym1}.
\end{rem}

We end this section by providing a lemma for a function appeared in (\ref{desing-para}), which is essential to characterize equilibria on the horizon from the viewpoint of asymptotic expansions. 
The gradient of the horizon $\mathcal{E} = \{p({\bf x}) = 1\}$ at ${\bf x}\in \mathcal{E}$ is given by
\begin{equation*}
\nabla p({\bf x}) = \frac{p({\bf x})^{\KMc{1-2c}}}{c} \left( \beta_1x_1^{2\beta_1-1}, \ldots, \beta_nx_n^{2\beta_n-1} \right)^T = \frac{1}{c} \left( \beta_1x_1^{2\beta_1-1}, \ldots, \beta_nx_n^{2\beta_n-1} \right)^T.
\end{equation*}
In particular, 
\begin{equation}
\label{grad-xast}
\nabla p({\bf x}_\ast) = \frac{1}{c}\left( \beta_1x_{\ast,1}^{2\beta_1-1}, \ldots, \beta_n x_{\ast,n}^{2\beta_n-1} \right)^T
\end{equation}
holds at an equilibrium ${\bf x}_\ast \in \mathcal{E}$.
Similarly, we observe that
\begin{equation*}
\nabla (p({\bf x})^{2c}) = 2\left( \beta_1x_1^{2\beta_1-1}, \ldots, \beta_nx_n^{2\beta_n-1} \right)^T 
\end{equation*}
for any ${\bf x}\in \overline{\mathcal{D}}.$
Using the gradient, the function $G({\bf x})$ in (\ref{Gx}) is also written by
\begin{equation*}
G({\bf x}) = \sum_{j=1}^n \frac{\beta_j}{c}x_j^{2\beta_j-1}\tilde f_j({\bf x}) = \frac{1}{2c}\nabla (p({\bf x})^{2c})^T \tilde f({\bf x}). 
\end{equation*}

\begin{lem}
\label{lem-G}
\begin{equation*}
\kappa^{-1} \frac{d\kappa}{d\tau} = G({\bf x}),
\end{equation*}
where $G({\bf x})$ is given in (\ref{Gx}).
\end{lem}

\begin{proof}
Direct calculations with (\ref{LCM}) yield that
\begin{align*}
\kappa^{-2} \frac{d\kappa}{d\tau} &\equiv -\frac{d(\kappa)^{-1}}{d\tau} \\
	&= \nabla (p({\bf x})^{2c})^T \frac{d{\bf x}}{d\tau}\\
	&= \nabla (p({\bf x})^{2c})^T\left( \frac{1}{2c}\left(1 + (2c-1)p({\bf x})^{2c} \right)\tilde f({\bf x}) - G({\bf x})\Lambda_\alpha {\bf x}\right) \\
	&= \left(1 + (2c-1)p({\bf x})^{2c} \right)G({\bf x}) - 2cp({\bf x})^{2c}G({\bf x})  \\
	&= (1-p({\bf x})^{2c})G({\bf x})\\
	&= \kappa^{-1}G({\bf x}).
\end{align*}
\end{proof}

\section{Correspondence between asymptotic expansions of blow-ups and dynamics at infinity}
\label{section-correspondence}

\KMa{In Part I \cite{asym1}, a systematic methodology of asymptotic expansions of blow-up solutions has been proposed}.
On the other hand, blow-up solutions can be also characterized by trajectories on the \KMi{local} stable manifold $\KMi{W^s_{\rm loc}}({\bf x}_\ast; g)$ of an equilibrium ${\bf x}_\ast$ on the horizon for the desingularized vector field $g$\KMi{, as reviewed in Section \ref{section-preliminary}}.
By definition, the manifold $\KMi{W^s_{\rm loc}}({\bf x}_\ast; g)$ consists of initial points converging to ${\bf x}_\ast$ as $\tau \to \infty$ and hence $\KMi{W^s_{\rm loc}}({\bf x}_\ast; g)$ characterizes the dependence of blow-up solutions on initial points, including the variation of $t_{\max}$.
One then expects that there is a common feature among algebraic information (\KMb{asymptotic expansions}) and geometric one (the stable manifold $\KMi{W^s_{\rm loc}}({\bf x}_\ast; g)$) for characterizing identical blow-up solutions.
\par
This section addresses several structural correspondences between asymptotic expansions of blow-up solutions and dynamics of equilibria on the horizon for desingularized vector fields.
As a corollary, we see that asymptotic expansions of blow-up solutions in the above methodology themselves provide a criterion for their existence\KMa{, as well as the existence of an intrinsic gap of stability information among two systems}.
Unless otherwise mentioned, let $g$ be the desingularized vector field \KMg{(\ref{desing-vector})} associated with $f$.

\subsection{Quick review of tools for multi-order asymptotic expansion of type-I blow-up solutions}
\label{section-review-asym}
Before \KMc{discussing correspondence of dynamical structures} among two systems, we quickly review the methodology of multi-order asymptotic expansions of type-I blow-up solutions proposed in \cite{asym1}.
The method begins with the following ansatz, which can be easily verified through the desingularized vector field and our blow-up characterization: Theorem \ref{thm:blowup}.

\begin{ass}
\label{ass-fundamental}
The asymptotically quasi-homogeneous system (\ref{ODE-original}) of type $\alpha$ and the order $k+1$ admits a solution 
\begin{equation*}
{\bf y}(t) = (y_1(t), \ldots, y_n(t))^T
\end{equation*}
 which blows up at $t = t_{\max} < \infty$ with the \KMb{type-I blow-up} behavior\KMb{, namely}\footnote{
\KMb{For two scalar functions $h_1$ and $h_2$, $h_1 \sim h_2$ as $t\to t_{\max}-0$ iff $(h_1(t)/h_2(t))\to 1$ as $t\to t_{\max}-0$. }
 }
\begin{equation}
\label{blow-up-behavior}
y_i(t) \KMb{ \sim c_i\theta(t)^{-\alpha_i / k}}, 
\quad t\to t_{\max}-0,\quad i=1,\ldots, n
\end{equation}
\KMb{for some constants $c_i\in \mathbb{R}$.}
\end{ass}
Our aim here is, under the above assumption, to write ${\bf y}(t)$ as
\begin{equation}
\label{blow-up-sol}
y_i(t) = \theta(t)^{-\alpha_i / k} Y_i(t),\quad {\bf Y}(t) = (Y_1(t), \ldots, Y_n(t))^T
\end{equation}
with the asymptotic expansion by means of {\em general asymptotic series}\footnote{
\KMb{For two scalar functions $h_1$ and $h_2$, $h_1 \ll h_2$ as $t\to t_{\max}-0$ iff $(h_1(t)/h_2(t))\to 0$ as $t\to t_{\max}-0$. 
For two vector-valued functions ${\bf h}_1$ and ${\bf h}_2$ with \KMc{${\bf h}_i = (h_{1;i}, \ldots, h_{n;i})$}, ${\bf h}_1 \ll {\bf h}_2$ iff $h_{l;1} \ll h_{l;2}$ for each $l = 1,\ldots, n$.}
}
\begin{align}
\notag
{\bf Y}(t) &= {\bf Y}_0 + \tilde {\bf Y}(t),\\
\label{Y-asym}
\tilde {\bf Y}(t) &= \sum_{j=1}^\infty {\bf Y}_j(t),\quad {\bf Y}_{j}(t) \ll {\bf Y}_{j-1}(t)\quad (t\to t_{\max}-0),\quad j=1,2,\ldots,\\
\notag
{\bf Y}_j(t) &= (Y_{j,1}(t), \ldots, Y_{j,n}(t))^T,\quad j=1,2,\ldots
\end{align}
and determine the concrete form of the factor ${\bf Y}(t)$.
As the first step, decompose the vector field $f$ into two terms as follows:
\begin{equation*}
f({\bf y}) = f_{\alpha, k}({\bf y}) + f_{\rm res}({\bf y}),
\end{equation*}
where $f_{\alpha, k}$ is the quasi-homogeneous component of $f$ and $f_{\rm res}$ is the residual (\KMe{i.e.,} lower-order) terms.
The componentwise expressions are 
\begin{equation*}
f_{\alpha, k}({\bf y}) = (f_{1;\alpha, k}({\bf y}), \ldots, f_{n;\alpha, k}({\bf y}))^T,\quad f_{\rm res}({\bf y}) = (f_{1;{\rm res}}({\bf y}), \ldots, f_{n;{\rm res}}({\bf y}))^T\KMe{,}
\end{equation*}
\KMe{respectively.
}
Substituting (\ref{blow-up-sol}) into (\ref{ODE-original}), we derive the system of ${\bf Y}(t)$, which is the following nonautonomous system:
\begin{equation}
\label{blow-up-basic}
\frac{d}{dt}{\bf Y} = \theta(t)^{-1}\left\{ -\KMf{ \frac{1}{k}\Lambda_\alpha }{\bf Y} + f_{\alpha, k}({\bf Y}) \right\} + \theta(t)^{\KMf{ \frac{1}{k}\Lambda_\alpha }} f_{{\rm res}}( \theta(t)^{-\KMf{ \frac{1}{k}\Lambda_\alpha }} {\bf Y}).
\end{equation}
From the asymptotic quasi-homogeneity of $f$, the most singular part of the above system yields the following identity which the leading term ${\bf Y}_0$ of ${\bf Y}(t)$ must satisfy.

\begin{dfn}\rm
\label{dfn-balance}
We call the identity
\begin{equation}
\label{0-balance}
-\frac{1}{k}\Lambda_\alpha {\bf Y}_0 + f_{\alpha, k}({\bf Y}_0) = 0
\end{equation}
{\em a balance law} for the blow-up solution ${\bf y}(t)$.
\end{dfn}
The next step is to derive the collection of systems \KMc{for $\{{\bf Y}_j(t)\}_{j\geq 1}$}. by means of {\em inhomogeneous linear systems}.
The key concept towards our aim is the following algebraic objects.

\begin{dfn}[Blow-up power eigenvalues]\rm
\label{dfn-blow-up-power-ev}
Suppose that a nonzero root ${\bf Y}_0$ of the balance law (\ref{0-balance}) is given.
We call the constant matrix
\begin{equation}
\label{blow-up-power-determining-matrix}
A = -\KMf{ \frac{1}{k}\Lambda_\alpha } + D f_{\alpha, k}({\bf Y}_0)
\end{equation}
the {\em \KMf{blow-up} power-determining matrix} for the blow-up solution ${\bf y}(t)$, and call the eigenvalues $\{\lambda_i\}_{i=1}^n \equiv {\rm Spec}(A)$ the {\em blow-up power eigenvalues}, where eigenvalues with nontrivial multiplicity are distinguished in this expression, except specifically noted.
\end{dfn}

Using the matrix $A$ and the Taylor expansion of the nonlinearity at ${\bf Y}_0$, we obtain the following system:
\begin{align}
\label{asym-eq}
&\frac{d}{dt} \tilde {\bf Y} = \theta(t)^{-1} \left[ A\tilde {\bf Y} + R_{\alpha, k}({\bf Y}) \right]
	+  \theta(t)^{\KMf{ \frac{1}{k}\Lambda_\alpha }}  f_{{\rm res}}\left(\theta(t)^{-\KMf{ \frac{1}{k}\Lambda_\alpha }}  {\bf Y} \right),\\
	\notag
&R_{\alpha, k}({\bf Y}) = f_{\alpha,k}({\bf Y}) - \left\{f_{\alpha,k}({\bf Y}_0) + Df_{\alpha,k}({\bf Y}_0)\tilde {\bf Y}\right\}.
\end{align}
The linear systems solving ${\bf Y}_j(t)$ for $j\geq 1$ are derived inductively from (\ref{asym-eq}), assuming the asymptotic relation (\ref{Y-asym}).
In particular, the algebraic eigenstructure of $A$ essentially determines concrete forms of ${\bf Y}_j(t)$.
\par
As a summary, the following objects play essential \KMb{roles} in determining multi-order asymptotic \KMb{expansions} of \KMb{blow-up solutions} ${\bf y}(t)$ with Assumption \ref{ass-fundamental}:
\begin{itemize}
\item Roots of the balance law (\ref{0-balance}): ${\bf Y}_0$ (not identically zero).
\item The blow-up power-determining matrix $A$ in (\ref{blow-up-power-determining-matrix}) and blow-up power eigenvalues.
\end{itemize}
\KMa{The precise expression of our asymptotic expansions of blow-up solutions is summarized in \cite{asym1}, and we omit the detail because we do not need the concrete expression of these expansions.}
Throughout the rest of this section, we derive the relationship between the above objects and the corresponding ones in desingularized vector fields.

\subsection{Balance law and equilibria on the horizon}

\KMc{The first issue for the correspondence of dynamical structures is \lq\lq equilibria" among two systems.}
Recall that equilibria for the desingularized vector field (\ref{desing-para}) associated with quasi-parabolic compactifications satisfy
\begin{equation}
\label{balance-para}
\KMg{ \left(1-\frac{2c-1}{2c}(1-p({\bf x})^{2c}) \right)\tilde f({\bf x}) = G({\bf x}) \Lambda_\alpha {\bf x}},
\end{equation}
where $G({\bf x})$ is given in (\ref{Gx}).
Equilibria ${\bf x}_\ast = (x_{\ast, 1}, \ldots, x_{\ast, n})^T$ on the horizon $\mathcal{E}$ satisfy $p({\bf x}_\ast) \equiv 1$ and hence the following identity holds:
\begin{equation}
\label{const-horizon-0}
\tilde f_i({\bf x}_\ast) = \alpha_i x_{\ast, i} G({\bf x}_\ast),
\end{equation}
equivalently
\begin{equation}
\label{const-horizon}
\frac{\tilde f_i({\bf x}_\ast)}{\alpha_i x_{\ast, i} } = G({\bf x}_\ast) \equiv C_\ast = C_\ast({\bf x}_\ast)
\end{equation}
provided $x_{\ast, i} \not = 0$.
Because at least one $x_i$ is not $0$ on the horizon, the constant $C_\ast$ is determined as a constant independent of $i$.
On the other hand, only the quasi-homogeneous part $f_{\alpha, k}$ of $f$ \KMi{involves equilibria on the horizon}.
In general, we have
\begin{align}
\notag
\tilde f_{i;\alpha, k}(\KMh{{\bf x}}) &= \kappa^{-(k+\alpha_i)} f_{i;\alpha, k}(\KMh{\kappa^{\Lambda_\alpha}{\bf x}} )\\ 	
\notag
	&= \kappa^{-(k+\alpha_i)}\kappa^{k+\alpha_i} f_{i;\alpha, k}(\KMh{{\bf x}})\\
\label{identity-f-horizon}
	&= f_{i;\alpha, k}(\KMh{{\bf x}})
\end{align}
for ${\bf x}=(x_1,\ldots, x_n)\in \mathcal{E}$.
The identity (\ref{balance-para}) is then rewritten as follows:
\begin{equation}
\label{const-horizon-0-Cast}
\KMh{ f_{\alpha,k}({\bf x}_\ast) = G({\bf x}_\ast) \Lambda_\alpha {\bf x}_\ast  = C_\ast\Lambda_\alpha {\bf x}_\ast }.
\end{equation}
Introducing a scaling parameter $r_{{\bf x}_\ast} (> 0)$, we have
\begin{equation*}
\KMh{ f_{\alpha,k}({\bf x}_\ast) = r_{{\bf x}_\ast}^{-(kI + \Lambda_\alpha)} f_{\alpha,k}(\KMh{r_{{\bf x}_\ast}^{\Lambda_\alpha} {\bf x}_\ast} ).}
\end{equation*}
Substituting this identity into (\ref{balance-para}), we have
\begin{equation*}
r_{{\bf x}_\ast}^{-(k+\alpha_i)} f_{i;\alpha,k}(r_{{\bf x}_\ast}^{\alpha_1}x_{\ast,1}, \ldots, r_{{\bf x}_\ast}^{\alpha_n}x_{\ast,n}) \KMd{= \alpha_i x_{\ast,i} C_\ast}.
\end{equation*}
Here we assume that $r_{{\bf x}_\ast}$ satisfies the following equation:
\begin{equation}
\label{balance-C01}
r_{{\bf x}_\ast}^k C_\ast = \frac{1}{k},
\end{equation}
which implies that $r_{{\bf x}_\ast}$ is uniquely determined once $C_\ast$ is given, {\em provided} $C_\ast > 0$.
The positivity of $C_\ast$ is nontrivial in general, while we have the following result.

\begin{lem}
\label{lem-Cast-nonneg}
Let ${\bf x}_\ast \in \mathcal{E}$ be a hyperbolic equilibrium for $g$ such that the local stable manifold $W^s_{\rm loc}({\bf x}_\ast; g)$ satisfies $W^s_{\rm loc}({\bf x}_\ast; g)\cap \mathcal{D}\not = \emptyset$.
Then $C_\ast \equiv G({\bf x}_\ast) \geq 0$.
\end{lem}
\begin{proof}
Assume that the statement is not true, namely $C_\ast < 0$.
We can choose a solution ${\bf x}(\tau)$ asymptotic to ${\bf x}_{\ast}$ whose initial point ${\bf x}(0)$ satisfies $\kappa({\bf x}(0)) < \infty$ by assumption.
Along such a solution, we integrate $G({\bf x}(\tau))$.
Lemma \ref{lem-G} indicates that
\begin{equation*}
\int_{\tau_0}^\tau G({\bf x}(\eta))d\eta = \ln \kappa({\bf x}(\tau)) - \ln \kappa({\bf x}(\tau_0)).
\end{equation*}
By the continuity of $G$, $G({\bf x}(\tau))$ is always negative along ${\bf x}(\tau)$ in a small neighborhood of ${\bf x}_\ast$ in $W^s_{\rm loc}({\bf x}_\ast; g)$.
On the other hand, ${\bf x}(\tau)\to {\bf x}_\ast\in \mathcal{E}$ holds as $\tau\to +\infty$, implying $\kappa = \kappa({\bf x}(\tau))\to +\infty$. 
The real-valued function $\ln r$ is monotonously increasing in $r$, and hence $\ln \kappa({\bf x}(\tau))$ diverges to $+\infty$ as $\tau\to \infty$, which contradicts the fact that the integral of $G({\bf x}(\tau))$ is negative.
\end{proof}
At this moment, we cannot exclude the possibility that $C_\ast = 0$.
Now we \KMe{{\em assume}} $C_\ast \not = 0$.
Then $C_\ast > 0$ holds and $r_{{\bf x}_\ast}$ in \KMd{(\ref{balance-C01})} is well-defined.
Finally the equation (\ref{balance-para}) is written by 
\begin{align*}
\frac{ \alpha_i}{k} r_{{\bf x}_\ast}^{\alpha_i} \KMc{x_{\ast, i}}  = f_{i;\alpha,k}(r_{{\bf x}_\ast}^{\alpha_1}\KMc{x_{\ast, 1}}, \ldots, r_{{\bf x}_\ast}^{\alpha_n} \KMc{x_{\ast, n}}),
\end{align*}
which is nothing but the balance law (\ref{0-balance}).
As a summary, we have the one-to-one correspondence among roots of the \KMf{balance law} and equilibria on the horizon \KMe{for the desingularized vector field (\ref{desing-para})}.

\begin{thm}[One-to-one correspondence of the balance]
\label{thm-balance-1to1}
Let ${\bf x}_\ast = (x_{\ast,1}, \ldots, x_{\ast,n})^T$ be an equilibrium on the horizon for the desingularized vector field (\ref{desing-para}).
Assume that $C_\ast$ in (\ref{const-horizon}) is positive so that $r_{{\bf x}_\ast} = (kC_\ast)^{-1/k}\KMe{>0}$ is well-defined.
Then \KMi{the vector} ${\bf Y}_0 = (Y_{0,1},\ldots, Y_{0,n})^T$ \KMi{given by}
\begin{equation}
\label{x-to-C}
(Y_{0,1},\ldots, Y_{0,n}) = (r_{{\bf x}_\ast}^{\alpha_1}x_{\ast,1},\ldots, r_{{\bf x}_\ast}^{\alpha_n}x_{\ast,n}) \equiv r_{{\bf x}_\ast}^{\Lambda_\alpha}{\bf x}_\ast 
\end{equation}
\KMi{is a root of the balance law (\ref{0-balance}).}
Conversely, let ${\bf Y}_0\not = 0$ be a root of the balance law (\ref{0-balance}). 
Then \KMi{the vector} ${\bf x}_\ast$ \KMi{given by}
\begin{equation}
\label{C-to-x}
(x_{\ast,1},\ldots, x_{\ast,n}) = (r_{{\bf Y}_0}^{-\alpha_1}Y_{0,1},\ldots, r_{{\bf Y}_0}^{-\alpha_n}Y_{0,n})\equiv r_{{\bf Y}_0}^{-\Lambda_\alpha}{\bf Y}_0
\end{equation}
\KMi{is an equilibrium on the horizon for (\ref{desing-para}),}
where $r_{{\bf Y}_0} = p({\bf Y}_0) > 0$.
\end{thm}

\begin{proof}
\KMf{We have already seen the proof of the first statement, and hence we shall prove the second statement here}.
First let
\begin{equation*}
r_{{\bf Y}_0} \equiv p({\bf Y}_0) > 0,\quad \bar Y_{0,i} := \frac{Y_{0,i}}{r_{{\bf Y}_0}^{\alpha_i}}.
\end{equation*}
By definition $p(\bar {\bf Y}_0) = 1$, where $\bar {\bf Y}_0 = (\bar Y_{0,1},\ldots, \bar Y_{0,n})^T$.
Substituting $\bar {\bf Y}_0$ into the right-hand side of (\ref{balance-para}), we have
\begin{align*}
\alpha_i \bar Y_{0,i} \sum_{j=1}^n \frac{\bar Y_{0,j}^{2\beta_j-1}}{\alpha_j}\tilde f_j(\bar {\bf Y}_0) = \alpha_i \bar Y_{0,i} \sum_{j=1}^n \frac{\bar Y_{0,j}^{2\beta_j-1}}{\alpha_j} \tilde f_{j; \alpha, k}(\bar {\bf Y}_0),
\end{align*}
where we have used the identity $p(\bar {\bf Y}_0) = 1$ and (\ref{identity-f-horizon}).
From quasi-homogeneity of $f_{\alpha, k}$ and the balance law (\ref{0-balance}), we further have
\begin{align*}
 \alpha_i \bar Y_{0,i} \sum_{j=1}^n \frac{\bar Y_{0,j}^{2\beta_j-1}}{\alpha_j} \tilde f_{j; \alpha, k}(\bar {\bf Y}_0) 
&= \alpha_i \bar Y_{0,i} \sum_{j=1}^n \frac{\bar Y_{0,j}^{2\beta_j-1}}{\alpha_j} r_{{\bf Y}_0}^{-(k+\alpha_j)}f_{j; \alpha, k}({\bf Y}_0) \\
&= \alpha_i \bar Y_{0,i} \sum_{j=1}^n \frac{\bar Y_{0,j}^{2\beta_j-1}}{\alpha_j} r_{{\bf Y}_0}^{-(k+\alpha_j)} \frac{\alpha_j}{k}Y_{0,j} 
 = \alpha_i \bar Y_{0,i} \frac{r_{{\bf Y}_0}^{-k}}{k}\sum_{j=1}^n \bar Y_{0,j}^{2\beta_j}\\
 &= \alpha_i \bar Y_{0,i} \frac{r_{{\bf Y}_0}^{-k}}{k} = r_{{\bf Y}_0}^{-(k+\alpha_i)} \frac{\alpha_i}{k} Y_{0,i}\\
&= r_{{\bf Y}_0}^{-(k+\alpha_i)} f_{i; \alpha, k}({\bf Y}_0) = f_{i; \alpha, k}(\bar {\bf Y}_0) = \tilde f_{i; \alpha, k}(\bar {\bf Y}_0),
\end{align*}
implying that $\bar {\bf Y}_0$ is a root of (\ref{balance-para}).
\end{proof}

\subsection{A special eigenstructure characterizing blow-ups}
The balance law determines the coefficients of type-I blow-ups, which \KMi{turn out to} correspond to equilibria on the horizon for the desingularized vector field.
This correspondence provides a relationship among two different vector fields involving blow-ups.
We further investigate a common feature characterizing blow-up behavior extracted by blow-up power-determining matrices, as well as desingularized vector fields.


\begin{thm}[Eigenvalue $1$. cf. \cite{AM1989, C2015}]
\label{thm-ev1}
\KMe{Consider \KMe{an} asymptotically quasi-homogeneous vector field $f$ of type $\alpha = (\alpha_1,\ldots, \alpha_n)$ and order $k+1$}.
Suppose that a nontrivial root ${\bf Y}_0$ of the balance law (\ref{0-balance}) is given.
Then the corresponding blow-up power-determining matrix $A$ has an eigenvalue $1$ with the associating eigenvector
\begin{equation}
\label{vector-ev1}
{\bf v}_{0,\alpha} = \KMf{\Lambda_\alpha {\bf Y}_0}.
\end{equation} 
\end{thm}

Note that the matrix $A$ only involves the quasi-homogeneous part $f_{\alpha, k}$ of $f$.

\begin{proof}
Consider \eqref{temporay-label4} at $\mathbf{y} = \mathbf{Y}_0$
with the help of (\ref{0-balance}):
\begin{equation}
(D f_{\alpha,k})(\mathbf{Y}_0)  \KMg{ \Lambda_\alpha } \mathbf{Y}_0 
= \left(  \KMg{ k I+ \Lambda_\alpha } \right) f_{\alpha,k}(\mathbf{Y}_0)
= \left( \KMg{  I + \frac{1}{k}\Lambda_\alpha } \right) \KMg{ \Lambda_\alpha } \mathbf{Y}_0.
\end{equation}
Then, using the definition of $A$, we have
\begin{eqnarray*}
A \KMg{ \Lambda_\alpha } \mathbf{Y}_0
= \left( -\KMf{ \frac{1}{k}\Lambda_\alpha } + Df_{\alpha,k}(\mathbf{Y}_0) \right) \KMg{ \Lambda_\alpha } \mathbf{Y}_0
= -\frac{1}{k}\Lambda_\alpha^2 \mathbf{Y}_0 + \left(I+\KMf{ \frac{1}{k}\Lambda_\alpha } \right)\KMg{  \Lambda_\alpha } \mathbf{Y}_0
=\KMg{ \Lambda_\alpha } \mathbf{Y}_0,
\end{eqnarray*}
which shows the desired statement\footnote{
\KMb{It follows from Theorem \ref{thm-blow-up-estr} below that $\Lambda_\alpha \mathbf{Y}_0$ is not a zero-vector}.
}.
\end{proof}
As a corollary, we have the following observations from the balance law, which extracts a common feature of blow-up solutions.

\begin{cor}
Under the same assumptions in Theorem \ref{thm-ev1}, 
the corresponding blow-up power-determining matrix $A$ has an eigenvalue $1$ with the associating eigenvector $f_{\alpha, k}({\bf Y}_0)$.
\end{cor}

Combining with Theorem \ref{thm-balance-1to1}, the eigenvector is also characterized as follows.
\begin{cor}
Suppose that all assumptions in Theorem \ref{thm-ev1} are satisfied.
Then the blow-up power-determining matrix $A$ associated with the blow-up solution given by the balance law (\ref{0-balance}) has an eigenvalue $1$ with the associating eigenvector $\KMc{r_{{\bf x}_\ast}^{\Lambda_\alpha}} f_{\alpha, k}({\bf x}_\ast)$, where ${\bf x}_\ast$ is the equilibrium on the horizon for the desingularized vector field, under the quasi-parabolic compactification of type $\alpha$, given by the formula (\ref{C-to-x}).
\end{cor}

The above arguments indicate one specific eigenstructure of the desingularized vector field at equilibria on the horizon {\em under a technical assumption}.
\KMe{To see this, we make the following} assumption to $f$, which is essential to the following arguments coming from the technical restriction due to the form of parabolic compactifications, while it can be relaxed for general systems.
\begin{ass}
\label{ass-f}
For each $i$,
\begin{equation*}
\tilde f_{i;{\rm res}}({\bf x}) = O\left(\kappa({\bf x})^{-(1+\epsilon)} \right),\quad \frac{\partial \tilde f_{i;{\rm res}}}{\partial x_l}({\bf x}) = o\left(\kappa({\bf x})^{-(1+\epsilon)}\right),\quad l=1,\ldots, n
\end{equation*}
hold for some $\epsilon > 0$ as ${\bf x}$ approaches to $\mathcal{E}$.
\KMf{Moreover, for any equilibrium ${\bf x}_\ast$ for $g$ \KMi{under consideration}, $C_\ast > 0$ holds, where $C_\ast = C_\ast({\bf x}_\ast)$ is given in (\ref{const-horizon}).}
\end{ass}

The direct consequence of the assumption is the following, which is used to the correspondence of eigenstructures among different matrices.
\begin{lem}
\label{lem-ass-f}
Let ${\bf x}_\ast\in \mathcal{E}$ be an equilibrium for $g$.
Under Assumption \ref{ass-f}, we have $D\tilde f({\bf x}_\ast) = D\tilde f_{\alpha, k}({\bf x}_\ast)$, where the derivative $D$ is with respect to ${\bf x}$.
\end{lem}

\begin{proof}
Now $\tilde f_{i;{\rm res}}$ is expressed as
\begin{align*}
\tilde f_{i;{\rm res}}(\KMh{{\bf x}}) &\equiv \kappa^{-(k+\alpha_i)} f_{i,{\rm res}}(\KMh{\kappa^{\Lambda_\alpha}{\bf x}} )\quad \text{(by (\ref{f-tilde}))} \\
	&\equiv \kappa^{-(1+\epsilon)} \tilde f_{i;{\rm res}}^{(1)} (\KMh{{\bf x}})
\end{align*}
with
\begin{equation*}
\tilde f_{i;{\rm res}}^{(1)}({\bf x}) = O(1),\quad \frac{\partial \tilde f_{i;{\rm res}}^{(1)}}{\partial x_l}({\bf x}) = o(\kappa({\bf x})^{1+\epsilon}), \quad l=1,\ldots, n
\end{equation*}
as ${\bf x}$ approaches to $\mathcal{E}$\KMc{.}
The partial derivative of the component $\tilde f_i$ with respect to $x_l$ at ${\bf x}_\ast$ is
\begin{equation*}
\frac{\partial \tilde f_i}{\partial x_l}({\bf x}_\ast) = \frac{\partial \tilde f_{i;\alpha, k}}{\partial x_l}({\bf x}_\ast) + 
\KMg{
(1+\epsilon)\kappa^{-\epsilon}\frac{\partial \kappa^{-1}}{\partial x_l}\tilde f_{i;{\rm res}}^{(1)}(\KMh{{\bf x}_\ast})  + \kappa^{-(1+\epsilon)} \frac{\partial \tilde f_{i;{\rm res}}^{(1)}}{\partial x_l}({\bf x}_\ast ).
}
\end{equation*}
Using the fact that $\kappa^{-1} = 0$ on the horizon, 
our present assumption implies that the \lq\lq gap" terms
\begin{equation*}
(1+\epsilon)\kappa^{-\epsilon}\frac{\partial \kappa^{-1}}{\partial x_l}\tilde f_{i;{\rm res}}^{(1)}(\KMh{{\bf x}_\ast})  + \kappa^{-(1+\epsilon)} \tilde f_{i;{\rm res}}^{(1)}({\bf x}_\ast )
\end{equation*}
are identically $0$ on the horizon and hence the Jacobian matrix $D\tilde f({\bf x}_\ast)$ with respect to ${\bf x}$ coincides with $D\tilde f_{\alpha, k}({\bf x}_\ast)$.
\end{proof}

\begin{thm}
\label{thm-ev-special}
\KMe{Suppose that Assumption \ref{ass-f} holds}.
Also suppose that ${\bf x}_\ast$ is an equilibrium on the horizon for the associated desingularized vector field (\ref{desing-para}).
Then the Jacobian matrix $Dg({\bf x}_\ast)$ always \KMc{possesses} the eigenpair $\{-C_\ast, {\bf v}_{\ast,\alpha}\}$, where
\begin{equation*}
{\bf v}_{\ast,\alpha} = \KMf{ \Lambda_\alpha {\bf x}_\ast }.
\end{equation*}
\end{thm}
\KMi{
Before the proof, it should be noted that
the inner product of the gradient $\nabla p({\bf x}_\ast)$ at an equilibrium ${\bf x}_\ast \in \mathcal{E}$ given in (\ref{grad-xast}) and the vector ${\bf v}_{\ast,\alpha}$ is unity:
\begin{equation}
\label{inner-gradp-v}
\nabla p({\bf x}_\ast)^T {\bf v}_{\ast, \alpha} = \frac{1}{c}\sum_{l=1}^n \beta_lx_{\ast, l}^{2\beta_l-1} \alpha_l x_{\ast, l} = \frac{c}{c} \sum_{l=1}^n x_{\ast, l}^{2\beta_l} = 1.
\end{equation}
}

\begin{proof}[Proof of Theorem \ref{thm-ev-special}]
First, it follows from (\ref{desing-vector}) that
\begin{align*}
Dg({\bf x}) &= (2c-1)p({\bf x})^{2c-1}\tilde f({\bf x})\nabla p({\bf x})^T +  \left( 1-\frac{2c-1}{2c}(1-p({\bf x})^{2c})\right)  D\tilde f({\bf x})\\
	&\quad  - (\alpha_1 x_1, \ldots, \alpha_n x_n)^T \nabla G({\bf x})^T - G({\bf x})\Lambda_\alpha, \\
Dg({\bf x}_\ast) &= (2c-1)\tilde f({\bf x}_\ast)\nabla p({\bf x}_\ast)^T + D\tilde f({\bf x}_\ast) - {\bf v}_{\ast, \alpha}\nabla G({\bf x}_\ast)^T - G({\bf x}_\ast)\Lambda_\alpha \quad \text{(from the definition of ${\bf v}_{\ast, \alpha}$)}\\
	&= (2c-1)\tilde f({\bf x}_\ast)\nabla p({\bf x}_\ast)^T + D\tilde f({\bf x}_\ast) - {\bf v}_{\ast, \alpha}\nabla G({\bf x}_\ast)^T - C_\ast \Lambda_\alpha \quad \KMi{\text{(from (\ref{const-horizon}))}}\\
	&= \left\{ - C_\ast \Lambda_\alpha + D\tilde f({\bf x}_\ast) \right\} + {\bf v}_{\ast, \alpha} \left( (2c-1) C_\ast \nabla p({\bf x}_\ast) - \nabla G({\bf x}_\ast) \right)^T \quad \text{(from (\ref{const-horizon-0}))}.
\end{align*}
Next, using (\ref{Gx}) and (\ref{const-horizon}), we have
\begin{align*}
\nabla G({\bf x}_\ast) &= {\rm diag}\left(\frac{2\beta_1-1}{\alpha_1}x_{\ast, 1}^{2\beta_1-2},\ldots, \frac{2\beta_n-1}{\alpha_n}x_{\ast, n}^{2\beta_n-2}\right)C_\ast {\bf v}_{\ast, \alpha} 
	+ (A_g + C_\ast \Lambda_\alpha)^T \left(\frac{x_{\ast, 1}^{2\beta_1-1}}{\alpha_1},\ldots, \frac{x_{\ast, n}^{2\beta_n-1}}{\alpha_n}\right)^T \\
	&= C_\ast \left(2\beta_1 x_{\ast, 1}^{2\beta_1-1},\ldots, 2\beta_n x_{\ast, n}^{2\beta_n-1}\right)^T
	+ A_g^T  \nabla p({\bf x}_\ast) \\
	&= 2cC_\ast \nabla p({\bf x}_\ast) + A_g^T  \nabla p({\bf x}_\ast),
\end{align*}
where
\begin{align*}
A_g &:= -C_\ast \Lambda_\alpha + D\tilde f({\bf x}_\ast).
\end{align*}
The Jacobian matrix $Dg({\bf x}_\ast)$ \KMh{then} has a decomposition $Dg({\bf x}_\ast) = A_g + B_g$, where 
\begin{align}
\label{Bg}
B_g &:= - {\bf v}_{\ast, \alpha} \nabla p({\bf x}_\ast)^T (A_g + C_\ast I).
\end{align}
Now Lemma \ref{lem-ass-f} implies
\begin{equation}
\label{Ag-QH}
A_g = -C_\ast \Lambda_\alpha + D\tilde f_{\alpha, k}({\bf x}_\ast).
\end{equation}
Therefore the same idea as the proof of Theorem \ref{thm-ev1} can be applied to obtaining
\begin{align*}
A_g\KMg{ \Lambda_\alpha } {\bf x}_\ast
&= (- \KMg{ C_\ast\Lambda_\alpha } + D\tilde f_{\alpha,k}({\bf x}_\ast))\KMg{ \Lambda_\alpha } {\bf x}_\ast \quad \KMh{\text{(from (\ref{Ag-QH}))}}\\
&= - \KMg{ C_\ast\Lambda_\alpha^2 } {\bf x}_\ast + (k I+ \Lambda_\alpha ) \tilde f_{\alpha, k}({\bf x}_\ast) \quad \KMi{\text{(from (\ref{temporay-label4}))}}\\ 
&= - \KMg{ C_\ast\Lambda_\alpha^2 } {\bf x}_\ast + C_\ast(k I+ \Lambda_\alpha )\KMg{ \Lambda_\alpha } {\bf x}_\ast  \quad \KMi{\text{(from (\ref{const-horizon-0-Cast}))}}\\ 
&= k\KMg{ C_\ast \Lambda_\alpha } {\bf x}_\ast,
\end{align*}
which shows that the matrix $A_g$ admits an eigenvector ${\bf v}_{\ast,\alpha}$ with associated eigenvalue $kC_\ast$.
In other words,
\begin{equation}
\label{vector-evkC}
A_g {\bf v}_{\ast,\alpha} = kC_\ast {\bf v}_{\ast,\alpha}.
\end{equation}
From (\ref{Bg}) and (\ref{inner-gradp-v}), we have
\begin{align*}
B_g {\bf v}_{\ast,\alpha}  &= - {\bf v}_{\ast, \alpha} \nabla p({\bf x}_\ast)^T (A_g + C_\ast I){\bf v}_{\ast,\alpha}\\
	&= - (k+1)C_\ast {\bf v}_{\ast, \alpha} \nabla p({\bf x}_\ast)^T {\bf v}_{\ast,\alpha}\quad \text{(from (\ref{vector-evkC}))}\\
	 &= - (k+1)C_\ast {\bf v}_{\ast, \alpha}\KMh{.} \quad \KMf{ \text{(from (\ref{inner-gradp-v}))} }
\end{align*}
Therefore we have
\begin{align*}
Dg({\bf x}_\ast) {\bf v}_{\ast,\alpha} = (A_g + B_g) {\bf v}_{\ast,\alpha} = \{ kC_\ast - (k+1)C_\ast\} {\bf v}_{\ast,\alpha} &= -C_\ast {\bf v}_{\ast,\alpha}
\end{align*}
and, as a consequence, the vector ${\bf v}_{\ast,\alpha}$ is an eigenvector of $Dg({\bf x}_\ast)$ associated with $-C_\ast$ and the proof is completed.
\end{proof}

This theorem and (\ref{inner-gradp-v}) imply that the eigenvector ${\bf v}_{\ast,\alpha}$ is transversal to the tangent space $T_{{\bf x}_\ast}\mathcal{E}$.
Combined with the $Dg$-invariance of the tangent bundle $T\mathcal{E}$ (cf. Remark \ref{rem-invariance}), we conclude that the eigenvector ${\bf v}_{\ast,\alpha}$ provides the blow-up direction in the linear sense.
Comparing Theorem \ref{thm-ev1} with Theorem \ref{thm-ev-special}, the eigenpair $\{1, {\bf v}_{0,\alpha}\}$ of the blow-up power-determining matrix $A$ provides a characteristic information of blow-up solutions.
\KMf{Similarly}, from the eigenpair $\{\KMf{-C_\ast}, {\bf v}_{\ast,\alpha}\}$, a direction of trajectories ${\bf x}(\tau)$ for (\ref{desing-para}) converging to ${\bf x}_\ast$ is uniquely determined.
Using this fact, we obtain the following corollary, which justifies the correspondence stated in Theorem \ref{thm-balance-1to1}.

\begin{cor}
\label{cor-Cast-pos}
\KMe{Suppose that Assumption \ref{ass-f} except the condition of $C_\ast$ holds}.
Let ${\bf x}_\ast\in \mathcal{E}$ be a hyperbolic equilibrium for $g$ satisfying assumptions in Lemma \ref{lem-Cast-nonneg}.
Then the constant $C_\ast = C_\ast({\bf x}_\ast)$ given in (\ref{const-horizon}) is positive.
\end{cor}

\begin{proof}
\KMe{Because} ${\bf x}_\ast$ is hyperbolic, then ${\rm Spec}(Dg({\bf x}_\ast))\cap i\mathbb{R} = \emptyset$.
In particular, $C_\ast \not = 0$ since $-C_\ast \in {\rm Spec}(Dg({\bf x}_\ast))$.
Combining this fact with Lemma \ref{lem-Cast-nonneg}, we have $C_\ast > 0$.
\end{proof}

\begin{rem}
The above corollary provides a sufficient condition so that $C_\ast > 0$ is satisfied, while the converse is not always true.
In other words, the nontrivial intersection ${\rm Spec}(Dg({\bf x}_\ast))\cap i\mathbb{R}$ can exist even if $C_\ast > 0$.
\end{rem}

\begin{rem}[Similarity to Painlev\'{e}-type analysis]
\label{rem-Painleve}
In studies involving Painlev\'{e}-type property from the viewpoint of algebraic geometry (e.g. \cite{AM1989, C2015}), the necessary and sufficient conditions which the {\em complex} ODE of the form 
\begin{equation}
\label{ODE-complex}
\frac{du}{dz} = f(z,u),\quad z\in \mathbb{C}
\end{equation}
possesses meromorphic solutions are considered.
The key point is that the matrix induced by (\ref{ODE-complex}) called {\em Kovalevskaya matrix}, essentially the same as the blow-up power-determining matrix, is diagonalizable and all eigenvalues, which are referred to as {\em Kovalevskaya exponents} in \cite{C2015, C2016_124, C2016_356}, are integers, in which case the meromorphic solutions generate a parameter family.
The number of free parameters is determined by that of integer eigenvalues with required sign.
\par
It is proved in the preceding studies that the Kovalevskaya matrix always admits the eigenvalue $-1$ and the associated eigenstructure is uniquely determined.
Theorem \ref{thm-ev1} is therefore regarded as a counterpart of the result to blow-up description.
Difference of the sign \lq\lq $-1$" from our result \lq\lq $+1$" comes from the different form of expansions of solutions.
Indeed, solutions as functions of $z-z_0$ with a movable singularity $z_0\in \mathbb{C}$ are considered in the Painlev\'{e}-type analysis, while solutions as functions of $\theta(t) =  t_{\max} - t$ are considered in the present study.
\end{rem}

\subsection{Remaining eigenstructure of $A$ and tangent spaces on the horizon}

As mentioned \KMi{before} the proof of Theorem \ref{thm-ev-special}, the vector ${\bf v}_{\ast, \alpha}$ is transversal \KMe{to} the horizon $\mathcal{E}$.
\KMe{Also, as} mentioned in Remark \ref{rem-invariance}, the horizon $\mathcal{E}$ is \KMe{a} codimension one invariant manifold for $g$ and hence the remaining $n-1$ independent (generalized) eigenvectors of $Dg({\bf x}_\ast)$ \KMe{span} the tangent space $T_{{\bf x}_\ast}\mathcal{E}$.
Our aim here is to investigate these eigenvectors and the correspondence among those for $Dg({\bf x}_\ast)$, $A_g$ and $A$.
We shall see below that the matrix  $B_g$ \KMe{given} in (\ref{Bg}) plays a key role whose essence is the determination of the following \KMb{object}.

\begin{prop}
\label{prop-proj-B}
Let $P_\ast := {\bf v}_{\ast,\alpha}\nabla p({\bf x}_\ast)^T$. 
Then $P_\ast$ as the linear mapping on $\mathbb{R}^n$ is the (nonorthogonal) projection\footnote{
\KMe{In the homogeneous case $\alpha = (1,\ldots, 1)$, this is orthogonal.}
} onto ${\rm span}\{{\bf v}_{\ast,\alpha}\}$.
Similarly, the map $I-P_\ast$ is the (nonorthogonal) projection onto the tangent space $T_{{\bf x}_\ast}\mathcal{E}$. 
\end{prop}
\begin{proof}
Note that the tangent space $T_{{\bf x}_\ast}\mathcal{E}$ is the orthogonal complement of the gradient $\nabla p({\bf x}_\ast)$.
We have
\begin{align*}
\nabla p({\bf x}_\ast)^T \KMg{(I-P_\ast)}
	&=  \nabla p({\bf x}_\ast)^T -  \nabla p({\bf x}_\ast)^T {\bf v}_{\ast,\alpha}\nabla p({\bf x}_\ast)^T \\
	&=  \nabla p({\bf x}_\ast)^T - (\nabla p({\bf x}_\ast)^T {\bf v}_{\ast,\alpha} )\nabla p({\bf x}_\ast)^T\\
	&=  \nabla p({\bf x}_\ast)^T - \nabla p({\bf x}_\ast)^T\\
	&= 0,
\end{align*}
which yields \KMg{that $(\KMg{I} - P_\ast)$ maps $\mathbb{R}^n$ to $({\rm span}\{\nabla p({\bf x}_\ast)\})^\bot = T_{{\bf x}_\ast}\mathcal{E}$.}
Moreover, the above indentity implies
\begin{equation*}
P_\ast ( \KMg{I} -P_\ast) = P_\ast - P_\ast^2 = 0,
\end{equation*}
that is, $P_\ast$ is idempotent and hence is a projection.
\end{proof}
Using the projection $P_\ast$, the matrix $B_g$, and hence $Dg({\bf x}_\ast)$ is rewritten as follows:
\begin{align}
\label{Bg-easy}
\KMg{B_g} &= \KMg{ -P_\ast (A_g + C_\ast \KMg{I})},\\
\label{Dg-proj}
Dg({\bf x}_\ast) &= A_g + B_g
	= A_g - P_\ast (A_g + C_\ast \KMg{I}) 
	= (\KMg{I} - P_\ast)A_g - C_\ast P_\ast.
\end{align}

Using this expression, we obtain the following proposition, which plays a key role in characterizing the correspondence of eigenstructures among different matrices.

\begin{prop}
\label{prop-corr-Dg-Ag}
For any $\lambda \in \mathbb{C}$ and $N\in \KMi{\mathbb{N}}$, we have
\begin{equation}
\label{Dg-I-P}
(Dg({\bf x}_\ast) - \lambda I)^N (I -P_\ast) = (I -P_\ast) (A_g - \lambda I)^N
\end{equation}
\KMi
{
and
\begin{equation}
\label{Dg-I-P-2}
(A_g - k C_\ast I) ( Dg({\bf x}_\ast) - \lambda I)^N (I-P_\ast) =  (A_g - k C_\ast I) (A_g- \lambda I)^N = (A_g- \lambda I)^N (A_g - k C_\ast I).
\end{equation}
}
\end{prop}

\begin{proof}
\KMf{Because} $P_\ast$ is a projection, we have
\begin{align}
\notag
Dg({\bf x}_\ast)( \KMg{I} -P_\ast) &= (\KMg{I} - P_\ast)A_g ( \KMg{I} - P_\ast) - C_\ast P_\ast ( \KMg{I} - P_\ast)\\
\label{Dg-I-P-1}
	&= (\KMg{I} - P_\ast)A_g - (\KMg{I} - P_\ast)A_g  P_\ast.
\end{align}
Moreover, multiplying the identity (\ref{vector-evkC}) by $\nabla p({\bf x}_\ast)^T$ from the right, we have
\begin{equation}
\label{A-kC-ortho-P}
\KMi{(A_g - kC_\ast I) P_\ast = 0}.
\end{equation}
Then
\begin{equation*}
(\KMg{I} - P_\ast)A_g  P_\ast = (I-P_\ast) k C_\ast P_\ast = 0.
\end{equation*}
Therefore, it follows from (\ref{Dg-I-P-1}) that
\begin{equation*}
Dg({\bf x}_\ast)( \KMg{I} -P_\ast) = (\KMg{I} - P_\ast)A_g.
\end{equation*}
Then, for any {\em complex} number $\lambda$ and for any $N\in \KMi{\mathbb{N}}$, we have the first statement.
\par
As for the second statement, direct calculations yield
\KMi{
\begin{align*}
(A_g - k C_\ast I) ( Dg({\bf x}_\ast) - \lambda I)^N (I-P_\ast)  
&= (A_g - k C_\ast I) (I-P_\ast) (A_g- \lambda I)^N \quad \text{(from (\ref{Dg-I-P}))}\\
&= (A_g - k C_\ast I) (A_g- \lambda I)^N\quad \text{(from (\ref{A-kC-ortho-P}))} \\
&= (A_g- \lambda I)^N (A_g - k C_\ast I).
\end{align*}
}
\end{proof}

The formula (\ref{Dg-I-P}) yields the correspondence of eigenstructures among $Dg({\bf x}_\ast)$ and $A_g$ in a simple way.

\begin{thm}
\label{thm-evec-Dg}
Let ${\bf x}_\ast\in \mathcal{E}$ be an equilibrium on the horizon for $g$ and suppose that Assumption \ref{ass-f} holds.
\begin{enumerate}
\item Assume that $\lambda \in {\rm Spec}(A_g)$ and let ${\bf w}\in \mathbb{C}^n$ be such that
${\bf w}\in \ker((A_g - \lambda I)^{m_\lambda})\setminus \ker((A_g - \lambda I)^{m_\lambda -1})$ with $(I-P_\ast){\bf w}\not = 0$ for some $m_\lambda \in \KMi{\mathbb{N}}$. 
\begin{itemize}
\item \KMi{If $\lambda \not = kC_\ast$,} then $(I -P_\ast){\bf w} \in \ker((Dg({\bf x}_\ast) - \lambda I)^{m_\lambda})\setminus \ker((Dg({\bf x}_\ast) - \lambda I)^{m_\lambda -1})$.
\item \KMi{If $\lambda = kC_\ast$, then either $(I-P_\ast){\bf w}\in \ker((Dg({\bf x}_\ast) - kC_\ast I)^{m_\lambda})\setminus \ker((Dg({\bf x}_\ast) - kC_\ast I)^{m_\lambda -1})$ 
or $(I - P_\ast){\bf w}\in \ker((Dg({\bf x}_\ast) - kC_\ast I)^{m_\lambda-1})\setminus \ker((Dg({\bf x}_\ast) - kC_\ast I)^{m_\lambda -2})$ holds.} 
\end{itemize}
\item 
Conversely, assume that $\lambda_g \in {\rm Spec}(Dg({\bf x}_\ast))$ and let ${\bf w}_g\in \mathbb{C}^n$ be such that
$(I - P_\ast){\bf w}_g\in \ker((Dg({\bf x}_\ast) - \lambda_g I)^{m_{\lambda_g}})\setminus \ker((Dg({\bf x}_\ast) - \lambda_g I)^{m_{\lambda_g} -1})$ with $(I-P_\ast){\bf w}_g\not = 0$ for some $m_{\lambda_g} \in \KMi{\mathbb{N}}$. 
\begin{itemize}
\item If $\lambda_g \not = kC_\ast$, then $(A_g - kC_\ast I){\bf w}_g \in \ker((A_g - \lambda_g I)^{m_{\lambda_g}})\setminus \ker((A_g - \lambda_g I)^{m_{\lambda_g} -1})$.
\item If $\lambda_g = kC_\ast$, \KMi{then either $(A_g - kC_\ast I){\bf w}_g\in \ker((A_g - kC_\ast I)^{m_{\lambda_g}})\setminus \ker((A_g - kC_\ast I)^{m_{\lambda_g} -1})$ 
or $(A_g - kC_\ast I){\bf w}_g\in \ker((A_g - kC_\ast I)^{m_{\lambda_g}+1})\setminus \ker((A_g- kC_\ast I)^{m_{\lambda_g}})$ holds.}
\end{itemize}
\end{enumerate}
\end{thm}

\begin{rem}
\label{rem-complex-evec}
If $\lambda \in {\rm Spec}(Dg({\bf x}_\ast)) \cap (\mathbb{C} \setminus \mathbb{R})$, the associated eigenvector ${\bf w}_g$ is also complex-valued.
Moreover, $(\bar \lambda, \overline{{\bf w}_g})$ is also an eigenpair of $Dg({\bf x}_\ast)$, which implies
\begin{equation*}
Dg({\bf x}_\ast) \begin{pmatrix}
{\bf w}_g & \overline{{\bf w}_g}
\end{pmatrix} = \begin{pmatrix}
{\bf w}_g & \overline{{\bf w}_g}
\end{pmatrix}\begin{pmatrix}
\lambda & 0 \\
0 & \bar \lambda
\end{pmatrix}.
\end{equation*}
Let $\lambda = \lambda_{\rm re} + i\lambda_{\rm im}$ with $\lambda_{\rm im} \not = 0$ and 
\begin{equation*}
Q = \begin{pmatrix}
1 & 1 \\
i & -i
\end{pmatrix}\quad \Leftrightarrow \quad Q^{-1} = \frac{1}{2}\begin{pmatrix}
1 & -i \\
1 & i
\end{pmatrix}.
\end{equation*}
Then we have
\begin{equation*}
Q\begin{pmatrix}
\lambda & 0 \\
0 & \bar \lambda
\end{pmatrix} = \begin{pmatrix}
\lambda_{\rm re} & \lambda_{\rm im} \\
-\lambda_{\rm im} & \lambda_{\rm re}
\end{pmatrix}Q
\end{equation*}
and hence
\begin{equation*}
Dg({\bf x}_\ast) \begin{pmatrix}
{\bf w}_g & \overline{{\bf w}_g}
\end{pmatrix}Q^{-1} = \begin{pmatrix}
{\bf w}_g & \overline{{\bf w}_g}
\end{pmatrix} Q^{-1} \begin{pmatrix}
\lambda_{\rm re} & \lambda_{\rm im} \\
-\lambda_{\rm im} & \lambda_{\rm re}
\end{pmatrix}, 
\end{equation*}
equivalently
\begin{equation*}
Dg({\bf x}_\ast) \begin{pmatrix}
{\rm Re}\,{\bf w}_g & {\rm Im}\,{\bf w}_g
\end{pmatrix} = \begin{pmatrix}
{\rm Re}\,{\bf w}_g & {\rm Im}\,{\bf w}_g
\end{pmatrix}  \begin{pmatrix}
\lambda_{\rm re} & \lambda_{\rm im} \\
-\lambda_{\rm im} & \lambda_{\rm re}
\end{pmatrix}.
\end{equation*}
Therefore ${\rm Re}\,{\bf w}_g$ and ${\rm Im}\,{\bf w}_g$ generate base vectors of \KMc{the} invariant subspace \KMc{$T_{{\bf x}_\ast} \mathcal{E}$.} 
\KMc{Indeed,} $-C_\ast$ is real and hence ${\bf v}_{\ast, \alpha}$ and ${\rm Re}\,{\bf w}_g$, ${\rm Im}\,{\bf w}_g$ are linearly independent.
As a consequence, a complex eigenvalue $\lambda \in {\rm Spec}(Dg({\bf x}_\ast))$ associates two independent vectors ${\bf w}_{gr}, {\bf w}_{gi} \in T_{{\bf x}_\ast} \mathcal{E}$ such that ${\bf w}_{gr} + i{\bf w}_{gi}$ is the eigenvector associated with $\lambda$, in which case all arguments in the proof are applied to ${\bf w}_{gr} + i{\bf w}_{gi}$.
\KMi{The similar observation holds for generalized eigenvectors with appropriate matrices realizing the above real form.} 
\end{rem}

\begin{proof}
First it follows from (\ref{Dg-I-P}) that, for any $\lambda \in \mathbb{C}$, any ${\bf w}\in \mathbb{C}^n$ and $N\in \KMi{\mathbb{N}}$, 
\begin{equation}
\label{Dg-I-PN}
(Dg({\bf x}_\ast) - \lambda I)^N (I -P_\ast){\bf w} = (I -P_\ast) (A_g - \lambda I)^N{\bf w}.
\end{equation}
1.
If ${\bf w}\in \mathbb{C}^n$ \KMc{is} such that $(I - P_\ast){\bf w} \not = 0$ and that
\begin{equation}
\label{w-eigen-Ag}
{\bf w} \in \ker((A_g - \lambda I)^{m_\lambda}) \setminus \ker((A_g - \lambda I)^{m_\lambda-1})
\end{equation}
with $\lambda \not = kC_\ast$ for some $m_\lambda \in \KMi{\mathbb{N}}$, we know from (\ref{Dg-I-PN}) that 
\begin{equation*}
(I-P_\ast){\bf w} \in \ker((Dg({\bf x}_\ast) - \lambda I)^{m_\lambda}) \setminus \ker((Dg({\bf x}_\ast) - \lambda I)^{m_\lambda-1}).
\end{equation*}
Here we have used the fact that $(A_g - \lambda I)^{m_\lambda - 1}{\bf w} \not \in {\rm span}\{{\bf v}_{\ast, \alpha}\}$, otherwise \KMc{$\bar {\bf w} \equiv (A_g - \lambda I)^{m_\lambda-1} {\bf w} = c{\bf v}_{\ast, \alpha}$ satisfies $(A_g - \lambda I)\bar {\bf w} = c(A_g - \lambda I){\bf v}_{\ast, \alpha} = 0$}. 
But the latter never occurs because $(A_g - kC_\ast I){\bf v}_{\ast,\alpha} = 0$ and $\lambda \not = kC_\ast$ is assumed at present.
\par
Now we move to the case (\ref{w-eigen-Ag}) with $\lambda = kC_\ast$.
Then there are two cases to be considered:
\begin{itemize}
\item $(A_g - kC_\ast I)^{m_\lambda - 1}{\bf w} \not \in {\rm span}\{{\bf v}_{\ast, \alpha}\}$.
\item $(A_g - kC_\ast I)^{m_\lambda - 2}{\bf w} \not \in {\rm span}\{{\bf v}_{\ast, \alpha}\}$ and $(A_g - kC_\ast  I)^{m_\lambda - 1}{\bf w} \in {\rm span}\{{\bf v}_{\ast, \alpha}\}$.
\end{itemize}
In the first case, the both sides (\ref{Dg-I-PN}) must vanish with $N=m_\lambda$, while do not vanish with $N=m_\lambda - 1$.
Under the assumption $(I-P_\ast){\bf w} \not = 0$, this property indicates that $(I-P_\ast){\bf w} \in \ker((Dg({\bf x}_\ast) - kC_\ast I)^{m_\lambda}) \setminus \ker((Dg({\bf x}_\ast) - kC_\ast I)^{m_\lambda-1})$.
In the second case, on the other hand, the both sides \KMb{in} (\ref{Dg-I-PN}) must vanish with $N=m_\lambda - 1$ from (\ref{vector-evkC}), while do not vanish with $N=m_\lambda - 2$.
Under the assumption $(I-P_\ast){\bf w} \not = 0$, this property indicates that $(I-P_\ast){\bf w} \in \ker((Dg({\bf x}_\ast) - kC_\ast I)^{m_\lambda-1}) \setminus \ker((Dg({\bf x}_\ast) - kC_\ast I)^{m_\lambda-2})$.

\par
\bigskip
\KMi{2.} 
\KMi{
Let ${\bf w}_g\in \mathbb{C}^n$ be such that 
$(I - P_\ast){\bf w}_g \not = 0$ and that
\begin{equation}
\label{wg-eigen-Dg}
(I - P_\ast){\bf w}_g \in \ker((Dg({\bf x}_\ast) - \lambda_g I)^{m_\lambda}) \setminus \ker((Dg({\bf x}_\ast) - \lambda_g I)^{m_\lambda-1})
\end{equation}
with $\lambda_g \not = kC_\ast$ for some $m_\lambda \in \KMi{\mathbb{N}}$.
Then (\ref{Dg-I-P}) indicates that either of the following properties holds:
\begin{itemize}
\item ${\bf w}_g \in \ker((A_g - \lambda_g I)^{m_\lambda})\setminus \ker((A_g - \lambda_g I)^{m_\lambda -1})$,
\item $(A_g - \lambda_g I)^{m_\lambda-1} {\bf w}_g \not \in {\rm span}\{{\bf v}_{\ast,\alpha}\}$, $(A_g - \lambda_g I)^{m_\lambda} {\bf w}_g \in {\rm span}\{{\bf v}_{\ast,\alpha}\}$.
\end{itemize}
In the latter case, the relation (\ref{vector-evkC}) yields $(A_g - kC_\ast I)(A_g - \lambda_g I)^{m_\lambda} {\bf w}_g = 0$.
From (\ref{Dg-I-P-2}), we concluded that
\begin{equation*}
(A_g - kC_\ast I){\bf w}_g \in \ker((A_g - \lambda_g I)^{m_\lambda})\setminus \ker((A_g - \lambda_g I)^{m_\lambda -1})
\end{equation*}
holds in both cases.
Notice that $(A_g - kC_\ast I){\bf w}_g \not = 0$ because $(I-P_\ast){\bf w}_g \in T_{{\bf x}_\ast}\mathcal{E}\setminus \{0\}$ and is assumed to satisfy (\ref{wg-eigen-Dg}) with $\lambda_g \not = kC_\ast$.
\par
Now we move to the case (\ref{wg-eigen-Dg}) with $\lambda_g = kC_\ast$.
Because $(I-P_\ast){\bf w}_g \in T_{{\bf x}_\ast}\mathcal{E} \setminus \{0\}$ is assumed, (\ref{Dg-I-P}) implies that there are two cases to be considered, similar to the first statement:
\begin{itemize}
\item $(A_g - kC_\ast I)^{m_{\lambda_g}}{\bf w}_g = 0$.
\item $(A_g - kC_\ast I)^{m_{\lambda_g}}{\bf w}_g \in {\rm span}\{{\bf v}_{\ast, \alpha}\}$.
\end{itemize}
Similar to the proof of the first statement, the identity (\ref{Dg-I-P-2}) yields that 
\begin{equation*}
(A_g - kC_\ast I){\bf w}_g \in \ker((A_g - \lambda_g I)^{N})\setminus \ker((A_g - \lambda_g I)^{N-1})
\end{equation*}
holds for either $N = m_{\lambda_g}$ or $m_{\lambda_g}+1$.
If $m_{\lambda_g} > 1$ and $(A_g - kC_\ast I)^{N}{\bf w}_g = 0$ for $1\leq N < m_{\lambda_g}$, then (\ref{Dg-I-P}) with $\lambda = kC_\ast$ implies $(Dg({\bf x}_\ast) - kC_\ast)^N(I-P_\ast){\bf w}_g = 0$, which contradicts the assumption.
}

\end{proof}


\KMf{We have unraveled the correspondence of eigenpairs between matrices $Dg({\bf x}_\ast)$ and $A_g$.}
Next consider the relationship of eigenpairs between matrices $A_g$ and $A$, \KMh{given in (\ref{Ag-QH}) and (\ref{blow-up-power-determining-matrix}), respectively.}

\par
\bigskip

\begin{prop}
\label{prop-correspondence-ev-Ag-A}
Let ${\bf x}_\ast\in \mathcal{E}$ be \KMf{an equilibrium on the horizon for $g$ and suppose that Assumption \ref{ass-f} holds}.
Also, let \KMf{$\lambda \in {\rm Spec}(A_g)$ and ${\bf u}\in \ker((A_g - \lambda I)^N)\setminus \ker((A_g - \lambda I)^{N-1})$ for some $N\in \mathbb{Z}_{\geq 1}$}, where ${\bf u}$ is linearly independent from ${\bf v}_{\ast, \alpha}$.
If
\begin{equation*}
\tilde \lambda:= r_{{\bf x}_\ast}^k \lambda,\quad {\bf U} := r_{{\bf x}_\ast}^{\Lambda_\alpha}{\bf u},
\end{equation*}
namely
\begin{equation*}
{\bf U} = (U_1,\ldots, U_n)^T,\quad U_i := r_{{\bf x}_\ast}^{\alpha_i}u_i,
\end{equation*}
then \KMf{$\tilde \lambda \in {\rm Spec}(A)$ and ${\bf U} \in \ker((A - \tilde \lambda I)^N)\setminus \ker((A - \tilde \lambda I)^{N-1})$}.
Conversely, \KMf{if $\tilde \lambda \in {\rm Spec}(A)$ and ${\bf U}\in \ker((A - \lambda I)^N)\setminus \ker((A - \lambda I)^{N-1})$ for some $N\in \mathbb{Z}_{\geq 1}$,
then the pair $\{\lambda, {\bf u}\}$ defined by
\begin{equation*}
\lambda:= r_{{\bf Y}_0}^{-k} \tilde \lambda,\quad {\bf u} := r_{{\bf Y}_0}^{-\Lambda_\alpha}{\bf U}
\end{equation*}
satisfy $\lambda \in {\rm Spec}(A_g)$ and ${\bf u} \in \ker((A_g - \KMg{\lambda} I)^N)\setminus \ker((A_g - \KMg{\lambda} I)^{N-1})$.
}
\end{prop}

\begin{proof}
\KMe{Similar to} arguments in the proof of Theorem \ref{thm-ev-special}, it is sufficient to consider the case that $f({\bf y})$, equivalently $\tilde f({\bf x})$, is quasi-homogeneous.
\KMi{That is, $f({\bf y}) = f_{\alpha, k}({\bf y})$ and $\tilde f({\bf x}) = \tilde f_{\alpha, k}({\bf x})$,} which are assumed in the following arguments.
Recall that \KMf{an equilibrium on the horizon ${\bf x}_\ast$ with $C_\ast > 0$} and the corresponding root ${\bf Y}_0$ of the balance law satisfy
\begin{align}
\label{identity-balance-equilibrium}
&\KMh{{\bf x}_\ast = r_{{\bf Y}_0}^{-\Lambda_\alpha} {\bf Y}_0,\quad {\bf Y}_0 = r_{{\bf x}_\ast}^{\Lambda_\alpha}{\bf x}_\ast},\\
\notag
&r_{{\bf Y}_0} = p({\bf Y}_0) = r_{{\bf x}_\ast} \equiv (kC_\ast)^{-1/k} > 0\KMf{.}
\end{align}
Similar to arguments in Lemma \ref{lem-identity-QHvf}, we have
\begin{equation}
\label{identity-QH-diff}
s^{\alpha_l}\frac{\partial f_i}{\partial x_l}( \KMh{ s^{\Lambda_\alpha}{\bf x}} ) = s^{k + \alpha_i}\frac{\partial f_i}{\partial x_l}( \KMh{{\bf x}} ), 
\end{equation}
while the left-hand side coincides with
\begin{equation*}
s^{\alpha_l}\frac{\partial f_i}{\partial x_l}( \KMh{ s^{\Lambda_\alpha}{\bf x}} ) = \frac{\partial f_i}{\partial (s^{\alpha_l} x_l)}( \KMh{ s^{\Lambda_\alpha}{\bf x}} )\frac{\partial (s^{\alpha_l} x_l)}{\partial \KMh{x_l}} \equiv \frac{\partial f_i}{\partial X_l}( \KMh{{\bf X}} )\frac{\partial (s^{\alpha_l} x_l)}{\partial x_l},
\end{equation*}
introducing an auxiliary variable ${\bf X} = (X_1, \ldots, X_n)^T$, $X_i := s^{\alpha_i} x_i$ for some $s > 0$.
\KMi{Let $D_{\bf X}$ be the derivative with respect to the vector variable ${\bf X}$.}
Note that $D_{\bf X}\KMi{\tilde f}({\bf X})|_{{\bf X} = \bar {\bf x}} = D_{\bf x} \KMi{\tilde f}(\bar {\bf x})$ when \KMg{the variable ${\bf X}$ is set as ${\bf x}$} and that $D_{\bf X}\KMi{f}({\bf X})|_{{\bf X} = \bar {\bf Y}} = D_{\bf Y} \KMi{f}(\bar {\bf Y})$ when \KMg{the variable ${\bf X}$ is set as ${\bf Y}$}.
Using the fact that $\KMi{f}({\bf Y})$ and $\KMi{\tilde f}({\bf x})$ have the identical form, we have
\begin{align*}
\KMh{
D_{\bf Y}f({\bf Y}_0) r_{{\bf x}_\ast}^{\Lambda_\alpha} = r_{{\bf x}_\ast}^{kI + \Lambda_\alpha} D_{\bf x}\KMi{f}({\bf x}_\ast)
}
\end{align*}
\KMh{with} $s = r_{{\bf x}_\ast}$ and ${\bf x} = {\bf x}_\ast$ in (\ref{identity-QH-diff}) and the identity (\ref{identity-balance-equilibrium}).
That is,
\begin{equation}
\label{identity-QH-diff-2-matrix}
D_{\bf Y}\KMi{f}({\bf Y}_0) = \KMh{ r_{{\bf x}_\ast}^{kI + \Lambda_{\alpha}} } D_{\bf x}\KMi{\tilde f}({\bf x}_\ast) r_{{\bf x}_\ast}^{-\Lambda_{\alpha}}.
\end{equation}
Then we have
\begin{align*}
A &= -\frac{1}{k} \Lambda_\alpha + D_{{\bf Y}} f({\bf Y}_0)\quad \text{(from (\ref{blow-up-power-determining-matrix}))}\\
	&= -r_{{\bf x}_\ast}^k C_\ast \Lambda_\alpha + D_{{\bf Y}} f({\bf Y}_0)\quad \text{(from (\ref{identity-balance-equilibrium}))}\\
	&= -r_{{\bf x}_\ast}^k C_\ast \Lambda_\alpha +  \KMh{ r_{{\bf x}_\ast}^{kI + \Lambda_{\alpha}}} D_{\bf x}\KMi{\tilde f}({\bf x}_\ast) r_{{\bf x}_\ast}^{-\Lambda_{\alpha}}\quad \text{(from (\ref{identity-QH-diff-2-matrix}))}\\
	&= \KMh{ r_{{\bf x}_\ast}^{kI + \Lambda_{\alpha}} } \left( -C_\ast \Lambda_\alpha +  D_{\bf x}\KMi{\tilde f}({\bf x}_\ast) \right)r_{{\bf x}_\ast}^{-\Lambda_{\alpha}}\\
	&\KMh{= r_{{\bf x}_\ast}^{kI + \Lambda_{\alpha}} A_g r_{{\bf x}_\ast}^{-\Lambda_{\alpha}} \quad \text{(from (\ref{Ag-QH}))}}
\end{align*}
and hence
\begin{equation}
\label{conj-A-Ag}
A = \KMh{ r_{{\bf x}_\ast}^{kI+\Lambda_{\alpha}}} A_g r_{{\bf x}_\ast}^{-\Lambda_{\alpha}} \quad \Leftrightarrow \quad 
A_g = \KMh{ r_{{\bf Y}_0}^{-(kI + \Lambda_{\alpha})}} A r_{{\bf Y}_0}^{\Lambda_{\alpha}},
\end{equation}
where we have used $r_{{\bf Y}_0} = r_{{\bf x}_\ast}$.
In particular, for any $\lambda\in \mathbb{C}$ and $N\in \KMi{\mathbb{N}}$ with the identity $\tilde \lambda =  r_{{\bf x}_\ast}^k \lambda$, we have
\begin{equation}
\label{conj-A-Ag-2}
(A - \tilde \lambda I)^N = \KMh{ r_{{\bf x}_\ast}^{kN I + \Lambda_{\alpha}}} (A_g - \lambda I)^N r_{{\bf x}_\ast}^{-\Lambda_{\alpha}}\quad \Leftrightarrow \quad
(A_g - \lambda I)^N = \KMh{ r_{{\bf Y}_0}^{-(kN I +\Lambda_{\alpha})}} (A - \tilde \lambda I)^N r_{{\bf Y}_0}^{\Lambda_{\alpha}}.
\end{equation}
This identity directly yields our statements.
For example, let ${\bf u} = (u_1,\ldots, u_n)$ be an eigenvector of $A_g$ associated with an eigenvalue $\lambda$:  
$A_g {\bf u} = \lambda {\bf u}$.
Then (\ref{conj-A-Ag}) yields
\begin{align*}
 \lambda {\bf u} &= A_g {\bf u}
 	= \KMh{ r_{{\bf Y}_0}^{-(k I + \Lambda_{\alpha})}} A r_{{\bf Y}_0}^{\Lambda_{\alpha}} {\bf u}
 	= \KMh{ r_{{\bf Y}_0}^{-(k I + \Lambda_{\alpha})}} A {\bf U},
\end{align*}
and hence
\begin{equation*}
A{\bf U} = r_{{\bf Y}_0}^{-k}\lambda {\bf U} = \tilde \lambda {\bf U}.
\end{equation*}
Repeating the same argument conversely assuming the eigenstructure $A{\bf U} = \tilde \lambda {\bf u}$, we know that an eigenpair $(\lambda, {\bf u})$ of $A_g$ is constructed from a given eigenpair $(\tilde \lambda, {\bf U})$ of $A$ through (\ref{identity-balance-equilibrium}) and (\ref{conj-A-Ag}).
Correspondence of generalized eigenvectors follows from the similar arguments through (\ref{conj-A-Ag-2}).
\end{proof}

\subsection{Complete correspondence of eigenstructures and another blow-up criterion}
\label{section-parameter-dep-Y}

Our results here determine the complete correspondence of eigenpairs among $A$ and $Dg({\bf x}_\ast)$.
In particular, blow-up power eigenvalues determining powers of $\theta(t)$ in asymptotic expansions of blow-up solutions are completely determined by ${\rm Spec}(Dg({\bf x}_\ast))$\KMf{, and vice versa}.
The complete correspondence of eigenstructures is obtained under a mild assumption of the corresponding matrices.

\begin{thm}
\label{thm-blow-up-estr}
Let ${\bf x}_\ast\in \mathcal{E}$ be \KMf{an equilibrium on the horizon for $g$ which is mapped to a nonzero root ${\bf Y}_0$ of the balance law (\ref{0-balance}) through (\ref{x-to-C}) and (\ref{C-to-x}), and suppose that Assumption \ref{ass-f} holds}.
\KMf{When all the eigenpairs of the blow-up power-determining matrix $A$ associated with ${\bf Y}_0$
 are determined, then all the eigenpairs of $Dg({\bf x}_\ast)$ are constructed through the correspondence listed in Table \ref{table-eigen1}.
Similarly, if all the eigenpairs of $Dg({\bf x}_\ast)$ are determined, then all the eigenpairs of $A$ are constructed through the correspondence listed in Table \ref{table-eigen2}.}
\par
\KMi{Moreover, the Jordan structure associated with eigenvalues, namely the number of Jordan blocks and their size, are identical except $kC_\ast \in {\rm Spec}(Dg({\bf x}_\ast))$ if exists, and $1\in {\rm Spec}(A)$.}
\end{thm}

\begin{proof}
\KMf{Correspondences} between {\em Common eigenvalue} and {\em Common eigenvector} \KMf{in Tables follow} from Theorems \ref{thm-ev1} and \ref{thm-ev-special}, \KMb{and \ref{thm-balance-1to1},} \KMf{while correspondences} between {\em Remaining eigenvalue} and {\em Remaining (generalized) eigenvector} \KMf{in Tables follow} from \KMg{Proposition \ref{prop-correspondence-ev-Ag-A} and Theorem \ref{thm-evec-Dg}}.
If $kC_\ast \not \in {\rm Spec}(Dg({\bf x}))$ (in particular, $1\in {\rm Spec}(A)$ is simple from Theorem \ref{thm-ev1} and Proposition \ref{prop-correspondence-ev-Ag-A}), the number and size of Jordan blocks are identical by Proposition \ref{prop-correspondence-ev-Ag-A} and Theorem \ref{thm-evec-Dg}.
\end{proof}

\begin{table}[ht]\em
\centering
{
\begin{tabular}{cccc}
\hline
  & $A$ & $Dg({\bf x}_\ast)$\\
\hline\\[-2mm]
Common eigenvalue & $1$ & $-C_\ast$\\ [1mm]
Common eigenvector & $\KMf{\Lambda_\alpha {\bf Y}_0}$ & \KMg{$\Lambda_\alpha r_{{\bf Y}_0}^{-\Lambda_\alpha} {\bf Y}_0$} \\  [1mm]
Remaining eigenvalue & $\tilde \lambda$ & $\lambda = r_{{\bf Y}_0}^{-k}\tilde \lambda$ \\  [1mm]
Remaining (generalized) eigenvector & \KMg{${\bf U}$} & \KMg{$(I-P_\ast) r_{{\bf Y}_0}^{-\Lambda_\alpha} {\bf U}$} \\  [1mm]
\hline 
\end{tabular}%
}
\caption{Correspondence of eigenstructures from $A$ to $Dg({\bf x}_\ast)$}
\flushleft
The constant $r_{{\bf Y}_0}$ is $p({\bf Y}_0)$. 
Once a nonzero root ${\bf Y}_0$ of the balance law and eigenpairs of $A$ are given, corresponding equilibrium on the horizon ${\bf x}_\ast$ and all eigenpairs of $Dg({\bf x}_\ast)$ are constructed by the rule on the table.
\label{table-eigen1}
\end{table}%

\begin{table}[ht]\em
\centering
{
\begin{tabular}{cccc}
\hline
  & $Dg({\bf x}_\ast)$  & $A$\\
\hline\\[-2mm]
Common eigenvalue & $-C_\ast$  & $1$ \\ [1mm]
Common eigenvector & $\KMf{\Lambda_\alpha {\bf x}_\ast}$ & \KMg{$\Lambda_\alpha r_{{\bf x}_\ast}^{\Lambda_\alpha} {\bf x}_\ast$} \\  [1mm]
Remaining eigenvalue & $\lambda$ & $\tilde \lambda = r_{{\bf x}_\ast}^k\lambda$ \\  [1mm]
Remaining (generalized) eigenvector & $(I-P_\ast){\bf u}$ & $r_{{\bf x}_\ast}^{\Lambda_\alpha}\KMi{(A_g - kC_\ast I)}{\bf u}$ \\  [1mm]
\hline 
\end{tabular}%
}
\caption{Correspondence of eigenstructures from $Dg({\bf x}_\ast)$ to $A$}
\flushleft
The constant $r_{{\bf x}_\ast}$ is $(kC_\ast)^{-1/k}$, which is positive whenever ${\bf x}_\ast$ is hyperbolic by Corollary \ref{cor-Cast-pos}.
Once \KMb{an} equilibrium on the horizon ${\bf x}_\ast$ and eigenpairs of $Dg({\bf x}_\ast)$  are given, corresponding (nonzero) root of the balance law ${\bf Y}_0$ and all eigenpairs of $A$ are constructed by the rule on the table.
\label{table-eigen2}
\end{table}%

The correspondence of {\em Common eigenvector}, and structure of $\mathcal{E}$ (and Remark \ref{rem-zero-comp-compactification} if necessary) yield that $\Lambda_\ast {\bf Y}_0$ is not a zero-vector.
As a \KMi{byproduct} of the above correspondence, another criterion of blow-ups is provided.
We have already reviewed in Section \ref{section-preliminary} that hyperbolic equilibria on the horizon for $g$ provide blow-up solutions.
On the contrary, the following result provides a criterion of the existence of blow-ups which can be applied {\em without a knowledge of desingularized vector fields}, while the correspondence to dynamics at infinity through $g$ is indirectly used.

\begin{thm}[Criterion of existence of blow-up from asymptotic expansions]
\label{thm-existence-blow-up}
\KMi{Let ${\bf Y}_0$ be a nonzero root of the balance law (\ref{0-balance}).}
Assume that the corresponding blow-up power determining matrix $A$ associated with ${\bf Y}_0$ is hyperbolic: ${\rm Spec}(A) \cap i\mathbb{R}=\emptyset$, \KMi{and that Assumption \ref{ass-f} holds}.
Then \KMf{(\ref{ODE-original}) possesses a blow-up solution ${\bf y}(t)$ with the asymptotic behavior $y_i(t) \sim Y_{0,i}\theta(t)^{-\alpha_i/k}$ as $t\to t_{\max} < \infty$, provided $Y_{0,i} \not = 0$.}
\end{thm}

\begin{proof}
Eigenvalues ${\rm Spec}(A)$ of $A$ consist of $1$ and remaining $n-1$ eigenvalues, all of which have nonzero real parts by our assumption.
\KMi{Because ${\bf Y}_0$ is nonzero, an equilibrium on the horizon ${\bf x}_\ast$ for $g$ is uniquely determined through the identity (\ref{C-to-x}). 
Moreover, the constant $C_\ast$ is defined through the identity $r_{{\bf x}_\ast} = (kC_\ast)^{-1/k} = r_{{\bf Y}_0}$, namely $C_\ast = 1/(k r_{{\bf Y}_0}^{k}) > 0$ from Corollary \ref{cor-Cast-pos}.}
The Jacobian matrix $Dg({\bf x}_\ast)$ has eigenvalues $-C_\ast$ and remaining $n-1$ eigenvalues determined one-to-one by ${\rm Spec}(A)\setminus \{1\}$, all of which have nonzero real parts, thanks to the correspondence \KMb{obtained} in \KMi{Theorem \ref{thm-blow-up-estr}}.
In particular, ${\bf x}_\ast$ is a hyperbolic equilibrium on the horizon satisfying $W^s_{\rm loc}({\bf x}_\ast; g)\cap \mathcal{D} \not = \emptyset$ \KMi{because} $-C_\ast < 0$\KMi{,} and the associated eigenvector ${\bf v}_{\ast, \alpha}$ determines the distribution of $W^s_{\rm loc}({\bf x}_\ast; g)$ transversal to $\mathcal{E}$.
Then Theorem \ref{thm:blowup} shows that $t_{\max} < \infty$ for the corresponding solution with the asymptotic behavior $y_i(t) = O(\theta(t)^{-\alpha_i/k})$, as long as $\KMh{x_{\ast, i}}\not = 0$.
Therefore, the bijection
\begin{equation*}
\frac{x_i(t)}{(1-p({\bf x}(t))^{2c})^{\alpha_i}} = \theta(t)^{-\alpha_i/k}Y_{0,i},\quad i=1,\cdots, n
\end{equation*}
provides the concrete form of the blow-up solution ${\bf y}(t)$ \KMf{whenever $Y_{0,i} \not = 0$}.
\end{proof}

We therefore conclude that {\em asymptotic expansions of blow-up solutions themselves provide a criterion of the existence of blow-up solutions}.
On the other hand, blow-up power eigenvalues do {\em not} extract exact dynamical properties around the corresponding blow-up solutions, as shown below.

\begin{thm}[Stability \KMf{gap}]
\label{thm-stability}
Let $f$ be an asymptotically quasi-homogeneous vector field of type $\alpha$ and order $k+1$ satisfying Assumption \ref{ass-f}.
Let ${\bf x}_\ast$ be a hyperbolic equilibrium on the horizon for the desingularized vector field $g$ associated with $f$ \KMf{such that $W^s_{\rm loc}({\bf x}_\ast; g)\cap \mathcal{D}\not = \emptyset$}, and ${\bf Y}_0$ be the corresponding root of the balance law which is not identically zero.
If
\begin{align*}
m &:= \dim W^s_{\rm loc}({\bf x}_\ast; g),\quad m_A = \sharp \{\lambda\in {\rm Spec}(A) \mid {\rm Re}\lambda < 0\},
\end{align*}
then we have $m = m_A+1$.
\end{thm}

\begin{proof}
Theorem \ref{thm-evec-Dg} and Proposition \ref{prop-correspondence-ev-Ag-A} indicate that $n-1$ eigenvalues of $Dg({\bf x}_\ast)$ and $A$ have the identical sign.
The only difference comes from eigenvalues $1$ associating the eigenvector ${\bf v}_{0,\alpha}$ of $A$, and $-C_\ast$ associating the eigenvector ${\bf v}_{\ast, \alpha}$ of $Dg({\bf x}_\ast)$, respectively.
From the hyperbolicity of ${\bf x}_\ast$, the constant $C_\ast$ is positive by Corollary \ref{cor-Cast-pos} and hence $1$ and $-C_\ast$ have mutually opposite sign, which shows $m=m_A+1$.
\end{proof}

\begin{rem}[Gap of dynamical information of blow-ups]
Theorem \ref{thm-stability} tells us an assertion to interpret stability information of blow-up solutions.
We have two vector fields for characterizing blow-up solutions: the desingularized vector field $g$ and (\ref{blow-up-basic}).
In both systems, the linear parts around steady states (equilibrium on the horizon and roots of the balance law, respectively) characterize the local asymptotic behavior of blow-ups under mild assumptions.
More precisely, {\em stable} eigenvalues in the sense that \KMf{their} real parts are negative and the associated eigenspaces parameterize the asymptotic behavior.
However, the number of such eigenvalues is {\em different} among two systems.
In the desingularized vector field $g$, all dynamical information for the original vector field $f$ are kept through compactifications and time-scale desingularizations, and hence all possible parameter dependence of blow-ups are extracted for a given equilibrium on the horizon.
The blow-up time $t_{\max}$ is regarded as \KMf{a} computable quantity given by solution trajectories for $g$\KMi{, namely (\ref{blow-up-time})}.
On the other hand, $t_{\max}$ is assumed to be fixed in (\ref{blow-up-basic}) and hence the dependence of $t_{\max}$ on initial points is neglected, which causes the gap of parameter dependence of blow-up solutions among two systems.
\KMa{Convergence of asymptotic series (Theorem 3.8 in \cite{asym1})} indicates that $t_{\max}$ expresses the gap of parameter dependence by means of stability information among two systems.
In \cite{LMT2021}, the maximal existence time $t_{\max}$ as \KMi{a} function of initial points is calculated through the parameterization of invariant manifolds (cf. \cite{CFdlL2005}) with computer-assistance, which constructs the foliation of $W^s_{\rm loc}({\bf x}_\ast)$ by means of level sets of $t_{\max}$.
This foliation can express the remaining parameter dependence of blow-ups.
\end{rem}

\section{Examples of asymptotic expansions revisited}
\label{section-examples}

Examples of asymptotic expansions of blow-up solutions shown in \cite{asym1} are \KMg{revisited}. Here the correspondence of algebraic information for describing asymptotic expansions to dynamics at infinity is revealed.

\subsection{One-dimensional ODEs}
\label{section-ex-1dim}

\subsubsection{A simple example}
The first example is 
\begin{equation}
\label{ex1-1dim-1}
y' = -y + y^3\KMb{,\quad {}' = \frac{d}{dt}}.
\end{equation}
If the initial point $y(0) > 0$ is sufficiently large, the corresponding solution would blow up in a finite time.
In \cite{asym1}, the third order asymptotic expansion of the type-I blow-up solution is derived as follows:
\begin{equation*}
y(t) \sim \frac{1}{\sqrt{2}}\theta(t)^{-1/2} + \frac{1}{2\sqrt{2}}\theta(t)^{1/2} +\frac{\sqrt{2}}{48}\theta(t)^{3/2}\quad \text{ as }\quad t\to t_{\max},
\end{equation*}
{\em assuming} its existence.
Here we pay attention to its existence through compactifications and the dynamical correspondence obtained in Section \ref{section-correspondence}.
\par
Here apply the (homogeneous) parabolic compactification and the time-scale desingularization to (\ref{ex1-1dim-1}).
First note that the ODE (\ref{ex1-1dim-1}) is asymptotically homogeneous (namely $\alpha = (1)$) of order $k+1 = 3$, in particular $k=2$.
The following parabolic compactification
\begin{equation*}
y = \frac{x}{1-x^2}.
\end{equation*}
is therefore applied.
The equation (\ref{ex1-1dim-1}) is transformed into
\begin{equation*}
x' = \frac{- x(1-x^2)^2 + x^3}{(1-x^2)(1+x^2)}.
\end{equation*}
Now the following time-scale desingularization is introduced:
\begin{equation*}
\frac{d\tau}{dt} = \frac{2}{(1-x^2)^2(1+x^2)}.
\end{equation*}
The corresponding desingularized vector field is 
\begin{equation}
\label{ex1-desing}
\frac{dx}{d\tau} = \frac{1}{2}(1-x^2)\left\{ - x(1-x^2)^2 + x^3\right\} = \frac{1}{2}x(1-x^2)\left\{ - 1+3x^2 - x^4 \right\}.
\end{equation}
Note that the horizon is $\{x=\pm 1\}$ and \KMb{that} the desingularized system admits equilibria on the horizon: $x=\pm 1$.
Our interest here is the blow-up solution associated with $x_\ast =1$.
The differential of the desingularized vector field at $x_\ast =1$ is 
\begin{equation*}
\left[ \frac{1}{2}(1-x^2)\left\{ - 1+3x^2 - x^4 \right\} - x^2 \left\{ - 1+3x^2 - x^4 \right\} + \frac{1}{2}x(1-x^2)\left\{ 6x - 4x^3 \right\}\right]_{x=x_\ast} = -1,
\end{equation*}
indicating that $x_\ast =1$ is a hyperbolic sink for (\ref{ex1-desing}).
Theorem \ref{thm:blowup} yields that this sink induces a solution $y(t)$ for (\ref{ex1-1dim-1}) with large initial points blowing up at $t = t_{\max} < \infty$ with the blow-up rate
\begin{equation*}
y(t) = O( \theta(t)^{-1/2})\quad \text{ as } \quad t\to t_{\max},
\end{equation*}
and hence the existence of blow-ups mentioned in Assumption \ref{ass-fundamental} is verified without assuming its existence.
The constant $C_\ast$ \KMg{in (\ref{const-horizon})} is
\begin{equation*}
C_\ast = \left. x \left\{ - x(1-x^2)^2 + x^3\right\}\right|_{x = x_\ast } = 1,
\end{equation*}
which is consistent with Theorem \ref{thm-ev-special}. 
Indeed, $-C_\ast = -1$ is the only eigenvalue of the Jacobian matrix $Dg(x_\ast)$. 
\par
Next, we \KMa{partially} review the asymptotic behavior of this blow-up solution to verify the dynamical correspondence.
Assume that
\begin{equation*}
y(t) = \theta(t)^{-1/2}Y(t), 
\end{equation*}
which yields the following equation solving $Y(t)$:
\begin{equation}
\label{system-asymptotic-1dim}
Y' = -Y + \theta(t)^{-1}\left\{ - \frac{1}{2}Y + Y^3\right\}.
\end{equation}
Under the asymptotic expansion of the positive blow-up solution:
\begin{equation*}
Y(t) = \sum_{n=0}^\infty Y_n(t)\quad \text{ with }\quad Y_n(t) \ll Y_{n-1}(t)\quad (t\to t_{\max}-0),\quad \lim_{t\to t_{\max}}Y(t)= Y_0 > 0,
\end{equation*}
the balance law requires $Y_0 = 1/\sqrt{2}$, which is the coefficient of the principal term of $y(t)$.
By definition of the functional $p({\bf y})$ in (\ref{func-p}), we have
\begin{equation*}
\frac{Y_0}{p(Y_0)} = \frac{1/\sqrt{2}}{((1/\sqrt{2})^2)^{1/2}} = 1 \equiv x_\ast,
\end{equation*}
while
\begin{equation*}
(kC_\ast)^{-1/k} x_\ast = 2^{-1/2} = \frac{1}{\sqrt{2}}\equiv Y_0.
\end{equation*}
Therefore the correspondence of roots in Theorem \ref{thm-balance-1to1} is verified.
The blow-up power-determining matrix at $Y_0$, coinciding with the blow-up power eigenvalue, is
\begin{equation*}
\left\{ - \frac{1}{2} + 3Y^2\right\}_{Y = Y_0} = 1,
\end{equation*}
which is consistent with Theorem \ref{thm-ev1}.
This eigenvalue has no contributions to $Y_n(t)$ with $n\geq 1$.
In particular, {\em blow-up power eigenvalues never contribute to determine the orders of \KMg{$\theta(t)$ in } asymptotic expansion of blow-up solutions for any one-dimensional ODEs}.

\subsubsection{Ishiwata-Yazaki's example}
The next example concerns with blow-up solutions of the following system:
\begin{equation}
\label{IY}
u' = a u^{\frac{a+1}{a}}v,\quad v' = a v^{\frac{a+1}{a}}u,
\end{equation}
where \KMg{$a\in (0,1)$} is a parameter.
\begin{rem}[cf. \cite{IY2003, Mat2019}]
\label{rem-IY}
\KMg{Consider} initial points $u(0), v(0) > 0$.
If $u(0) \not = v(0)$, then the solution $(u(t), v(t))$ blows up at $t=t_{\max} < \infty$ with the blow-up rate $O(\theta(t)^{-a})$.
On the other hand, if $u(0) = v(0)$, the solution $(u(t), v(t))$ blows up at $t=t_{\max} < \infty$ with the blow-up rate $O(\theta(t)^{-a/(a+1)})$.
\end{rem}
Introducing the first integral
\begin{equation*}
I = I(u, v) := v^{1-\frac{1}{a}} - u^{1-\frac{1}{a}},
\end{equation*}
the system (\ref{IY}) is reduced to a one-dimensional ODE
\begin{equation}
\label{IY-1dim}
u' = a u^{\frac{a+1}{a}}\left( u^{1 - \frac{1}{a}} + I \right)^{\frac{a}{a-1}}.
\end{equation}
Blow-up solutions of the rate $O(\theta(t)^{-a})$ corresponds to $I \not = 0$, while those of the rate $O(\theta(t)^{-a/(a+1)})$ corresponds to $I=0$.
We pay attention to the case $u(0) > v(0)$ when $I\not = 0$, in which case $I>0$ holds.

\par
\bigskip
First consider the case $I>0$, where the vector field (\ref{IY-1dim}) is asymptotically homogeneous of the order $1+a^{-1}$.
Using the asymptotic expansion
\begin{equation*}
u(t) = \theta(t)^{-a}U(t) = \theta(t)^{-a}\left(\sum_{n=0}^\infty U_n(t)\right),\quad \lim_{t\to t_{\max}} U(t) = U_0,\\
\end{equation*}
the system becomes
\begin{align}
\notag
U' &= a\theta(t)^{-1} \left\{ -U + \KMf{ I^{\frac{a}{a-1}} } U^{\frac{a+1}{a}} \left( I^{-1} \theta(t)^{-(a-1)}U^{1 - \frac{1}{a}} + 1 \right)^{\frac{a}{a-1}}\right\}\\
\label{IY-system-U}
	&= a\theta(t)^{-1} \left\{ -U + \KMf{ I^{\frac{a}{a-1}} } U^{\frac{a+1}{a}} \left(  \sum_{k=0}^\infty \begin{pmatrix}
\frac{a}{a-1} \\ k
\end{pmatrix}\left( I^{-1}\theta(t)^{1-a}U^{\frac{a-1}{a}} \right)^k \right) \right\}, 
\end{align}
where
\begin{equation*}
\begin{pmatrix}
\frac{a}{a-1} \\ k
\end{pmatrix} = \frac{\left(\frac{a}{a-1}\right)_k}{k!},\quad 
\left(\frac{a}{a-1}\right)_k = \frac{a}{a-1} \left(\frac{a}{a-1}-1\right)\left(\frac{a}{a-1}-2\right) \cdots \left(\frac{a}{a-1}-k+1\right).
\end{equation*}

The balance law then yields
\begin{equation*}
\KMa{-U_0 +  I^{\frac{a}{a-1}} U_0^{\frac{a+1}{a}} = 0} \quad \Rightarrow \quad U_0 = I^{-a^2/(a-1)},
\end{equation*}
\KMf{where the above choice of $U_0$ is consistent with the setting mentioned in Remark \ref{rem-IY}.}
The \KMf{corresponding} blow-up power-determining matrix is
\begin{align*}
\frac{d}{dU}\left(a \left\{ -U + U^{\frac{a+1}{a}} I^{\frac{a}{a-1}} \right\} \right)_{U=U_0} &=a \left\{ -1 + \frac{a+1}{a}U_0^{\frac{1}{a}} I^{\frac{a}{a-1}} \right\} = 1,
\end{align*}
which is consistent with Theorem \ref{thm-ev1}.
\par
\bigskip
Similarly, the blow-up power eigenvalue in the case $I=0$ is confirmed to be consistent with Theorem \ref{thm-ev1}. 
Details are shown in \cite{asym1}.
\par
\bigskip
As a summary, we obtain the following result for asymptotic expansions of blow-up solutions.
\begin{align}
\label{IY-sol-Ipos}
u(t) &\sim I^{-a^2/(a-1)} \theta(t)^{-a} + I^{\frac{-2a^2+1}{a-1}}
\frac{a^2}{(1-a)(2-a)} \theta(t)^{1-2a},\quad
v(t) \sim I^{\frac{a}{a-1}} - I^{\frac{1}{a-1}-a} \frac{a}{1-a} \theta(t)^{1-a}
\end{align}
with $I> 0$, while
\begin{equation}
\label{IY-sol-I0}
u(t) = v(t) = \left( \frac{1}{a+1}\right)^{\frac{a}{a+1}}\theta(t)^{-a/(a+1)}
\end{equation}
with $I=0$ as $t\to t_{\max}-0$.
Note that the solution through the above argument coincides with that obtained by the method of separation of variables in the original equation
\begin{equation*}
u' = au^{2+\frac{1}{a}}.
\end{equation*}
Details are summarized in \cite{asym1}.
\par
\bigskip
\begin{rem}[Correspondence of parameter dependence]
\KMa{
The expansion (\ref{IY-sol-Ipos}) contains two free parameters: $t_{\max}$ and $I$.
This fact reflects the dynamical property that the solution (\ref{IY-sol-I0}) is induced by the hyperbolic sink on the horizon admitting the {\em two}-dimensional stable manifold (after further time-scale desingularizations), as observed in \cite{Mat2019}.
}
On the other hand, when $I=0$, $u(t) \equiv v(t)$ \KMg{holds} for $t\geq 0$ by the invariance of the first integral $I$.
The expansion (\ref{IY-sol-I0}) is parameterized by only {\em one} parameter $t_{\max}$, namely initial points $u(0) = v(0)$.
Similar to (\ref{IY-sol-Ipos}), this fact reflects the dynamical property that the solution (\ref{IY-sol-I0}) is induced by the hyperbolic saddle \KMg{on the horizon} admitting the {\em one}-dimensional stable manifold (\cite{Mat2019}).
\KMa{Gaps of the number of free parameters are consistent with Theorem \ref{thm-stability}.} 
\end{rem}

\subsection{Two-phase flow model}
\label{section-ex-2phase}

The following system is reviewed next (see e.g. \cite{KSS2003, Mat2018} for the details of the system):
\begin{equation}
\label{two-fluid-1}
\begin{cases}
\beta' = vB_1(\beta) - c\beta - c_1, & \\
v' =  v^2 B_2(\beta) - cv - c_2, & 
\end{cases}\quad {}'=\frac{d}{dt},
\end{equation}
where
\begin{equation*}
B_1(\beta) = \frac{(\beta-\rho_1)(\beta-\rho_2)}{\beta},\quad B_2(\beta) = \frac{\beta^2- \rho_1\rho_2}{2\beta^2}
\end{equation*}
with $\rho_2 > \rho_1 > 0$, 
\begin{equation*}
c = \frac{v_R B_1(\beta_R) - v_L B_1(\beta_L)}{\beta_R - \beta_L}
\end{equation*}
and $(c_1,c_2) = (c_{1L}, c_{2L})$ or $(c_{1R}, c_{2R})$, where 
\begin{equation*}
\label{constants-two-phase}
\begin{cases}
c_{1L} = v_L B_1(\beta_L) - c\beta_L, & \\
c_{2L} = v_L^2 B_2(\beta_L) -cv_L, & \\
\end{cases}
\quad
\begin{cases}
c_{1R} = v_R B_1(\beta_R) - c\beta_R, & \\
c_{2R} = v_R^2 B_2(\beta_R) -cv_R. & \\
\end{cases}
\end{equation*}
Points $(\beta_L, v_L)$ and $(\beta_R, v_R)$ are given in advance.
The system (\ref{two-fluid-1}) is asymptotically quasi-homogeneous of type $(0,1)$ and order $2$.
Following arguments in \cite{Mat2018}, we observe that there is a blow-up solution with the asymptotic behavior
\begin{equation}
\label{two-fluid-2}
\beta(t) \sim \rho_2,\quad v(t)\sim V_0\theta(t)^{-1}\quad\text{ as }\quad t\to t_{\max}-0,
\end{equation}
which is consistent with arguments in \cite{KSS2003}.
In particular, type-I blow-up solutions are observed.
\begin{rem}
In \cite{Mat2018}, two hyperbolic saddles on the horizon for the desingularized vector field are observed.
One of these saddles admits the {\em $1$-dimensional stable manifold}, which associates a family of blow-up solutions of the above form.
Another saddle admits the {\em $1$-dimensional unstable manifold}, which associates a family of blow-up solutions of the similar form {\em with time reversing}.
\end{rem}
\par
Our main concern here is to derive multi-order asymptotic expansion of the blow-up solution (\ref{two-fluid-2}) for (\ref{two-fluid-1}).
To this end, write the blow-up solution $(\beta(t), v(t))$ as follows:
\begin{align}
\notag
\beta(t) &= b(t),\quad v(t) = \theta(t)^{-1}V(t),\\
\label{form-b-V}
b(t) &= \sum_{n=0}^\infty b_n(t) \equiv b_0 + \tilde b(t),\quad b_0 = \rho_2,\quad b_n(t) \ll b_{n-1}(t)\quad (t\to t_{\max}-0),\quad n\geq 1,\\
\notag
V(t) &= \sum_{n=0}^\infty V_n(t)\equiv V_0 + \tilde V(t),\quad V_n(t) \ll V_{n-1}(t)\quad (t\to t_{\max}-0),\quad n\geq 1.
\end{align}
The balance law which $(b_0, v_0)$ satisfies can be easily derived. 
Substituting the form (\ref{form-b-V}) into (\ref{two-fluid-1}), we have
\begin{align*}
\beta' &= b' \\
	&= \theta(t)^{-1} V B_1(b) - cb - c_1,\\
v' &= \theta(t)^{-2} V + \theta(t)^{-1}V'\\
	&= \theta(t)^{-2} V^2 B_2(b) - c\theta(t)^{-1}V - c_2.
\end{align*}
Dividing the first equation by $\theta(t)^{0} \equiv 1$ and the second equation by $\theta(t)^{-1}$, we have 
\begin{align}
\label{2phase-asym-main}
\frac{d}{dt}\begin{pmatrix}
b \\
V
\end{pmatrix} = \theta(t)^{-1} \begin{pmatrix}
V B_1(b) \\
-V  + V^2 B_2(b)
\end{pmatrix} - \begin{pmatrix}
cb + c_1 \\
cV  + \theta(t) c_2
\end{pmatrix}
\end{align}
The balance law is then 
\begin{equation*}
\begin{pmatrix}
V_0 B_1(b_0) \\
-V_0  + V_0^2 B_2(b_0)
\end{pmatrix} = \begin{pmatrix}
0 \\ 
0
\end{pmatrix},
\end{equation*}
that is,
\begin{equation*}
V_0  \frac{(b_0-\rho_1)(b_0-\rho_2)}{\beta} = 0,\quad -V_0 + V_0^2 \frac{b_0^2- \rho_1\rho_2}{2b_0^2} = 0.
\end{equation*}
In the present case, $b_0 = \rho_2$ is already determined as the principal term of $b(t)$, which satisfies the first equation.
Substituting $b_0 = \rho_2$ into the second equation, we have
$V_0 = 2\rho_2 / (\rho_2- \rho_1)$, provided $V_0 \not =0$.
As a summary, the root of the balance law (under (\ref{two-fluid-2})) is uniquely determined by
\begin{equation}
\label{balance-two-phase}
(b_0, V_0) = \left(\rho_2, \frac{2\rho_2}{\rho_2- \rho_1}\right).
\end{equation}
Letting
\begin{equation*}
\KMg{\bar f}(b,V) := \begin{pmatrix}
V B_1(b) \\
-V  + V^2 B_2(b)
\end{pmatrix} \KMg{ \equiv -\frac{1}{k}\Lambda_\alpha \begin{pmatrix}
b \\ V
\end{pmatrix}+ f_{\alpha, k}(b,V)},
\end{equation*}
we have
\begin{align*}
D\KMg{\bar f}(b_0, V_0) &= \begin{pmatrix}
V_0\frac{d}{d\beta}B_1(\beta)|_{\beta=b_0} & B_1(b_0)\\
V_
0^2 \frac{d}{d\beta}B_2(\beta)|_{\beta=b_0} & -1 + 2V_0 B_2(b_0)\\
\end{pmatrix}
	= \begin{pmatrix}
V_0(1 - \rho_1\rho_2b_0^{-2}) & \frac{(b_0-\rho_1)(b_0-\rho_2)}{\KMf{b_0}}\\
V_0^2 \rho_1\rho_2 b_0^{-3} & -1 + 2V_0 \frac{b_0^2- \rho_1\rho_2}{2b_0^2}\\
\end{pmatrix},
\end{align*}
which is the blow-up power determining matrix associated with the blow-up solution (\ref{two-fluid-2}).
Using (\ref{balance-two-phase}), we have
\begin{align*}
A\equiv D\KMg{\bar f}(b_0, V_0) &= \begin{pmatrix}
 \frac{2\rho_2}{\rho_2- \rho_1}(1 - \rho_1\rho_2^{-1}) & 0\\
( \frac{2\rho_2}{\rho_2- \rho_1})^2 \rho_1\rho_2 b_0^{-3} & -1 + 2 \frac{2\rho_2}{\rho_2- \rho_1} \frac{\rho_2- \rho_1}{2\rho_2}\\
\end{pmatrix}\\
	&=\begin{pmatrix}
2 & 0\\
\frac{4\rho_1}{(\rho_2- \rho_1)^2} & 1\\
\end{pmatrix}.
\end{align*}
Indeed, $1$ is one of eigenvalues.
The corresponding eigenvector is
\begin{align*}
\begin{pmatrix}
2 & 0\\
\frac{4\rho_1}{(\rho_2- \rho_1)^2} & 1\\
\end{pmatrix}\begin{pmatrix}
x_1 \\ x_2
\end{pmatrix} = \begin{pmatrix}
x_1 \\ x_2
\end{pmatrix} \quad & \Rightarrow \quad \begin{pmatrix}
x_1 \\ x_2
\end{pmatrix} = \begin{pmatrix}
0 \\ 1
\end{pmatrix}
\end{align*}
and hence the present argument is consistent with Theorem \ref{thm-ev1}.
Recall that the type $\alpha$ is now $(0,1)$.

\subsection{Andrews' system I}
\label{section-ex-Andrews1}

Consider the following system (cf. \cite{A2002, Mat2019})\footnote{
There is a typo of the system in \cite{Mat2019} (Equation (4.7)). 
The correct form is (\ref{Andrews1}) below.
The rest of arguments in \cite{Mat2019} is developed for the correct system (\ref{Andrews1}).
}:
\begin{equation}
\label{Andrews1}
\begin{cases}
\displaystyle{
\frac{du}{dt} = \frac{1}{\sin \theta} vu^2 - \frac{2a\cos\theta}{\sin \theta}u^3
}
, & \\
\displaystyle{
\frac{dv}{dt} =  \frac{a}{\sin \theta} uv^2 + \frac{1-a}{\sin \theta} \frac{uv^3}{v + 2\cos\theta u}
}, &
\end{cases}
\end{equation}
where $a\in (0,1/2)$ and $\theta\in (0,\pi/2)$ are constant parameters, according to our interests (\cite{asym1}).
We easily see that the system (\ref{Andrews1}) is homogeneous of the order $3$, in particular $k=2$, and arguments in \cite{Mat2019} show that the first term of stationary blow-up solutions has the form
\begin{equation*}
u(t) = O(\theta(t)^{-1/2}),\quad v(t) = O(\theta(t)^{-1/2}).
\end{equation*}
Introducing
\begin{equation*}
w = u\cos\theta,\quad s = \frac{t}{\sin\theta \cos\theta},
\end{equation*}
the vector field (\ref{Andrews1}) is then transformed into
\begin{equation}
\frac{dw}{ds} = w^2 (v - 2a w),\quad \frac{dv}{ds} = wv^2 \frac{v + 2aw}{v+2w},
\end{equation}
which is independent of $\theta$.
Because $\theta \in (0,\pi/2)$, then $\sin \theta > 0$ and $\cos \theta > 0$\KMg{,} and hence dynamics of $(u,v)$ in $t$-time scale and $(w, v)$ in $s$-variable are mutually smoothly equivalent. 
\par
The following asymptotic expansions of blow-up solutions are considered:
\begin{equation*}
w(t) = \tilde \theta(s)^{-1/2}\sum_{n=0}^\infty W_n(s),\quad v(t) = \tilde \theta(s)^{-1/2}\sum_{n=0}^\infty V_n(s), \quad \tilde \theta(s) = s_{\max}-s = \frac{2}{\sin(2\theta)}\theta(t).
\end{equation*}
\KMa{In particular, we pay attention to solutions with {\em positive} initial points according to arguments in \cite{asym1}.}
%
The balance law for the solution under the assumption $W_0, V_0 \not = 0$ is 
\begin{equation*}
\frac{1}{2} = W_0 V_0 - 2aW_0^2,\quad 
\frac{1}{2} = W_0 V_0\frac{V_0 + 2aW_0}{V_0 + 2W_0}.
\end{equation*}
We therefore have
\begin{equation}
\label{balance-Andrews-result}
W_0 = \pm \sqrt{\frac{1-2a}{8a^2}},\quad V_0 = \frac{2a}{1-2a}W_0 = \pm \sqrt{\frac{1}{2(1-2a)}}, \end{equation}
\KMa{where the positive roots are chosen according to our setting.}
Moreover, the root $V_0$ can be achieved as a real value by the assumption of $a$.
\par
As seen in \cite{asym1}, the blow-up power-determining matrix is $A = -\frac{1}{2} I_2 + C(W_0, V_0; a)$, where 
\begin{align*}
&C(W_0, V_0; a) 
=\begin{pmatrix}
2V_0W_0 - 6aW_0^2 & W_0^2 \\
V_0^2 \left\{ \frac{V_0+2aW_0}{V_0+2W_0} - W_0 \frac{2(1-a)V_0}{(V_0+2W_0)^2} \right\} & V_0W_0 \left\{ 2 \frac{V_0+2aW_0}{V_0+2W_0} + V_0 \frac{2(1-a)W_0}{(V_0+2W_0)^2} \right\} \\ 
\end{pmatrix}.
\end{align*}
In \cite{asym1}, we have simplified the matrix into
\begin{align*}
C(W_0, V_0; a) &= 
\begin{pmatrix}
\frac{3}{2} - \frac{1}{4a}  & W_0^2 \\
\frac{a}{1-a} V_0^2 & \frac{3}{2} - \frac{1}{4(1-a)} 
\end{pmatrix}
\end{align*}
after lengthy calculations and eigenpairs of $A$ are calculated as follows:
\begin{align*}
\left\{1, \begin{pmatrix}
\sqrt{\frac{1-2a}{8a^2}}\\
\sqrt{\frac{1}{2(1-2a)}}
\end{pmatrix}\right\},\quad 
\left\{ 1 - \frac{1}{4a(1-a)}, \begin{pmatrix}
\sqrt{\frac{1-2a}{8a^2}}\\
\frac{a}{1-a}\sqrt{\frac{1}{2(1-2a)}}
\end{pmatrix}
\right\}.
\end{align*}
The first eigenpair is indeed consistent with Theorem \ref{thm-ev1}. 
Compare with (\ref{balance-Andrews-result}).

\subsection{Andrews' system II}
\label{section-ex-Andrews2}

The next example is the following system:
\begin{align}
\label{Andrews2}
	\begin{cases}
		u' = u^2 (2av - bu),
		\\
		v' = bu v^2
	\end{cases}
\end{align}
with parameters $a, b$ with $a > 0$ and $2a > b > 0$.
Our interest here is the asymptotic expansion of blow-up solutions with $u(0),v(0)>0$.
As in previous examples, we introduce 
\begin{equation}
\label{trans-Andrews2}
\tilde u = \frac{u}{\sqrt{b}},\quad \tilde v = \frac{v}{\sqrt{b}},\quad \KMg{s = b^2 t},\quad 2a = b(1+\sigma)
\end{equation}
with an auxiliary parameter $\sigma$, which transform (\ref{Andrews2}) into
\begin{align}
\label{Andrews2-transformed}
\frac{d\tilde u}{ds} = \tilde u^2 \left\{ (1+\sigma) v - \tilde u \right\},\quad \frac{d\tilde v}{ds} = \tilde u \tilde v^2.
\end{align}
In particular, the system becomes a {\em one}-parameter family.
Our interest here is then the blow-up solution $(\tilde u(s), \tilde v(s))$ with the following blow-up rate
\begin{equation*}
\tilde u(t) = O(\tilde \theta(s)^{-1/2}),\quad \tilde v(t) = O(\tilde \theta(s)^{-1/2}),\quad \tilde \theta(s) = s_{\max}-s = \KMg{b^2}\theta(t).
\end{equation*}
Expand the solution $(\tilde u(s), \tilde v(s))$ as the asymptotic series
\begin{align}
\notag
\tilde u(s) &= \tilde \theta(s)^{-1/2}U(s) \equiv \tilde \theta(s)^{-1/2}\sum_{n=0}^\infty U_n(s),\quad U_n(s) \ll U_{n-1}(s),\quad \lim_{s\to s_{\max}}U_n(s) = U_0, \\
\label{series-Andrews2}
\tilde v(s) &= \tilde \theta(s)^{-1/2}V(s) \equiv \tilde \theta(s)^{-1/2}\sum_{n=0}^\infty V_n(s),\quad V_n(s) \ll V_{n-1}(s), \quad \lim_{s\to s_{\max}}V_n(s) = V_0.
\end{align} 
Substituting 
\eqref{series-Andrews2} into \eqref{Andrews2-transformed}, we have
\begin{align}
\label{asymptotic-system-Andrews2}
\frac{dU}{ds} = \tilde \theta(s)^{-1} \left\{ -\frac{1}{2}U + U^2\{ (1+\sigma) V - U\} \right\},\quad 
\frac{dV}{ds} = \tilde \theta(s)^{-1} \left\{ -\frac{1}{2}V + UV^2\right\}.
\end{align}

The balance law under $(U_0, V_0)\not = (0, 0)$ requires
\begin{align*}
-\frac{1}{2} + U_0 \{(1+\sigma)V_0 - U_0 \} = 0,\quad -\frac{1}{2} + U_0V_0 = 0\KMg{,}
\end{align*}
and hence 
\begin{equation}
\label{identity-Andrews2}
1 = 2U_0\{(1+\sigma) V_0 - U_0 \},\quad 1 = 2U_0V_0\KMg{,}
\end{equation}
which are used below.
These identities yield the following consequence:
\begin{equation}
\label{balance-Andrews2}
U_0 = \sqrt{\frac{\sigma}{2}},\quad V_0 = \frac{1}{\sqrt{2\sigma}}\KMg{,}
\end{equation}
and we \KMg{have} the first order asymptotic expansion of blow-up solutions:
\begin{equation*}
\tilde u(s)\sim \sqrt{\frac{\sigma}{2}} \tilde \theta(s)^{-1/2},\quad 
\tilde v(s)\sim \frac{1}{\sqrt{2\sigma}} \tilde \theta(s)^{-1/2}
\end{equation*}
as $s\to s_{\max}$.
\par
\bigskip
The blow-up power-determining matrix is $A = -\frac{1}{2}I_2 +D(U_0, V_0; \sigma)$, where
\begin{align*}
&D(U_0, V_0; \sigma) =  
\begin{pmatrix}
 2(1+\sigma)UV - 3U^2 & (1+\sigma)U^2 \\
 V^2 & 2UV
\end{pmatrix}_{(U,V) = (U_0, V_0)} = 
 \begin{pmatrix}
 1 - \frac{\sigma}{2} & \frac{\sigma(1+\sigma)}{2} \\
 \frac{1}{2\sigma} & 1
\end{pmatrix}
\end{align*}
under the identity (\ref{identity-Andrews2}).
The blow-up power eigenvalues are
\begin{equation*}
\lambda = 1,\quad -\frac{\sigma}{2}.
\end{equation*}
\KMa{
The eigenvector associated with $\lambda = 1$ is
\begin{equation*}
 \begin{pmatrix}
 \frac{1}{2} - \frac{\sigma}{2} & \frac{\sigma(1+\sigma)}{2} \\
 \frac{1}{2\sigma} & \frac{1}{2}
\end{pmatrix}
 \begin{pmatrix}
x_1 \\
x_2
\end{pmatrix}
= \begin{pmatrix}
x_1 \\
x_2
\end{pmatrix} \quad \Rightarrow \quad 
\begin{pmatrix}
x_1 \\
x_2
\end{pmatrix} = \begin{pmatrix}
\sigma \\ 1
\end{pmatrix}.
\end{equation*}
Therefore the existence of eigenpair associated with the eigenvalue $\lambda = 1$ is consistent with Theorem \ref{thm-ev1} (see also (\ref{balance-Andrews2})).
}

\subsection{Keyfitz-Kranser-type system}
\label{section-ex-KK}

The next example is a $2$-dimensional system
\begin{equation}
\label{KK}
	u' = u^2 - v, \quad v' = \frac{1}{3} u^3 - u\KMg{,}
\end{equation}
originated from the Keyfitz-Kranser system \cite{KK1990} which is a system of conservation laws admitting a singular shock.
The system is asymptotically quasi-homogeneous of type $\alpha=(\alpha_1, \alpha_2)=(1,2)$ and order $k+1 = 2$, consisting of the quasi-homogeneous part $f_{\alpha, k}$ and the lower-order part $f_{\rm res}$ given as follows: 
\begin{equation}
	f_{\alpha, k}(u,v) = \begin{pmatrix}
 u^2 - v\\
 \frac{1}{3} u^3	
 \end{pmatrix}, \quad f_{\mathrm{res}}(u,v) = \begin{pmatrix}
 0\\
 -u	
 \end{pmatrix}.
\end{equation}
It is proved in \cite{Mat2018} that the system (\ref{KK}) admits the following solutions blowing up as $t \to t_{\max} - 0$ associated with {\em two} different equilibria on the horizon:
\begin{equation}
	u(t) = O(\theta(t)^{-1}), \quad v(t) = O (\theta(t)^{-2}), \quad \text{as} \quad t \to t_{\max} - 0.
\end{equation}
In \cite{asym1}, the quasi-homogeneous part $f_{\alpha, k}$ and the full system (\ref{KK}) are individually investigated.
However, the balance law and associated blow-up power-determining matrices and their eigenvalues are identical, because the quasi-homogeneous part of vector fields are identical between two systems.
\par
Expand the solution $(u(t),v(t))$ as the asymptotic series
\begin{align}
\notag
u(t) &= \theta(t)^{-1}U(t) \equiv \theta(t)^{-1}\sum_{n=0}^\infty U_n(t),\quad U_n(t) \ll U_{n-1}(t),\quad \lim_{t\to t_{\max}}U_n(t) = U_0, \\
\label{series-KK}
v(t) &= \theta(t)^{-2}V(t) \equiv \theta(t)^{-2}\sum_{n=0}^\infty V_n(t),\quad V_n(t) \ll V_{n-1}(t), \quad \lim_{t\to t_{\max}}V_n(t) = V_0.
\end{align}
The balance law for (\ref{KK}) is
\begin{equation*}
	U_0 = U_0^2 - V_0,\quad 2V_0 = \frac{1}{3} U_0^3,
\end{equation*}
which yields
\begin{equation}
\label{balance-KK}
	U_0 = 3 \pm \sqrt{3}, \quad V_0 = \frac{1}{6} U_0^3.
\end{equation}
In particular, we have two different solutions of the balance law, which correspond to different equilibria on the horizon inducing blow-up solutions in the forward time direction, as mentioned above.
Depending on the choice of $U_0$, the eigenstructure of the blow-up power-determining matrix \KMg{changes}.
\begin{description}
\item[Case 1. $U_0 = 3-\sqrt{3}$.] 
\end{description}
In this case, the blow-up power-determining matrix $A$ is
\begin{equation*}
A  = \begin{pmatrix}
5 - 2\sqrt{3} & -1 \\
12 - 6\sqrt{3} & -2
\end{pmatrix}.
\end{equation*}
The associated eigenpairs are
\begin{equation*}
\left\{1, \begin{pmatrix}
1 \\ 4-2\sqrt{3}
\end{pmatrix} \right\},\quad \left\{2-2\sqrt{3},  \begin{pmatrix}
1 \\ 3
\end{pmatrix} \right\}.
\end{equation*}
Note that the eigenvector $(1, 4-2\sqrt{3})^T$ associated with $\lambda = 1$ is consistent with Theorem \ref{thm-ev1}.

\begin{description}
\item[Case 2. $U_0 = 3+\sqrt{3}$.] 
\end{description}
In this case, the blow-up power-determining matrix $A$ is
\begin{equation*}
A = \begin{pmatrix}
5 + 2\sqrt{3} & -1 \\
12 + 6\sqrt{3} & -2
\end{pmatrix}.
\end{equation*}
The associated eigenpairs are
\begin{equation*}
\left\{1, \begin{pmatrix}
1 \\ 4+2\sqrt{3}
\end{pmatrix} \right\},\quad \left\{2+2\sqrt{3},  \begin{pmatrix}
1 \\ 3
\end{pmatrix} \right\}.
\end{equation*}
Note that the eigenvector $(1, 4+2\sqrt{3})^T$ associated with $\lambda = 1$ is consistent with Theorem \ref{thm-ev1}.

\subsubsection{Correspondence to desingularized vector fields: numerical study}

We shall \KMg{investigate the correspondence of information we have obtained above to dynamical information in the desingularized vector field} to confirm our results obtained in Section \ref{section-correspondence}.
The desingularized vector field associated with (\ref{KK}) is (cf. \cite{LMT2021, MT2020_1})
\begin{align}
\notag
\frac{dx_1}{d\tau} &= \frac{1}{4}\left(1 + 3p({\bf x})^4\right)(x_1^2 - x_2) - x_1G({\bf x}),\\
\label{KK-desing}
\frac{dx_2}{d\tau} &= \frac{1}{4}\left(1 + 3p({\bf x})^4\right)\left(\frac{1}{3}x_1^3 - (1-p({\bf x})^4)^2 x_1 \right) - 2x_2G({\bf x}),\\ 
\notag
p({\bf x}) &= (x_1^4 + x_2^2)^{1/4},\\ 
\notag
G({\bf x}) &= x_1^3(x_1^2 - x_2) + \frac{x_2}{2}\left(\frac{1}{3}x_1^3 - (1-p({\bf x})^4)^2 x_1 \right).
\end{align}
Note that the vector field $\frac{d{\bf x}}{d\tau} = g({\bf x})$ is polynomial of order $13$, while the original vector field $f$ is order at most $3$.
Our interest here is the following two equilibria on the horizon\footnote{
It is reported in \cite{Mat2018} that (\ref{KK-desing}) admits four equilibria on the horizon.
The remaining two equilibria induce blow-up solutions {\em in reverse time direction}.
In \cite{LMT2021}, {\em computer assisted proof}, equivalently {\em rigorous numerics} is applied to proving the existence of true \KMb{sink} $p_{\infty}^+$ and the true \KMb{saddle} $p_{\infty, s}^{+}$.
}:
\begin{equation*}
p_\infty^+ \approx (0.989136995894978, 0.206758557005181),\quad 
p_{\infty,s}^+ \approx (0.88610812897803, 0.61925794892101). 
\end{equation*}
The values $C_\ast$ corresponding to these equilibria are
\begin{equation*}
C_\ast^+ \equiv C_\ast(p_\infty^+) \approx 0.780107753370182,\quad C_{\ast,s}^+ \equiv C_\ast(p_{\infty,s}^+) \approx 0.187256681090721.
\end{equation*}
The parameter dependence associated with blow-up power eigenvalues is different among different $U_0$.
This difference reflects the dynamical property of corresponding equilibria on the horizon, as seen below and preceding works.
One root of the balance law is ${\bf Y}_0\equiv (U_0, V_0) = (3-\sqrt{3}, 9-5\sqrt{3})$. 
The corresponding scale parameter $r_{{\bf Y}_0}\equiv r_{{\bf Y}_0}^-$ is
\begin{align*}
r_{{\bf Y}_0}^- \equiv p({\bf Y}_0) &= \left(U_0^4 + V_0^2\right)^{1/4}\\
	&= \left\{ (3-\sqrt{3})^4 + (9-5\sqrt{3})^2 \right\}^{1/4}\\
	&= (408 - 234\sqrt{3})^{1/4}\\	
	&\approx \left\{2.70011102888\cdots \right\}^{1/4}\\	
	&\approx 1.2818741971,
\end{align*}
\KMa{where the functional $p({\bf y})$ is given in (\ref{func-p})}
It follows from direct calculation that
\begin{equation*}
r_{{\bf Y}_0}^- \approx \frac{1}{0.780107753370182},
\end{equation*}
which agrees with $1/C_\ast^+ = r_{p_{\infty}^+}$\KMg{($= r_{{\bf x}_\ast}$ in Theorem \ref{thm-balance-1to1})}, and implies the identity $r_{{\bf Y}_0}^- = r_{p_{\infty}^+}$ with $k=1$ mentioned in Theorem \ref{thm-balance-1to1}.
From (\ref{C-to-x}), \KMg{we have}
\begin{align}
\notag
(x_{\ast,1}, x_{\ast,2}) &\equiv \left(\frac{U_0}{\KMg{r_{{\bf Y}_0}^-}}, \frac{V_0}{\KMg{(r_{{\bf Y}_0}^-)^2}}\right) \\
\notag
	&= \left(\frac{3-\sqrt{3}}{(408 - 234\sqrt{3})^{1/4}}, \frac{9 - 5\sqrt{3}}{(408 - 234\sqrt{3})^{1/2}}\right) \\
\label{sink-KK-validated}
	&\approx \left(0.98913699589, 0.206758557\right),
\end{align}
which is indeed an equilibrium on the horizon $p_{\infty}^{+}$ for (\ref{KK-desing}).
\par
Similarly, consider another root of the balance law ${\bf Y}_0\equiv (U_0, V_0) = (3+\sqrt{3}, 9+5\sqrt{3})$. 
The corresponding scale parameter $r_{{\bf Y}_0}\equiv r_{{\bf Y}_0}^+$ is
\begin{align*}
r_{{\bf Y}_0}^+ &= \left(U_0^4 + V_0^2\right)^{1/4}\\
	&= \left\{ (3+\sqrt{3})^4 + (9+5\sqrt{3})^2 \right\}^{1/4}\\
	&= \left\{ (12 + 6\sqrt{3})^2 + (9+ 5\sqrt{3})^2 \right\}^{1/4}\\
	&= \left\{ (144 + 108 + 144\sqrt{3}) + (81 + 75 + 90\sqrt{3}) \right\}^{1/4}\\	
	&= (408 + 234\sqrt{3})^{1/4}\\	
	&\approx 5.34026339768.	
\end{align*}
It follows from direct \KMg{calculations} that
\begin{equation*}
r_{{\bf Y}_0}^+ \approx \frac{1}{0.187256681090721},
\end{equation*}
which agrees with $1/C_{\ast,s}^+ = r_{p_{\infty,s}^+}$, and implies the identity $r_{{\bf Y}_0}^+ = r_{p_{\infty,s}^+}$ with $k=1$ similar to the case \KMg{of} $p_{\infty}^+$.
From (\ref{C-to-x}),
\begin{align}
\notag
(x_{\ast,1}, x_{\ast,2}) &\equiv \left(\frac{U_0}{\KMg{r_{{\bf Y}_0}^+}}, \frac{V_0}{\KMg{(r_{{\bf Y}_0}^+)^2}}\right) \\
\notag
	&= \left(\frac{3+\sqrt{3}}{(408 + 234\sqrt{3})^{1/4}}, \frac{9+5\sqrt{3}}{(408 + 234\sqrt{3})^{1/2}}\right) \\
\notag
	&\approx \left(\frac{4.73205080757}{5.34026339768}, \frac{17.6602540378}{28.5184131566}\right)\\
\label{saddle-KK-validated}
	&\approx \left(0.88610812897, 0.61925754917\right),
\end{align}
which is indeed an equilibrium on the horizon $p_{\infty,s}^{+}$ for (\ref{KK-desing}).
\par
Next eigenvalues of the Jacobian matrices $Dg(p_{\infty}^+)$ and $Dg(p_{\infty,s}^+)$ are computed, respectively:
\begin{align*}
{\rm Spec}(Dg(p_{\infty}^+)) &= \{ -0.780107753370184,  -1.142157021690769\},\\
{\rm Spec}(Dg(p_{\infty, s}^+)) &= \{ -0.187256681090720, 1.023189533593166\}.
\end{align*}
It immediately follows from the above calculations that $-C_\ast^+$ and $-C_{\ast,s}^+$ are eigenvalues of $Dg(p_{\infty}^+)$ and $Dg(p_{\infty,s}^+)$, respectively.
It also follows that another eigenvalues satisfy the following identity, which \KMg{is} consistent with Theorem \ref{thm-ev1}:
\begin{align*}
\frac{2-2\sqrt{3}}{r_{{\bf Y}_0}^-} &= \frac{2-2\sqrt{3}}{(408 - 234\sqrt{3})^{1/4} } \approx -1.142157021690769,\\
\frac{2+2\sqrt{3}}{r_{{\bf Y}_0}^+} &= \frac{2+2\sqrt{3}}{(408 + 234\sqrt{3})^{1/4} } \approx 1.023189533593166.
\end{align*}
In particular, \KMa{$p_{\infty}^+$} is a sink admitting two-dimensional stable manifold for the desingularized vector field, while $p_{\infty, s}^{+}$ is a saddle admitting one-dimensional stable manifold.
Both stable manifolds admit nonempty intersections with the interior of compactified phase space \KMb{$\overline{\mathcal{D}} = \{p({\bf x}) \leq 1\}$}, which follows from the inequality $-C_{\ast} < 0$. 
\par
\bigskip
Finally we investigate the correspondence of eigenvectors associated with the above eigenvalues.
Eigenpairs of $Dg(p_{\infty}^+)$ are
\begin{equation}
\label{epair-KK-sink}
\left\{ -C_\ast^+, \begin{pmatrix}
0.922620374008554 \\
0.385709275833906
\end{pmatrix}\right\},\quad \left\{ \frac{2-2\sqrt{3}}{r_{{\bf Y}_0}^-}, 
\begin{pmatrix}
-0.1062185325758257\\
0.9943428097680588
\end{pmatrix}\right\},
\end{equation}
where we have used the identity of eigenvalues derived above.
It is easily checked that
\begin{equation*}
\frac{0.922620374008554}{0.385709275833906} \approx \frac{0.989136995894978}{2\times 0.206758557005181} \approx \frac{(p_\infty^+)_1}{2(p_\infty^+)_2},
\end{equation*}
which is consistent with Theorem \ref{thm-ev1} about eigenvectors associated with the eigenvalue $-C_\ast$.
To check the correspondence of another eigenvectors, we consider the projection $P_\ast$ defined at ${\bf x}_\ast\in \mathcal{E}$ onto ${\rm span}\{{\bf v}_{\ast, \alpha}\}$, which is calculated as
\begin{align*}
P_\ast =\frac{1}{2} \left. \begin{pmatrix}
x_1\\ 2x_2
\end{pmatrix}\begin{pmatrix}
2x_1^3 \\ x_2
\end{pmatrix}^T \right|_{{\bf x} = p_{\infty}^+} = \begin{pmatrix}
x_1^4 & \frac{1}{2}x_1x_2\\
2x_1^3 x_2 & x_2^2
\end{pmatrix}_{{\bf x} = {\bf x}_\ast}
\end{align*}
and hence the projection $I-P_{\ast}$ is
\begin{equation*}
I-\KMg{P_\ast} = \begin{pmatrix}
1 - x_1^4 & -\frac{1}{2}x_1x_2\\
-2x_1^3 x_2 & 1 - x_2^2
\end{pmatrix}_{{\bf x} = {\bf x}_\ast}.
\end{equation*}
Letting $P_{\infty}^+$ be the projection $P_{\ast}$ with ${\bf x}_\ast = p_\infty^+$, direct calculations yield
\begin{align*}
(I-P_\infty^+) \begin{pmatrix}
1 / r_{{\bf Y}_0}^- \\
3 / (r_{{\bf Y}_0}^-)^2 \\
\end{pmatrix} &\approx \begin{pmatrix}
0.04274910089486506 & -0.1022562689758424 \\
-0.4001868606922550 & +0.9572508991051355 
\end{pmatrix} \begin{pmatrix}
1/ (408 - 234\sqrt{3})^{1/4} \\
3/ (408 - 234\sqrt{3})^{1/2}
\end{pmatrix}\\
&\approx \begin{pmatrix}
-0.1533408070194538 \\
1.435468229567002
\end{pmatrix},
\end{align*}
where the vector $(1,3)^T$ is the eigenvector of $A$ associated with the eigenvalue $2-2\sqrt{3}$.
Then we obtain
\begin{equation*}
\frac{1.435468229567002}{-0.1533408070194538} \approx \frac{0.9943428097680588}{-0.1062185325758257},
\end{equation*}
where the latter ratio is calculated from the eigenvector associated with the eigenvalue $\frac{2-2\sqrt{3}}{r_{{\bf Y}_0}^-}$ shown in (\ref{epair-KK-sink}).
The correspondence of eigenvectors stated in Theorem \ref{thm-ev1} is therefore confirmed for $p_\infty^+$.
\par
Similarly, we know that eigenpairs of $Dg(p_{\infty,s}^+)$ are
\begin{equation}
\label{epair-KK-saddle}
\left\{ -C_{\ast.s}^+, \begin{pmatrix}
0.581870201791748 \\
0.813281666009280
\end{pmatrix}\right\},\quad \left\{ \frac{2+2\sqrt{3}}{r_{{\bf Y}_0}^+}, \begin{pmatrix}
0.4065790070098367\\ 
-0.9136156254458955
\end{pmatrix}\right\},
\end{equation}
where we have used the identity of eigenvalues derived above.

Letting $P_{\infty,s}^+$ be the projection $P_{\ast}$ with ${\bf x}_\ast = p_{\infty,s}^+$, direct calculations yield
\begin{align*}
(I-P_{\infty,s}^+) \begin{pmatrix}
1 / r_{{\bf Y}_0}^+ \\
3 / (r_{{\bf Y}_0}^+)^2 \\
\end{pmatrix} &\approx \begin{pmatrix}
0.3834804073018562 & -0.2743647512365853 \\
-0.8617112200159814 & 0.6165195926981430
\end{pmatrix} \begin{pmatrix}
1/ (408 + 234\sqrt{3})^{1/4} \\
3/ (408 + 234\sqrt{3})^{1/2}
\end{pmatrix}\\
&\approx \begin{pmatrix}
-0.2017528727505380  \\
0.4533548802214182
\end{pmatrix},
\end{align*}
where the vector $(1,3)^T$ is the eigenvector of $A$ associated with the eigenvalue $2+2\sqrt{3}$.
Then we obtain
\begin{equation*}
\frac{-0.2017528727505380}{0.4533548802214182} \approx \frac{0.4065790070098367}{-0.9136156254458955},
\end{equation*}
where the latter ratio is calculated from the eigenvector associated with the eigenvalue $\frac{2+2\sqrt{3}}{r_{{\bf Y}_0}^+}$ shown in (\ref{epair-KK-sink}).
The correspondence of eigenvectors stated in Theorem \ref{thm-ev1} is therefore confirmed for $p_{\infty,s}^+$.

\subsection{An artificial system in the presence of Jordan blocks}
\label{section-ex-log}

The next example \KMg{concerns} with an artificial system \KMg{such that} the blow-up power-determining matrix has a non-trivial Jordan block.
\KMa{In \cite{asym1}, asymptotic expansions of blow-up solutions are calculated assuming their existence. 
Here we investigate if blow-up solutions of the systems we are interested in here indeed exist, as well as correspondences of associated eigenstructures stated in Section \ref{section-correspondence}.}

\subsubsection{The presence of terms of order $k+\alpha_i - 1$}
First we consider
\begin{equation}
\label{log}
	u' = u^2 + v, \quad v' = au^3 + 3uv - u^2\KMg{,}
\end{equation}
where $a\in\mathbb{R}$ is a parameter.
This system is asymptotically quasi-homogeneous of type $\alpha = (1,2)$ and order $k+1=2$, 
consisting of the quasi-homogeneous part $f_{\alpha, k}$ and the lower-order part $f_{\rm res}$ given as follows: 
\begin{equation}
\label{log-vf}
	f_{\alpha, k}(u,v) = \begin{pmatrix}
u^2 + v\\
 au^3 + 3uv
 \end{pmatrix}, \quad f_{\mathrm{res}}(u,v) = \begin{pmatrix}
 0\\
 -u^2	
 \end{pmatrix}.
\end{equation}
To verify the existence of blow-up solutions, we investigate the dynamics at infinity.
To this end, we apply the parabolic compactification
\begin{equation*}
u = \kappa x_1,\quad v = \kappa^2 x_2,\quad \kappa = (1-p({\bf x})^4)^{-1},\quad p({\bf x})^4 = x_1^4 + x_2^2
\end{equation*}
and the time-scale desingularization
\begin{equation*}
d\tau = \frac{1}{4}(1-p({\bf x})^4)^{-1}\left( 1 + 3 p({\bf x})^4 \right)^{-1}dt
\end{equation*}
to (\ref{log}) and we have the desingularized vector field
\begin{align}
\label{desing-log}
\begin{aligned}
\dot x_1 &= \frac{1}{4}\left(1+ 3p({\bf x})^4 \right)\left( x_1^2 + x_2 \right) - x_1 G_a({\bf x}),\\
\dot x_2 &= \frac{1}{4}\left(1+ 3p({\bf x})^4 \right)\left(  ax_1^3 + 3x_1x_2 - \kappa^{-1}x_1^2\right) - 2 x_2 G_a({\bf x}),
\end{aligned}
\end{align}
where
\begin{equation*}
G_a({\bf x}) = x_1^3 \left( x_1^2 + x_2 \right) + \frac{1}{2}x_2 \left( ax_1^3 + 3x_1x_2 - \kappa^{-1}x_1^2 \right).
\end{equation*}
We pay attention to the case $a=0$.
Then equilibria on the horizon satisfy
\begin{equation*}
x_1^2 + x_2  - x_1 C = 0,\quad 3x_1x_2  - 2 x_2 C = 0,\quad 
C = x_1^3 \left( x_1^2 + x_2 \right) + \frac{3}{2} x_1x_2^2.
\end{equation*}
One easily find an equilibrium on the horizon $(x_1, x_2) = (1,0) \equiv {\bf x}_\ast$.
The value $C_\ast$ given in (\ref{const-horizon}) is $C_\ast = 1$.
The Jacobian matrix of the vector field (\ref{desing-log}) with $a=0$ at $(x_1, x_2)$ is
\begin{align*}
J({\bf x}) &= \begin{pmatrix}
J_{11} & J_{12}  \\
J_{21} & J_{22}
\end{pmatrix},\\
J_{11} &= 3x_1^3(x_1^2 + x_2) + \frac{1}{2}x_1 R({\bf x}) - G_0({\bf x}) - x_1 \frac{\partial G_0}{\partial x_1}({\bf x}),\\
J_{12} &= \frac{3}{2}x_2(x_1^2 + x_2) + \frac{1}{4}R({\bf x}) - x_1 \frac{\partial G_0}{\partial x_2}({\bf x}),\\
J_{21} &= 3x_1^3(3x_1x_2 - \kappa^{-1}x_1^2) + \frac{1}{4} R({\bf x})(3x_2 +4x_1^5 - 2\kappa^{-1}x_1 )  - \KMd{2x_2} \frac{\partial G_0}{\partial x_1}({\bf x}),\\
J_{22} &= \frac{3}{2}x_2(3x_1x_2 - \kappa^{-1}x_1^2) + \frac{1}{4}R({\bf x})(3x_1 + 2x_1^2 x_2 )  - 2G_0({\bf x}) - 2x_2 \frac{\partial G_0}{\partial x_2}({\bf x}),
\end{align*}
where $R({\bf x}) = 1+3p({\bf x})^4$ and
\begin{align*}
\frac{\partial G_0}{\partial x_1}({\bf x}) &= 5x_1^4 + 3x_1^2x_2 + \frac{3}{2}x_2^2 - \kappa^{-1}x_1x_2 + 2x_1^5 x_2,\\
\frac{\partial G_0}{\partial x_2}({\bf x}) &= x_1^3 + 3x_1x_2 - \frac{1}{2}\kappa^{-1}x_2^2 + x_1^2 x_2^2.
\end{align*}
Substituting $(x_1, x_2) = {\bf x}_\ast$ into $J({\bf x})$, we have
\begin{equation*}
J({\bf x}_\ast) = \begin{pmatrix}
3 + 2 - 1 - 5 & 1 - 1 \\
0 + (0+4-0) - \KMd{0} & 0 + 3 - 2 
\end{pmatrix} = \begin{pmatrix}
- 1 & 0 \\
\KMd{4} & 1
\end{pmatrix},
\end{equation*}
which implies that the equilibrium ${\bf x}_\ast$ is the hyperbolic saddle.
Moreover, the eigenvector associated with the eigenvalue $-1$ is $\KMd{(1,-2)^T}$, while the eigenvector associated with the eigenvalue $+1$ is $(0,1)^T$.
The latter is tangent to the horizon at ${\bf x}_\ast$.
As a consequence, the stable manifold of ${\bf x}_\ast$ is extended inside $\mathcal{D}$ and hence the local stable manifold $W^s_{\rm loc}({\bf x}_\ast)$ induces (finite-time) blow-up solutions of (\ref{log}) with $a=0$ in forward time direction.
Note that the eigenstructure of $J({\bf x}_\ast)$ is {\em not contradictory} to Theorem \ref{thm-ev1} \KMg{because} there {\em is} a nonzero term of order $k+\alpha_2 - 1 = 2$; $-u^2$, in the second component of (\ref{log}).
In contrast, we cannot determine whether $v$ blows up at $t = t_{\max}$ yet, \KMg{because} $\KMb{x_{2;\ast}} = 0$.
This consequence is not contradictory to Theorem \ref{thm:blowup}, either.
In other words, subsequent terms must be investigated to determine the asymptotic behavior of $v(t)$ as $t \to t_{\max}$.
\par
\KMa{
Asymptotic expansions calculated in \cite{asym1} indeed clarify this ambiguity.
Now introduce the asymptotic expansion}
\begin{align}
\notag
u(t) &= \theta(t)^{-1}U(t) \equiv \theta(t)^{-1}\sum_{n=0}^\infty U_n(t),\quad U_n(t) \ll U_{n-1}(t),\quad \lim_{t\to t_{\max}}U_n(t) = U_0, \\
\label{series-log}
v(t) &= \theta(t)^{-2}V(t) \equiv \theta(t)^{-2}\sum_{n=0}^\infty V_n(t),\quad V_n(t) \ll V_{n-1}(t), \quad \lim_{t\to t_{\max}}V_n(t) = V_0.
\end{align}
As seen in \cite{asym1}, the balance law is
\begin{equation*}
\begin{pmatrix}
U_0 \\ 2V_0
\end{pmatrix} = \begin{pmatrix}
U_0^2 + V_0\\
aU_0^3 + 3U_0V_0
 \end{pmatrix}.
\end{equation*}
Our particular interest here is the case $a=0$, in which case the root is $(U_0, V_0) = (1,0)$, which is consistent with \KMg{${\bf x}_\ast$}.
We fix $a=0$ for a while.
The blow-up power-determining matrix at $(U_0, V_0)$ is
\begin{equation}
\label{blow-up-A-Jordan}
A = \begin{pmatrix}
-1 & 0 \\ 0 & -2
\end{pmatrix} + \begin{pmatrix}
2U_0 & 1\\
3V_0 & 3U_0
 \end{pmatrix} = \begin{pmatrix}
1& 1\\
0 & 1
 \end{pmatrix},
\end{equation}
that is, the matrix $A$ has nontrivial Jordan block.
The eigenvector associated with the double eigenvalue $\lambda = 1$ is $(1,0)^T$, while the vector $(0,1)^T$ is the generalized eigenvector.
In the present case, the latter corresponds to the eigenvector $J({\bf x}_\ast)$ associated with the eigenvalue $+1$ (not $-1$ !).
We see the common eigenstructure stated in Theorem \ref{thm-ev1}.
Note again that the present situation is {\em not} contradictory to Proposition \ref{prop-correspondence-ev-Ag-A} because Assumption \ref{ass-f} is not satisfied in the present example.
We finally obtain the following second order asymptotic expansion of the blow-up solution $(u(t), v(t))$ as $t\to t_{\max}-0$:
\begin{equation*}
u(t) \sim \theta(t)^{-1} - \frac{1}{4} - \KMd{\frac{1}{144}}\theta(t),\quad v(t) \sim \frac{1}{2}\theta(t)^{-1} - \frac{1}{24}.
\end{equation*}
\KMa{We see that $v(t)$ blows up, while the rate is smaller than $\theta(t)^{-2}$ expected in Theorem \ref{thm:blowup}. }

\subsubsection{The absence of terms of order $k+\alpha_i - 1$}

Next we consider
\begin{equation}
\label{log2}
	u' = u^2 + v, \quad v' = au^3 + 3uv - u
\end{equation}
with a real parameter $a\in\mathbb{R}$, instead of (\ref{log}).
The difference from (\ref{log}) is the replacement of $-u^2$ by $-u$ in the second component of the vector fields.
Because the quasi-homogeneous part is unchanged, \KMg{the balance law and blow-up power-determining matrix are the same} as those of (\ref{log}), whereas the Jacobian matrix $J({\bf x}_\ast)$ \KMg{for the desingularized vector field} changes due to the absence of terms of the order $k+\alpha_i-1$.
The desingularized vector field is
\begin{align}
\label{desing-log-2}
\begin{aligned}
\dot x_1 &= \frac{1}{4}\left(1+ 3p({\bf x})^4 \right)\left\{ x_1^2 + x_2 \right\} - x_1 \tilde G_a({\bf x}),\\
\dot x_2 &= \frac{1}{4}\left(1+ 3p({\bf x})^4 \right)\left\{  ax_1^3 + 3x_1x_2 - \kappa^{-2}x_1\right\} - 2 x_2 \tilde G_a({\bf x}),
\end{aligned}
\end{align}
where
\begin{equation*}
\tilde G_a({\bf x}) = x_1^3 \left\{ x_1^2 + x_2 \right\} + \frac{1}{2}x_2 \left\{ ax_1^3 + 3x_1x_2 - \kappa^{-2}x_1 \right\}.
\end{equation*}
\KMg{Under the constraint $a=0$,} one easily \KMg{finds} an equilibrium on the horizon $(x_1, x_2)^T = (1,0)^T \equiv {\bf x}_\ast$.
The Jacobian matrix of the vector field (\ref{desing-log-2}) with $a=0$ at ${\bf x}_\ast$ is now
\begin{align*}
\KMa{J({\bf x}_\ast)} &= \begin{pmatrix}
J_{11} & J_{12}  \\
J_{21} & J_{22}
\end{pmatrix},\\
J_{11} &= \left(3x_1^3(x_1^2 + x_2) + \frac{1}{2}x_1 R({\bf x}) - \tilde G_0({\bf x}) - x_1 \frac{\partial \tilde G_0}{\partial x_1}({\bf x}) \right)_{{\bf x} = {\bf x}_\ast} = -1,\\
J_{12} &= \left(\frac{3}{2}x_2(x_1^2 + x_2) + \frac{1}{4}R({\bf x}) - x_1 \frac{\partial \tilde G_0}{\partial x_2}({\bf x})\right)_{{\bf x} = {\bf x}_\ast} = 0,\\
J_{21} &= \left( 3x_1^3(3x_1x_2 - \kappa^{-2}x_1) + \frac{1}{4} R({\bf x})(3x_2 + 8x_1^4\kappa^{-1} - \kappa^{-2} )  - 2x_2 \frac{\partial \tilde G_0}{\partial x_1}({\bf x}) \right)_{{\bf x} = {\bf x}_\ast} = 0,\\
J_{22} &= \left( \frac{3}{2}x_2(3x_1x_2 - \kappa^{-2}x_1) + \frac{1}{4}R({\bf x})(3x_1 + 4x_1 x_2\kappa^{-1} )  - 2\tilde G_0({\bf x}) - 2x_2 \frac{\partial \tilde G_0}{\partial x_2}({\bf x}) \right)_{{\bf x} = {\bf x}_\ast} = 1,
\end{align*}
where $R({\bf x}) = 1+3p({\bf x})^4$ and
\begin{align*}
\frac{\partial \tilde G_0}{\partial x_1}({\bf x}) &= 5(x_1^2 + 3x_2)x_1^2 + \frac{3}{2}x_2^2 + 4x_1^4 x_2 \kappa^{-1} - \frac{1}{2}x_2\kappa^{-1},\\
\frac{\partial \tilde G_0}{\partial x_2}({\bf x}) &= x_1^3 + \frac{3}{2}x_1x_2 - \frac{1}{2}\kappa^{-2}x_1 + \frac{3}{2}x_1x_2 + 2x_1x_2^2 \kappa^{-1}.
\end{align*}
\KMa{
We therefore know that the equilibrium ${\bf x}_\ast$ is the hyperbolic saddle and that the eigenvector associated with the eigenvalue $-1$ is ${\bf v}_{0,\alpha} = (1,0)^T$, while the eigenvector associated with the eigenvalue $+1$ is $(0,1)^T$.
Now the matrix $A_g$ stated in Theorem \ref{thm-evec-Dg} is 
\begin{equation*}
A_g \equiv \begin{pmatrix}
1 & 1 \\ 0 & 1
\end{pmatrix} = A.
\end{equation*}
In particular, the eigenstructure of $A_g$ is exactly the same as that of $A$.
Therefore the correspondence of eigenstructure stated in Theorem \ref{thm-evec-Dg} is considered between those of $A$ and $J({\bf x}_\ast)$.
Now the blow-up power-determining matrix $A$ for (\ref{log2}) is the same as (\ref{blow-up-A-Jordan}).
As seen in the previous example, the vector $(0,1)^T$ is the generalized eigenvector of $A$ associated with the {\em double} eigenvalue $\lambda = 1$.
It turns out here that the vector $(0,1)^T$ is the eigenvector of $J({\bf x}_\ast)$ associated with the {\em simple} eigenvalue $+1$.
The gap of multiplicity is exactly what we have stated in Theorem \ref{thm-evec-Dg} with $m_\lambda = 2$ and $m_{\lambda_g} = 1$.
Indeed, letting ${\bf w} = (0,1)^T$ and ${\bf w}_g = (0,1)^T$, we have
\begin{align*}
{\bf v}_{\ast, \alpha} &= \begin{pmatrix}
1 \\ 0
\end{pmatrix},\quad \nabla p({\bf x}_\ast) =  \begin{pmatrix}
1 \\ 0
\end{pmatrix} \quad \Rightarrow \quad I-P_\ast = \begin{pmatrix}
0 & 0  \\ 0 & 1
\end{pmatrix},\\
B_g &= -P_\ast (A_g + C_\ast I) = -\begin{pmatrix}
1 & 0  \\ 0 & 0
\end{pmatrix}\left\{ \begin{pmatrix}
1 & 1  \\ 0 & 1
\end{pmatrix} + \begin{pmatrix}
1 & 0  \\ 0 & 1
\end{pmatrix}\right\},\\
A_g - kC_\ast I &= \begin{pmatrix}
1 & 1  \\ 0 & 1
\end{pmatrix} - \begin{pmatrix}
1 & 0  \\ 0 & 1
\end{pmatrix} = \begin{pmatrix}
0 & 1  \\ 0 & 0
\end{pmatrix},\quad  (A_g - kC_\ast I)^2 =  \begin{pmatrix}
0 & 0  \\ 0 & 0
\end{pmatrix},\\
(I-P_\ast){\bf w} &= \begin{pmatrix}
0 \\ 1
\end{pmatrix}\equiv {\bf w}_g, \quad (A_g - kC_\ast I){\bf w}_g = \begin{pmatrix}
0 \\ 1
\end{pmatrix} \equiv {\bf w},\\
Dg({\bf x}_\ast) - kC_\ast I &= J({\bf x}_\ast) - kC_\ast I =  \begin{pmatrix}
-1 & 0  \\ 0 & 1
\end{pmatrix} - \begin{pmatrix}
1 & 0  \\ 0 & 1
\end{pmatrix} = \begin{pmatrix}
-2 & 0  \\ 0 & 0
\end{pmatrix}.
\end{align*}
Obviously, we see that $(I-P_\ast){\bf w} \in \ker(Dg({\bf x}_\ast) - kC_\ast I )$ and $(A_g - kC_\ast I) = {\bf v}_{\ast,\alpha} \in \ker((A_g - kC_\ast I))$.
}
\par
In \cite{asym1}, the following asymptotic expansion for the blow-up solution for (\ref{log2}) is calculated:
\begin{equation*}
u(t) \sim \theta(t)^{-1} - \frac{1}{9} \theta(t),\quad v(t) \sim \frac{1}{3}.
\end{equation*}
\KMa{Unlike the system (\ref{log}), we see that $v(t)$ remains bounded as $t\to t_{\max}$, which is not  expected in Theorem \ref{thm:blowup}. 
This result as well as that in the previous example extract the importance of multi-order asymptotic expansions for blow-up solutions such that the concrete asymptotic behavior cannot be clearly described from Theorem \ref{thm:blowup}.}

\begin{rem}\rm
Persistence of hyperbolicity under perturbations of vector fields yields the existence of hyperbolic equilibria on the horizon for $a\not = 0$ sufficiently close to $0$, in which case ${\bf x} = (1,0)^T$ is not an equilibrium and the new equilibrium with $a\not = 0$ possesses nonzero second component in general.
In particular, the blow-up rate $O(\theta(t)^{-2})$ in $v(t)$ becomes active under such perturbations, according to Theorem \ref{thm:blowup}.
This implies that the blow-up rate can depend on vector fields in a {\em discontinuous} manner, \KMg{even if equilibria or general invariant sets on the horizon depend {\em continuously} on parameters}.
\end{rem}

\section*{Concluding Remarks}

In this paper, we have provided the correspondence of \KMb{coefficients characterizing the leading terms of blow-ups and} eigenstructure between the system associated with asymptotic expansions of blow-ups proposed in Part I \cite{asym1}, and desingularized vector fields through compactifications and time-scale desingularizations.
We have shown that equilibria of these transformed systems, introduced to describe type-I finite-time blow-ups in forward time direction for the original system, correspond one-to-one to each other, and that there is a natural correspondence of eigenstructures between Jacobian matrices at the above equilibria.
As a consequence, both equilibria are hyperbolic once either of them turns out to be hyperbolic and, as a corollary, hyperbolic structure of the leading terms of asymptotic expansions of blow-ups, constructed assuming their existence, indeed guarantees their existence.
In particular, in the case of ODEs with asymptotically quasi-homogeneous vector fields, 
\KMb{asymptotic expansions} of type-I blow-ups themselves provide their dynamical properties, including their existence, and vice versa.
Such correspondence can be seen in various examples presented in Part I \cite{asym1}.
We believe that the correspondence obtained in the present paper \KMc{provides} a significant insight into blow-up studies for a wide class of differential equations.


\par
\bigskip
We end this paper by providing comments related to the present study.

\subsection*{Asymptotic expansion of complex blow-ups such as oscillatory blow-ups}
According to \cite{Mat2018} and \cite{Mat2019}, blow-up behavior (for asymptotically quasi-homogeneous systems) can be understood through the dynamical structure of {\em invariant sets} on the horizon for defingularized vector fields.
In the present study, we have focused {\em only on stationary blow-ups}, namely blow-ups induced by equilibria on the horizon for desingularized vector fields, to discuss multi-order asymptotic expansions.
It is then natural to question how multi-order asymptotic expansions of blow-ups induced by {\em general invariant sets} on the horizon behave.
As concrete examples, {\em oscillatory blow-up behavior} is studied originated from suspension bridge problems in mechanical engineering (e.g. \cite{GP2011, GP2013, HBT1989, LM1990}).
Later this complex behavior is extracted by means of {\em local stable manifolds of periodic orbits} with computer-assisted proofs or {\em rigorous numerics} (\cite{DALP2015}).
Independently, it is proposed in \cite{Mat2018} that such oscillatory blow-up behavior is characterized by periodic orbits on the horizon for desingularized vector fields, which is referred to as {\em periodic blow-up}.
From the viewpoint of complete understandings of blow-up solutions themselves (for ODEs), asymptotic expansions of periodic blow-up solutions are suitable issues as the sequel to the present study.
\par
It should be noted that, in \cite{DALP2015}, oscillatory blow-up behavior is extracted through validation of {\em unstable} periodic orbits, where the ansatz of oscillatory blow-ups like those mentioned in Assumption \ref{ass-fundamental} is assumed to consider its dynamical property.
On the other hand, such an approach can extract wrong stability property about the original blow-up behavior, as indicated in Theorem \ref{thm-stability} in the case of stationary blow-ups.
The preceding works and the present study motivate to study the true correspondence of dynamical properties between desingularized vector fields $g$ and the counterpart of the system (\ref{blow-up-basic}) to periodic blow-ups\KMf{,} and general complex blow-up behavior.

\subsection*{\KMh{Infinite-dimensional problems}}
In the field of partial differential equations (PDEs for short), the {\em rescaling} of functions for (parabolic-type) equations is widely used for studying asymptotic behavior near finite-time singularity (e.g., behavior of time-space dependent solutions $u(t,x)$ with $t < t_{\max} < \infty$ near $t_{\max}$).
In many studies of blow-ups for parabolic-type PDEs, the rescaling is applied assuming that the type-I blow-up occurs (e.g. \cite{GK1985, HV1993, MZ1998}).
Our present approach of asymptotic expansion of blow-up solutions introduced in \KMb{Part I \cite{asym1}} begins with the ansatz of blow-up profiles with {\em type-I blow-up rates}, which turns out to be similar to typical approaches applied to PDEs mentioned above.
Therefore the system of our interest (\ref{blow-up-basic}) has the similar form to that for rescaled profiles of blow-ups, which are referred to as {\em backward self-similar profiles}, for \KMh{parabolic-type} PDEs (see references mentioned above).
Except a special case where the system itself is scale invariant (corresponding to the quasi-homogeneity in our setting), blow-up solutions are typically {\em assumed} to exist for studying their asymptotic behavior.
In other words, their existence is discussed independently through {\em nonlinear} and/or expensive analysis.
In contrast, one of our results, Theorem \ref{thm-existence-blow-up}, shows that, under mild assumptions for vector fields, {\em linear} information associated with asymptotic profiles provides the existence of blow-up solutions as well as their dynamical properties, even if the existence is not known a priori.
The present results discussed in this paper will provide a new insight into the characterization of blow-up behavior for evolutionary equations, including infinite-dimensional ones such as parabolic-type PDEs.
\KMh{
Our characterizations of blow-up solutions rely on the approach reviewed in Section \ref{section-preliminary}. 
Once its infinite-dimensional analogue is constructed, the dynamical correspondence of blow-up solutions presented in this paper can be extended to infinite-dimensional problems, although there are significant difficulties to be overcome, as mentioned in \cite{Mat2019}. 
}

\subsection*{Presence of logarithmic terms}
Observations in Section \ref{section-ex-log} indicate that logarithmic terms in asymptotic expansions can be present {\em only if} a {\em negative} blow-up power eigenvalue $\lambda$ generates a nontrivial Jordan block of the blow-up power-determining matrix $A$.
To this end, the vector field (\ref{ODE-original}) must have the dimension more than two, \KMg{because} the matrix $A$ always possess the eigenvalue $1$.
In other words, in {\em two}-dimensional systems, the eigenvalue $1$ is the only one admitting the nontrivial Jordan block for the matrix $A$\KMg{,} and this is the only case of the presence of logarithmic functions in the fundamental matrix for rational vector fields with rational blow-up rates (the first terms of blow-ups).
We therefore leave the following conjecture.
\begin{itemize}
\item In any two dimensional rational vector fields, all possible blow-up solutions satisfying Assumption \ref{ass-fundamental} includes {\em no logarithmic functions of $\theta(t)$} in their finite-order asymptotic expansions.
\end{itemize}

\subsection*{Efficient rigorous numerics of local stable manifolds of equilibria at infinity}

We shall connect the preceding works, {\em computer-assisted proofs} or {\em rigorous numerics} of blow-up solutions with concrete and rigorous bounds of blow-up times, to the present study.
In \cite{MT2020_1, TMSTMO2017}, computer-assisted proofs for the existence of blow-up solutions\KMg{, their concrete profiles, and their} blow-up times are provided, which are based on desingularized vector fields associated with (admissible global) compactifications and time-scale desingularizations.
As seen in several examples in the present paper, these compactifications cause the significant increase of degree of polynomials depending on the type $\alpha$ of the original vector field $f$, even if $f$ consists of polynomials of low degree.
In general, admissible global compactifications require lengthy calculations of vector fields themselves, as well as their equilibria, linearized matrices or {\em parameterization of invariant manifolds} (cf. \cite{LMT2021} \KMg{for constructing stable manifolds in an efficient way}).
On the other hand, the proposed methodology for asymptotic expansions discussed in Part I \cite{asym1} indicate that several essential dynamical information of equilibria on the horizon can be extracted from the simpler system (\ref{blow-up-basic}), which can contribute to {\em efficient} validation methodology of blow-ups through simplification of computing objects.
We have seen in Part I and the present paper that equilibria on the horizon and eigenstructures for desingularized vector fields are shown to be calculated in an efficient way through asymptotic expansions. 
The next issue for problems mentioned here is the efficient computation of {\em parameterization of invariant manifolds}, which provides mappings describing invariant manifolds as the graphs as well as conjugate relation between the original nonlinear dynamics on them and the simpler, in particular {\em linearized} one.
This concept is originally developed in \cite{CFdlL2003-1, CFdlL2003-2, CFdlL2005} and there are many successful applications to describe nonlinear invariant dynamics with computer assistance.
An application to blow-up validation is shown in \cite{LMT2021}.
Using the knowledge of asymptotic expansions, one expects that invariant manifolds describing a family of blow-up solutions can be also constructed in more efficient way than desingularized vector fields, as achieved in calculations of equilibria and eigenstructures.
It should be noted, however, that the system (\ref{blow-up-basic}) itself can extract wrong dynamical properties of blow-up solutions due to the intrinsic difference of eigenvalue distribution among two systems of our interest \KMg{(Theorem \ref{thm-stability})}.
The present study implies that careful treatments are required to validate geometric and dynamical properties of blow-up solutions in an efficient way.

\section*{Acknowledgements}
The essential ideas in the present paper are inspired in Workshop of Unsolved Problems in Mathematics 2021 sponsored by Japan Science and Technology Agency (JST). 
All authors appreciate organizers and sponsors of the workshop for providing us with an opportunity to create essential ideas of the present study.
KM was partially supported by World Premier International Research Center Initiative (WPI), Ministry of Education, Culture, Sports, Science and Technology (MEXT), Japan.
\KMb{KM and AT were} JSPS Grant-in-Aid for Scientists (B) (No. JP21H01001).
\bibliographystyle{plain}
\bibliography{blow_up_asymptotic}

\begin{thebibliography}{10}

\bibitem{AM1989}
M.~Adler and P.~van Moerbeke.
\newblock The complex geometry of the {K}owalewski-{P}ainlev{\'e} analysis.
\newblock {\em Inventiones Mathematicae}, 97(1):3--51, 1989.

\bibitem{A2002}
B.~Andrews.
\newblock Singularities in crystalline curvature flows.
\newblock {\em Asian J. Math.}, 6(1):101--122, 2002.

\bibitem{asym1}
T.~Asai, H.~Kodani, K.~Matsue, H.~Ochiai, and T.~Sasaki.
\newblock Multi-order asymptotic expansion of blow-up solutions for autonomous
  {ODE}s. {I} - {M}ethod and {J}ustification.
\newblock {\em submitted}, 2022.

\bibitem{CFdlL2003-1}
X.~Cabr{\'e}, E.~Fontich, and R.~de~la Llave.
\newblock The parameterization method for invariant manifolds {I}: manifolds
  associated to non-resonant subspaces.
\newblock {\em Indiana University Mathematics Journal}, pages 283--328, 2003.

\bibitem{CFdlL2003-2}
X.~Cabr{\'e}, E.~Fontich, and R.~de~la Llave.
\newblock The parameterization method for invariant manifolds {II}: regularity
  with respect to parameters.
\newblock {\em Indiana University Mathematics Journal}, pages 329--360, 2003.

\bibitem{CFdlL2005}
X.~Cabr{\'e}, E.~Fontich, and R.~De~La~Llave.
\newblock The parameterization method for invariant manifolds {III}: overview
  and applications.
\newblock {\em Journal of Differential Equations}, 218(2):444--515, 2005.

\bibitem{C2015}
H.~Chiba.
\newblock {K}ovalevskaya exponents and the space of initial conditions of a
  quasi-homogeneous vector field.
\newblock {\em Journal of Differential Equations}, 259(12):7681--7716, 2015.

\bibitem{C2016_124}
H.~Chiba.
\newblock The first, second and fourth {P}ainlev{\'e} equations on weighted
  projective spaces.
\newblock {\em Journal of Differential Equations}, 260(2):1263--1313, 2016.

\bibitem{C2016_356}
H.~Chiba.
\newblock The third, fifth and sixth {P}ainlev{\'e} equations on weighted
  projective spaces.
\newblock {\em Symmetry, Integrability and Geometry: Methods and Applications},
  12:019, 2016.

\bibitem{DALP2015}
L.~D'Ambrosio, J.-P. Lessard, and A.~Pugliese.
\newblock Blow-up profile for solutions of a fourth order nonlinear equation.
\newblock {\em Nonlinear Analysis: Theory, Methods \& Applications},
  121:280--335, 2015.

\bibitem{D1993}
F.~Dumortier.
\newblock Techniques in the theory of local bifurcations: Blow-up, normal
  forms, nilpotent bifurcations, singular perturbations.
\newblock In {\em Bifurcations and Periodic Orbits of Vector Fields}, pages
  19--73. Springer, 1993.

\bibitem{DH1999}
F.~Dumortier and C.~Herssens.
\newblock Polynomial {L}i\'{e}nard equations near infinity.
\newblock {\em Journal of differential equations}, 153(1):1--29, 1999.

\bibitem{DLA2006}
F.~Dumortier, J.~Llibre, and J.C. Art{\'e}s.
\newblock {\em Qualitative theory of planar differential systems}.
\newblock Springer, 2006.

\bibitem{EG2006}
U.~Elias and H.~Gingold.
\newblock Critical points at infinity and blow up of solutions of autonomous
  polynomial differential systems via compactification.
\newblock {\em Journal of mathematical analysis and applications},
  318(1):305--322, 2006.

\bibitem{GP2011}
F.~Gazzola and R.~Pavani.
\newblock Blow up oscillating solutions to some nonlinear fourth order
  differential equations.
\newblock {\em Nonlinear Analysis: Theory, Methods \& Applications},
  74(17):6696--6711, 2011.

\bibitem{GP2013}
F.~Gazzola and R.~Pavani.
\newblock Wide oscillation finite time blow up for solutions to nonlinear
  fourth order differential equations.
\newblock {\em Archive for Rational Mechanics and Analysis}, 207(2):717--752,
  2013.

\bibitem{GK1985}
Y.~Giga and R.V. Kohn.
\newblock Asymptotically self-similar blow-up of semilinear heat equations.
\newblock {\em Comm. Pure Appl. Math.}, 38:297--319, 1985.

\bibitem{G2004}
H.~Gingold.
\newblock Approximation of unbounded functions via compactification.
\newblock {\em Journal of Approximation Theory}, 131(2):284--305, 2004.

\bibitem{HV1993}
M.A. Herrero and J.J.L. Vel{\'a}zquez.
\newblock Blow-up behaviour of one-dimensional semilinear parabolic equations.
\newblock In {\em Annales de {l}'Institut {H}enri {P}oincar{\'e} {C}, {A}nalyse
  non lin{\'e}aire}, volume~10, pages 131--189. Elsevier, 1993.

\bibitem{HBT1989}
G.W. Hunt, H.M. Bolt, and J.M.T. Thompson.
\newblock Structural localization phenomena and the dynamical phase-space
  analogy.
\newblock {\em Proceedings of the Royal Society of London. A. Mathematical and
  Physical Sciences}, 425(1869):245--267, 1989.

\bibitem{IY2003}
T.~Ishiwata and S.~Yazaki.
\newblock On the blow-up rate for fast blow-up solutions arising in an
  anisotropic crystalline motion.
\newblock {\em Journal of Computational and Applied Mathematics},
  159(1):55--64, 2003.

\bibitem{KSS2003}
B.L. Keyfitz, R.~Sanders, and M.~Sever.
\newblock Lack of hyperbolicity in the two-fluid model for two-phase
  incompressible flow.
\newblock {\em DISCRETE AND CONTINUOUS DYNAMICAL SYSTEMS SERIES B},
  3(4):541--564, 2003.

\bibitem{KK1990}
H.C. Kranzer and B.L. Keyfitz.
\newblock A strictly hyperbolic system of conservation laws admitting singular
  shocks.
\newblock In {\em Nonlinear evolution equations that change type}, pages
  107--125. Springer, 1990.

\bibitem{LM1990}
A.C. Lazer and P.J. McKenna.
\newblock Large-amplitude periodic oscillations in suspension bridges: some new
  connections with nonlinear analysis.
\newblock {\em {SIAM} Review}, 32(4):537--578, 1990.

\bibitem{LMT2021}
J.-P. Lessard, K.~Matsue, and A.~Takayasu.
\newblock A geometric characterization of unstable blow-up solutions with
  computer-assisted proof.
\newblock {\em arXiv:2103.12390}, 2021.

\bibitem{Mat2018}
K.~Matsue.
\newblock On blow-up solutions of differential equations with
  {P}oincar\'{e}-type compactifications.
\newblock {\em SIAM Journal on Applied Dynamical Systems}, 17(3):2249--2288,
  2018.

\bibitem{Mat2019}
K.~Matsue.
\newblock Geometric treatments and a common mechanism in finite-time
  singularities for autonomous {ODE}s.
\newblock {\em Journal of Differential Equations}, 267(12):7313--7368, 2019.

\bibitem{MT2020_1}
K.~Matsue and A.~Takayasu.
\newblock Numerical validation of blow-up solutions with quasi-homogeneous
  compactifications.
\newblock {\em Numerische Mathematik}, 145:605--654, 2020.

\bibitem{MT2020_2}
K.~Matsue and A.~Takayasu.
\newblock Rigorous numerics of blow-up solutions for {ODE}s with exponential
  nonlinearity.
\newblock {\em Journal of Computational and Applied Mathematics}, 374:112607,
  2020.

\bibitem{MZ1998}
F.~Merle and H.~Zaag.
\newblock Optimal estimates for blowup rate and behavior for nonlinear heat
  equations.
\newblock {\em Communications on {P}ure and {A}pplied {M}athematics},
  51(2):139--196, 1998.

\bibitem{Rob}
C.~Robinson.
\newblock {\em Dynamical systems - {S}tability, {S}ymbolic {D}ynamics, and
  {C}haos}.
\newblock Studies in Advanced Mathematics. CRC Press, Boca Raton, FL, second
  edition, 1999.

\bibitem{TMSTMO2017}
A.~Takayasu, K.~Matsue, T.~Sasaki, K.~Tanaka, M.~Mizuguchi, and S.~Oishi.
\newblock Numerical validation of blow-up solutions for ordinary differential
  equations.
\newblock {\em Journal of Computational and Applied Mathematics}, 314:10--29,
  2017.

\end{thebibliography}


\end{document}